\documentclass[onefignum,onetabnum]{siamart171218}

\usepackage{lipsum}
\usepackage{amsfonts}
\usepackage{graphicx}
\usepackage{epstopdf}
\usepackage{subfigure}
\usepackage{algorithmic}
\usepackage[figuresright]{rotating}
\usepackage{color}

\ifpdf
  \DeclareGraphicsExtensions{.eps,.pdf,.png,.jpg}
\else
  \DeclareGraphicsExtensions{.eps}
\fi

\newsiamremark{remark}{Remark}
\newsiamremark{hypothesis}{Hypothesis}
\crefname{hypothesis}{Hypothesis}{Hypotheses}
\newsiamthm{claim}{Claim}

\headers{Optimal control for multiscale eqns with rough coefficients}{Y.Chen, J.Zeng, X.Liu, L.Zhang}

\title{Optimal control for multiscale elliptic equations with rough coefficients\thanks{Submitted to the editors DATE.
\funding{J. Zeng and Y. Chen were partially supported by the National Natural Science Foundation of China (11671157, 91430104, 11510044). X. Liu and L. Zhang were partially supported by the National Natural Science Foundation of China  (11871339, 11471214, 11571314)}}}

\author{Y.P. Chen\thanks{ School of Mathematical Sciences, South China Normal University, Guangdong,  510631, China.
  (\email{yanpingchen@scnu.edu.cn}, \email{zengjiaoyan2013@163.com}).}
\and J.Y. Zeng\footnotemark[2]
  \and X.L. Liu\thanks{School of Mathematical Sciences, Institute of Natural Sciences, and Ministry of Education Key
Laboratory of Scientific and Engineering Computing (MOE-LSC), Shanghai Jiao Tong University, Shanghai, 200240, China. (\email{liuxinliang@sjtu.edu.cn},\email{lzhang2012@sjtu.edu.cn}).}
\and  L. Zhang\footnotemark[3]}

\usepackage{amsopn}

\ifpdf
\hypersetup{
  pdftitle={optimal control for multi-scale eqns with rough coefficients},
  pdfauthor={Y.Chen, J.Zeng, X.Liu, L.Zhang}
}
\fi

\newtheorem{thm}{Theorem}[section]
\newtheorem{lem}{Lemma}[section]

\def\bq{\begin{equation}}
\def\eq{\end{equation}}
\def\br{\begin{eqnarray}}
\def\er{\end{eqnarray}}
\def\brr{\bq\begin{array}{rlll}}
\def\err{\end{array}\eq}

\def\text#1{\hbox{#1}}
\newcommand{\ds}{\displaystyle}
\newcommand{\nn}{\nonumber}

\newcommand{\diam}{\mathrm{diam}}
\newcommand{\I}{\mathrm{I}}
\newcommand{\NH}{\mathrm{NH}}

\newcommand{\amin}{a_{\min}}
\newcommand{\amax}{a_{\max}}
\def\diiv{\operatorname{div}}

\DeclareMathOperator*{\spn}{span}

\begin{document}
\maketitle
\begin{abstract}
This paper concerns the convex optimal control problem governed by multiscale elliptic equations with arbitrarily rough $L^\infty$ coefficients, which has important applications in composite materials and geophysics. We use one of the recently developed numerical homogenization techniques, the so-called Rough Polyharmonic Splines (RPS) and its generalization (GRPS) for the efficient resolution of the elliptic operator on the coarse scale. Those methods have optimal convergence rate which do not rely on the regularity of the coefficients nor the concepts of scale-separation or ergodicity. As the iterative solution of the OCP-OPT formulation of the optimal control problem requires solving the corresponding (state and co-state) multiscale elliptic equations many times with different right hand sides, numerical homgogenization approach only requires one-time pre-computation on the fine scale and the following iterations can be done with computational cost proportional to coarse degrees of freedom. Numerical experiments are presented to validate the theoretical analysis.

\end{abstract}



\section{Introduction}
 We consider the following convex optimal control problem (OCP) governed by
elliptic partial differential equations with rough coefficients $a(x)\in L^\infty(\Omega)$
\begin{equation}
\begin{array}{cl}
& \min\limits_{u\in K \subset L^2(\Omega_U)}g(y)+h(u) \\
\text{subject to} & \left\{
\begin{array}{cc}
 -\diiv(a(x)\nabla y)=f+Bu, & \text{ in } \Omega,\\
y=0, & \text{ on } \partial \Omega.
\end{array}
\right.
\end{array}
\label{eqn:ocp}
\end{equation}
where $\Omega$ ($\Omega_U$)  is a bounded convex polygon in $\mathbb{R}^d$ ($d=2,3$) with
Lipschitz boundary $\partial \Omega$ ($\partial \Omega_U$). We adopt the standard notation $W^{m,p}$ for Sobolev spaces on $\Omega$ with norm $\|\cdot\|_{m,p}$ and seminorm $|\cdot|_{m,p}$ \cite{Adams2003Sobolev}. We define $W^{m,p}_{0}:=\{\psi\in W^{m,p}:|\psi|_{\partial \Omega}=0\}$.  For $p=2$, we denote $H^m(\Omega)=W^{m,2}(\Omega)$, and $\|\cdot\|_m=\|\cdot\|_{m,2}$.  Take $\|\cdot\|=\|\cdot\|_{0,2}$. In the rest of the paper, we take $V=H^1_0(\Omega)$ as the \emph{state space}, $U=L^2(\Omega_U)$ as the \emph{control space}, and $Y=L^2(\Omega)$ as the \emph{observation space}.

The admissible set $K\subset U$ for the control variable $u$ is a closed convex set, for example,
we can take $K=\{u\in L^2(\Omega_U)|u\geq 0\}$, or $K=\{u\in
L^2(\Omega_U)|\int_{\Omega_U}u\geq 0\}$, or $K=\{u\in L^2(\Omega_U)|a\leq
u\leq b\}$. $y\in V$ is the state variable. $g(y)+h(u)$ is the objective functional. $g(\cdot)$,
$h(\cdot)$ are two convex differentiable functions on $Y$ and $U$,
respectively.  $B: U \to L^2(\Omega)$ is a continuous linear operator such
that $\|Bw\|_Y\leq C_B \|w\|_U$ for any $w\in U$, for example,
$(Bu)(x)=c(x)u(x)$, where $c(x)$ is a density factor. Assume that the forcing term $f(x)\in L^2(\Omega)$.

Define the following inner products,
\begin{align*}
(u,v)_Y&:=\int _{\Omega}uv, \text{ for } u, v \in Y,\\
(u,v)_U&:=\int _{\Omega_U}uv,  \text{ for } u, v \in U.
\end{align*}
Denote $\|\cdot\|_Y$ and $\|\cdot\|_U$ as the corresponding norms induced by $(\cdot, \cdot)_Y$ and
$(\cdot, \cdot)_U$, respectively. We may drop the subscript $Y$ and $U$ if no confusion arises.

We impose the following assumptions on the coefficients and data.

(A1) The coefficients matrix $a(x)=(a_{ij}(x))$ is a symmetric $d\times d$ matrix which
satisfies the uniformly elliptic condition, i.e.,
\begin{eqnarray}
\amin |\xi|^2\leq\xi^T a(x)\xi\leq \amax |\xi|^2, \forall \xi\in R^d/\{0\}, x\in \bar{\Omega}.
\end{eqnarray}
where $\amin$ and $\amax$  are positive constants. $\kappa = \amax/\amin$ is the contrast.

(A2) Let $g'(\cdot)$ be the G\^{a}teaux derivative of the functional $g(\cdot)$ on $Y$, $g'\in Y'= Y$. Similarly, let $h'(\cdot)$ be the G\^{a}teaux derivative of the functional $h(\cdot)$ on $U$, $h'\in U'=U$. Assume that $g'$ and $h'$ are Lipschitz continuous, i.e.,
\begin{align}
\|g'(u)-g'(v)\|_{Y}\leq L_g\|u-v\|_{Y}, \forall u,v\in Y,\label{eqn:Lg}\\
\|h'(u)-h'(v)\|_{U}\leq L_h\|u-v\|_{U}, \forall u,v\in U,\label{eqn:Lh}
\end{align}
where $L_g$ and $L_h$ are positive constants.

(A3) The functionals $g(\cdot)$ and $h(\cdot)$ are uniformly convex,
\begin{align}
(g'(u)-g'(\tilde{u}),u-\tilde{u})_Y\geq M_g\|u-\tilde{u}\|^2_{Y},\forall u,\tilde{u}\in Y,\label{gu}\\
(h'(u)-h'(\tilde{u}),u-\tilde{u})_U\geq M_h\|u-\tilde{u}\|^2_{U},\forall u,\tilde{u}\in U,\label{hu}
\end{align}
where $M_g$ and $M_h$ are positive constants.

Optimal control plays an increasingly important role in many engineering branches, and efficient
numerical methods are essential to its successful application \cite{Hinze:2008,Troeltzsch:2010}. Over the past 30 years, finite element method (FEM) has become one of the most widely used numerical methods for optimal control problems. For the optimal control of elliptic or parabolic equations:
  a priori error estimates \cite{Hinze:2008,geveci1979approximation,yang2008},
   a posteriori error estimates  \cite{liu2003posteriori,benedix2009posteriori},
   and some superconvergence results \cite{chen2008superconvergence,chen2009superconvergence} for the FEM methods have been developed in the literature. However, those error estimates require the $H^2$ regularity of the solutions.

The optimal control problems governed by partial differential equations with rough coefficients
(such as permeabilities in reservoir modelling) have become a great challenge,
owing to the lack of regularity of the coefficients $a(x)$ and therefore the solutions
$y(x)$ and $p(x)$ ($p(x)$ is the co-state variable in \eqref{eqn:ocp-opt}).
Even if $a(x)$ is a smooth but highly oscillatory function in the form of
$\displaystyle a(x, \frac{x}{\varepsilon})$ with a small parameter $\varepsilon\ll 1$
(such as material properties of composite materials), conventional FEMs based on
piecewise polynomial basis will not be effective \cite{HouWu:1997}.
To be more precise, the convergence of conventional FEMs relies on the
$H^2$-regularity of the solution, but the prefactor of the error is of the order $O(1/\varepsilon)$.
Thus, conventional FEMs require prohibitively small mesh size $h< \varepsilon$ to yield good numerical approximations for the optimal control problem.

Numerical homogenization for problems with multiple scales have attracted
increasing attention in recent years. If the coefficient $a(x)$ has structural
properties such as scale separation and periodicity, together with some
regularity assumptions (e.g., $a(x)\in W^{1,\infty}$), classical
homogenization \cite{bensoussan1978asymptotic,jikov2012homogenization} can
be used to construct efficient  multiscale computational methods and have
been applied to optimal control problems, such as multiscale asymptotic
expansions method\cite{cao2004asymptotic,Cao2012A,liu2013multiscale}, multiscale finite
element method (MsFEM) \cite{HouWu:1997,chen2003mixed,chen2015mixed,chen2011numerical}, and heterogeneous
multiscale method (HMM) \cite{Weinan:2005, Ge2017Heterogeneous}.

For multi-scale PDEs with non-separable scales and high-contrast
coefficients which appear in many applications such as reservoir modelling and damage in
composite materials, the coefficients do not have structural properties such as periodicity/scale
separation. Numerical homogenization with non-separable scales concerns
approximation of the solution space of such problems by a (coarse) finite
dimensional space, instead of focusing on the classical issue of the
homogenization limit. Fundamental questions for numerical homogenization
are: How to approximate the high dimensional solution space by a low
dimensional approximation space with optimal error control, and furthermore,
how to construct the approximation space efficiently, for example, whether
its basis can be localized on a coarse patch of size $O(H)$. Several novel
approaches for numerical homogenization and their rigorous error
analysis have been developed recently, such as: rough polyharmonic splines
(RPS) method \cite{OwhadiZhangBerlyand:2014,Owhadi:2015,OwhadiMultigrid:2017,Liu:2018}; Localized
Orthogonal Decomposition (LOD) method \cite{Malqvist:2014b,engwer2016efficient}; and Generalized
multiscale finite element method (GMsFEM)
\cite{Yalchin2013Generalized,Gao2015Generalized}.

\def\div {\rm div}

In the context of optimal control, homogenization based methods have been
applied to problems governed by multiscale PDEs with separable scales
\cite{chen2015mixed,liu2013multiscale,chen2011numerical}. To the best of our
knowledge, few literature concerns the optimal control with non-separable
scales, which is of great importance for applications.

The purpose of this work is to obtain the convergence result for the solution of optimal control problems governed by multiscale elliptic equations with rough coefficients using the so-called generalized rough polyharmonic splines (GRPS) method \cite{Liu:2018}. Those optimal control problems often arise in the optimal design of composite materials and the control of water injection in reservoir simulation. The GRPS approximation space is generated by an interpolation basis minimizing an appropriate energy norm subject to certain constrains. The resulting approximation space leads to a quasi-optimal $H^1$-accurate approximation of the solution space together with quasi-optimal localization properties. The GRPS approximation can be cast as a Bayesian inference problem under partial information \cite{Owhadi:2015}.

The paper is organized as follows: We introduce the optimal control problem in Section \S~\ref{sec:optimalcontrol}. We formulate the GRPS method and show its numerical properties in Section \S~\ref{sec:grps}. In Section \S~\ref{sec:convergence}, we present the convergence analysis for the solution of optimal control problem using GRPS method. Numerical algorithms and results are given in Section \S~\ref{sec:numerics} to complement and justify the theoretical analysis. We conclude the paper in \S~\ref{sec:conclusion}.
\subsubsection*{Notation}

For a finite set $A$, we will use $\#A$ to denote the cardinality of $A$. For a measurable set $\tau$, we use $|\tau|$ to denote its measure.

The symbol $C$ denotes generic positive constant that may change from one line
of an estimate to the next. The dependencies of $C$ will normally be clear from the context
or stated explicitly.

\section{Formulation of the Optimal Control Problem}
\label{sec:optimalcontrol}
%

In this section, we present the equivalent formulations of the optimal control problem \eqref{eqn:ocp}, as well as the corresponding finite element formulations and the error estimates.

Recall that $V=H^1_0(\Omega)$ is the state space, and $U=L^2(\Omega_U)$ is the control space, we define the following energy product,
\begin{displaymath}
a(y,v):=\int _{\Omega}a(x)\nabla y\cdot\nabla v, \text{ for } y,v \in V.\\
\end{displaymath}

The weak formulation of the optimal control problem (OCP) \eqref{eqn:ocp} is: find $(y,u)\in V\times U$ such that
\begin{equation}
\begin{array}{cc}
&\min\limits_{u\in K \subset U }g(y)+h(u)\\
&\text{subject to} \quad a(y,v)=(f+Bu,v), \forall v\in V.
\end{array}
\label{eqn:wf-ocp}
\end{equation}

It is well know that (see, e.g, \cite[Theorem 1.46]{Hinze:2008}
and \cite{lions1971optimal}) the convex optimal control problem
\eqref{eqn:ocp} has a unique solution $(y,u)$, and that $(y,u)$ is
the solution of \eqref{eqn:ocp} if and only if there exists a
co-state $p\in V$ such that the triple $(y,u,p)$ satisfies the following optimality
conditions (OCP-OPT)
\begin{equation}
\begin{array}{cc}
a(y, v)=(f+Bu,v)_Y, &\forall v\in V,\\
a(p, q)=(g'(y),q)_Y, &\forall q\in V,\\
(h'(u)+B^*p,\tilde{u}-u)_U\geq 0,&\forall \tilde{u}\in K \subset U.
\end{array}
\label{eqn:ocp-opt}
\end{equation}
where $B^*$ is the adjoint operator of $B$.  In
\eqref{eqn:ocp-opt}, the first equation is satisfied by the state function $y$ in \eqref{eqn:ocp}, the second equation is satisfied by the
co-state function $p$, and the third one is a variational inequality satisfied
by the control function $u$.
\begin{remark}
For some $K$ the variational inequality in \eqref{eqn:ocp-opt} admits an
analytical solution. One such as example is the admissible set $K=\{u\in
L^2(\Omega_U)|\int_{\Omega_U}u\geq 0\}$, the variational inequality in
\eqref{eqn:ocp-opt} is equivalent to
$h'(u)=-B^*p+\max\{0,\overline{B^*p}\}$, where
$\bar{p}=\int_{\Omega_U}p/\int_{\Omega_U}1$ is the integral average of $p$
on $\Omega_U$. Another example is $K=\{u\in L^2(\Omega_U)|a\leq u\leq b\}$, and the variational inequality is equivalent to
$h'(u)=\max\{a,\min(b,-B^*p)\}$.

%

\end{remark}

Both $y$ and $p$ are solutions of the elliptic problem of the following form,
\begin{equation}
	a(z, w) = (\rho, w),\quad\text{for any } w\in V,
	\label{eqn:weakform}
\end{equation}
where $\rho\in Y = L^2(\Omega)$.


\def\T{\mathcal{T}}

The finite element approximation ${\rm (OCP)}_h$ for the optimal control problem \eqref{eqn:ocp} can be obtained by the restriction of $U$ and $V$ to their finite dimensional subspaces $U_h$ and $V_h$, respectively: find
$(y_h,u_h)\in V_h\times U_h$ such that
\begin{equation}
\begin{array}{cc}
&\min\limits_{u_h\in K_h \subset U_h}g(y_h)+h(u_h)\\
\text{subject to} &a(y_h,v_h)=(f+Bu_h,v_h), \forall v_h\in V_h.
\end{array}
\label{eqn:ocp_h}
\end{equation}

Again, the discretized control problem in \eqref{eqn:ocp_h} has a unique solution
$(y_h,u_h)$ if
and only if there is a co-state $p_h\in V_h$ such that the triplet
$(y_h,u_h,p_h)$ satisfies the following optimality conditions $\textrm{
(OCP-OPT)}_h$ \cite{liu2008adaptive}:
\begin{equation}
\begin{array}{cc}
a(y_h,v_h)=(f+Bu_h,v_h), &\forall v_h\in V_h,\\
a( p_h, q_h)=(g'(y_h),q_h), &\forall q_h\in V_h,\\
(h'(u_h)+B^*p_h,\tilde{u}_h-u_h)_U\geq 0, &\forall \tilde{u}_h\in K_h \subset U_h.
\end{array}
\label{eqn:ocp-opt-h}
\end{equation}

Let $\T^h$ ($\T^h_U$) be a quasi-uniform triangulation of $\Omega$ ($\Omega_U$).  Note that the regularities of the control variable $u$ is lower than the regularity of the state variable $y$ and the co-state variable $p$. We can choose the piecewise constant finite element space over $\T^h_U$ as $U_h$, and $K_h:= K\cap U_h$.
If weak solutions of elliptic problem \eqref{eqn:weakform} admit $H^2$ regularity, namely,
$\|z\|_{H^2}\leq C \|\rho\|_{Y}$, $V_h$ can be taken as the conforming piecewise linear finite element space over $\T^h$.
We have the following theorem for the error estimates of $(y_h, u_h, p_h)$ \cite[Theorem
4.1.1]{liu2008adaptive}.
\begin{thm}
Let $(y,u,p)$ be the solution of (\ref{eqn:ocp-opt}), and
$(y_h,u_h,p_h)$ be the finite element solution of (\ref{eqn:ocp-opt-h}).
Assume that $u\in H^1(\Omega_U)$, $y,p\in H^2(\Omega)$, it holds true that
\begin{equation}
\|y-y_h\|_{1,\Omega}+\|p-p_h\|_{1,\Omega}+\|u-u_h\|_{0,\Omega_U}\leq C(h_U+h) (\|y\|_2+\|p\|_2+\|u\|_1).\label{eqn:estimate}
\end{equation}
\label{thm:estimatep1}
\end{thm}

\def\ve{\varepsilon}
For optimal control governed by multiscale equations, we assume $a(x) = a(x, x/\ve)\in W^{1, \infty}$ in the scale separation case. The Dirichlet problem
\begin{equation}
\left\{
\begin{split}
-\nabla \cdot(a(x,x/\ve)\nabla z^\ve)=\rho, & \quad x\in \Omega,\\
 z^\ve=0, & \quad x\in \partial \Omega.
\end{split}
\right.
\label{eqn:scalesep}
\end{equation}
has the following regularity estimate $\ds \|z^\ve\|_2\leq \frac{C}{\ve}\|\rho\|$ for $\rho\in L^2(\Omega)$.
Therefore, the error bound in Theorem \ref{thm:estimatep1} becomes $\ds O(\frac{h_U+h}{\ve})$, and one may need to take extremely small
$h\ll \ve$ and $h_U \ll \ve$ in order to obtain accurate solution with conventional piecewise polynomial FEM spaces, and the computational costs will become prohibitive.

Multiscale numerical methods such as multiscale finite element method (MsFEM) \cite{HouWu:1997, HouWu:1999} can be used
to efficiently capture the large scale components of the solution on the coarse grid. In \cite[Theorem 4.11]{chen2011numerical} and \cite{li:2010},
MsFEM was applied to solve the control problem governed by multiscale PDEs with separable scales,  a priori error estimates
can be obtained as follows,
\begin{displaymath}
\|y-y_h\|_{1,\Omega}+\|p-p_h\|_{1,\Omega}+\|u-u_h\|_{0,\Omega_U}\leq
C(h+\sqrt{\ve/h}).
\end{displaymath}

However, numerical homogenization methods based on concepts such as scale separation and periodicity/ergodicity cannot
be applied directly to problems with nonseparable scales. In the next section, we will introduce the concept of numerical
homogenization that does not rely on the classical assumptions of homogenization theory such as scale separation and
ergodicity, but only on the compactness of the solution space \cite{Owhadi:2011a}.

\section{Numerical Homogenization and Generalized Rough Polyharmonic Splines}
\label{sec:grps}

\def\dx{\textrm{dx}}
\def\dz{\textrm{dz}}
\def\dy{\textrm{dy}}
\def\dof{\textrm{dof}}

It is clear that a good approximation space for the elliptic equation \eqref{eqn:elliptic} is important for the accurate solution of the optimal control problems \eqref{eqn:ocp} with arbitrarily rough coefficients $a(x)\in L^\infty$,
\begin{equation}
\left\{
\begin{split}
-\nabla \cdot(a(x)\nabla z)=\rho(x), & \quad x\in \Omega,\\
 z=0, & \quad x\in \partial \Omega.
\end{split}
\right.
\label{eqn:elliptic}
\end{equation}

This is closely linked to the problem of \emph{numerical homogenization} with non-separable scales. Up to now, numerical homogenization has become a large field. It is motivated by the fact that standard methods, such as finite-element
method with piecewise linear elements \cite{BaOs:2000} can perform arbitrarily badly for PDEs with rough coefficients.
Although some numerical homogenization methods such as multiscale finite element methods \cite{HouWu:1997, Wu:2002, HouWu:1999, EfendievHou:2009}, heterogeneous multiscale methods \cite{EEngquist:2003,Engquist:2011} are directly inspired from classical homogenization concepts such as periodic homogenization and scale separation \cite{bensoussan1978asymptotic}, one of the main objectives of numerical homogenization is to achieve a numerical approximation of the solution space of \eqref{eqn:elliptic} with arbitrary rough coefficients (i.e., in particular, without the assumptions found in classical homogenization, such as scale separation, ergodicity at fine scales and $\ve$-sequences of operators). Methods such as harmonic coordinates \cite{OwZh:2007a}, generalized multiscale finite
element \cite{Yalchin2013Generalized,Chung:2014}, flux-norm based approaches \cite{Berlyand:2010} have been proposed for numerical homogenization with arbitrary rough coefficients.

We first formulate the problem of numerical homogenization. Suppose $\varepsilon$ is the smallest scale of the elliptic problem, let $H$ be an artificial scale determined by available computational power and/or desired precision, $N$ is the corresponding $\dof$ such that $H\sim N^{-1/d}$, and the scales $\varepsilon \ll H \ll 1$. The goal of numerical homogenization is to construct a finite dimensional space $V_H$ ($V_N$), and to find an approximate solution $z_H\in V_H$, such that,
\begin{itemize}
	\item $z_H$ has guaranteed error estimate in certain norm $\|\cdot\|$, e.g., $\|z-z_H\|\leq CH^\alpha$ or equivalently $\|z-z_N\|\leq CN^{-\alpha/d}$, with $C$ independent of $\varepsilon$ (and contrast) and $\alpha$ is the optimal convergence rate.
	\item $V_H$ is constructed via precomputed subproblems which are optimally localized and can be solved in parallel, also those subproblems do not depend on the forcing term and boundary condition (analog of cell problems in classical homogenization).
\end{itemize}

The basis of numerical homogenization needs to be pre-computed. The computation of each basis will be independent and the support of each basis needs to be localized on a small patch.
The possibility to compute such bases on localized sub-domains of the global domain
 without loss of accuracy is therefore a problem of practical importance. We refer to \cite{ChuHou09,EfGaWu:2011,BaLip10, Owhadi:2011a,Malqvist:2014b} for recent localization results for divergence-form elliptic PDEs.

We will introduce the \emph{generalized rough polyharmonic splines} (GRPS) in \cite{OwhadiZhangBerlyand:2014, Liu:2018}. Given $N$ measurement functions $\psi_i$, $i=1, \dots, N$,
define
\begin{displaymath}
	V_i:=\{\phi\in V| \int_\Omega \phi(x)\psi_j(x)\dx = \delta_{i,j}, j=1,\dots ,N \}.
\end{displaymath}

\def\<{\langle}
\def\>{\rangle}
The GRPS basis is given by the solution of the following constrained minimization problem which is strictly convex and has a unique minimizer $\phi_i\in V_i$,
		\begin{equation}
			\begin{cases}
				\text{Minimize }\|\phi\|^2\\
				\text{Subject to }\phi\in V_i
			\end{cases}
			\label{eqn:phi}
		\end{equation}	
for an appropriate norm $\|\cdot\|$.
	
\begin{remark}
The GRPS basis can be given by the Bayesian Formulation in \cite{Owhadi:2015}. We can ask the following question: Given $N$ \emph{observables} $\Psi_i = \int_\Omega z(x)\psi_i(x) \dx$ of the solution $z(x)$ of $-\diiv (a \nabla z) = \rho$, $i=1,\dots, N$, what is the best guess of $z(x)$? The answer can be given by the following procedure of randomization and conditioning.
	\begin{enumerate}
		\item Randomization: Put a prior on $z(x)$, e.g. $z(x)$ is given by
			\begin{displaymath}
    				\left\{
    				\begin{array}{lr}
					-\text{div} (a\nabla z)=\rho(x) & x\in\Omega \\
    					z=0 & x\in\partial\Omega \\
    				\end{array}
    				\right.
	   	 	\end{displaymath}
			where $\xi(x)$ is a Gaussian field with covariance function $\Lambda(x,y)$, therefore $z(x)$ is a Gaussian field with covariance function
\begin{displaymath}
	\Gamma(x,y) = \int_{\Omega^2} G(x,z) \Lambda(z,z') G(y,z') \dz\dz'.
\end{displaymath}

		\item Conditioning: take the conditional exception $\mathbb{E}[z(x)|\Psi]$, we have
		\begin{equation}
			\mathbb{E}[z(x)|\Psi] = \sum_{i=1}^N \Psi_i \phi_i(x), \quad\text{ where } \phi_i(x):=\sum_{j=1}^N \Theta_{ij}^{-1}\int_\Omega \Gamma(x,y)\psi_j(y)\dy
			\label{eqn:condexp}
		\end{equation}
		and $\Theta_{i,j}:=\int_{\Omega^2} \psi_i(x)\Gamma(x,y)\psi_j(y) \dx\dy$ is the covariance matrix of $\Psi$.
The $\phi_i$ given by the conditional expectation in \eqref{eqn:condexp} is exactly the one given in the variational formulation in \eqref{eqn:phi} (by choosing an appropriate $\|\cdot\|$).
\end{enumerate}
this Bayesian framework can be further generalized to the game/decision theoretical framework as in \cite{OwhadiMultigrid:2017}. 		
\end{remark}	

%
%
%
%
%
%
%

\def\spn{\rm span}
\def\IH{\mathfrak{I}_H}
\def\dx{\textrm{dx}}
\def\L{\mathcal{L}}
\def\divagrad{\diiv a \nabla}
\def\loc{\textrm{loc}}

In fact, we have some flexibility to choose $\|\cdot\|$ in \eqref{eqn:phi} (corresponding to $\Lambda(x,y)$ in the Bayesian framework). For example, it can be taken as $\|\diiv a(x) \nabla u\|_{L^2(\Omega)}$ which is used in the \emph{rough polyharmonic splines} (RPS) paper \cite{OwhadiZhangBerlyand:2014}, or the energy norm $\|u\|_a$ as in \cite{OwhadiMultigrid:2017,Liu:2018}. 	


We also have different choices for the measurement function $\psi_i$, for example,
		\begin{itemize}
			\item $\psi_i = \delta(x-x_i)$ (point value observables), together with the norm $\|\diiv a(x) \nabla z\|$, we recover the \emph{RPS basis} in \cite{OwhadiZhangBerlyand:2014};
			\item  $\psi_i$ as characteristic functions of patches in a coarse triangulation $\T_H$, observables are patch averages of $z$. We refer to this basis as the \emph{GRPS basis} in the current paper;
			\item  $\psi_i$ as characteristic functions of edges in a coarse triangulation $\T_H$, observables are edge averages of $z$ ;
			\item For a quasi-interpolation operator $\IH$ (Clement, Oswald, etc.), there exists $\{\psi_i\}$ such that $\IH(z)(x_i) = \int_\Omega \psi_i z\dx$, which recovers the LOD approach by Peterseim et.al \cite{Malqvist:2014b}.
		\end{itemize}

%

The GRPS approach naturally induces a two-level decomposition: Define the coarse space as
$V_H: = {\rm span}\{\phi_i\}$, the fine space can be naturally defined as (given the full space $V = H^1_0(\Omega)$ or $V_h$),
\begin{center}
$V_f: = \{v\in X|\int \psi_i v \dx = 0, \forall i\}.$
\end{center}

	\def\zin{z_{\rm I}}
	The $\<\cdot, \cdot\>$ inner product is induced by the norm $\|\cdot\|$. Then $V_H \perp V_f$ with respect to the $\<\cdot, \cdot\>$ product. Furthermore, we have the following optimal recovery property of $z(x)$ in $V_H$. Let $\zin: = \sum_i (\int_\Omega \psi_i z \dx) \phi_i(x)$ be the interpolant of u in $V_H$, we have
		\begin{equation}
			\|z\|^2 = \|\zin\|^2 + \|z-\zin\|^2.
			\label{eqn:interpolation}
		\end{equation}

The following properties of GRPS is important for the proof of the convergence for the optimal control problem, the proof of those results can be found in \cite{Liu:2018,OwhadiZhangBerlyand:2014,OwhadiMultigrid:2017}.
	
\def\IH{\mathfrak{I}_H}

\begin{thm}\label{rpsin}[Optimal Approximation Property of $V_H$]:
		\begin{displaymath}
		\|\nabla z-\nabla z_{\I}\|_{L^2(\Omega)}\leq C_{\I} H \|\rho\|_{L^2(\Omega)}.
		\end{displaymath}
		This is true for the following constructions
		\begin{itemize}
			\item RPS basis, using higher order Poincare inequality $\|\nabla v\|_{L^2(\Omega)}\leq CH\|{\rm div}a\nabla v\|_{L^2(\Omega)}$ for $v\in V_f$.
			\item GRPS basis, using Poincare inequality for $v\in V_f$.
			\item LOD construction, approximation property for quasi-interpolation operator $\IH$:
			$H^{-1}\|v-\IH v\|_{L^2(T)}+\|\nabla(v-\IH v)\|_{L^2(T)}\leq C\|\nabla v\|_{L^2(\omega_T)}.$
		\end{itemize}
\end{thm}

Note that in the variational formulation \eqref{eqn:phi}, the minimization is done for functions defined on whole $\Omega$, and we call the corresponding basis global basis.

\begin{thm}
For the finite element solution $z_H\in V_H$, where $V_H$ is the space of global basis. We have
    	\begin{equation}
    		\|z-z_H\|_{H^1_0(\Omega)}\leq C^g_{\NH}H\|\rho\|_{L^2(\Omega)}.
   		\end{equation}
\end{thm}

Of course, it is not preferable to use global basis in practical computation. The good thing is, the global basis has the following exponential decay property which can be proved by using Cacciopoli like argument for harmonic functions \cite{Malqvist:2014b,OwhadiMultigrid:2017}.

\begin{thm}
We have the following exponential decay property of global basis.
			\begin{displaymath}	
					\int_{\Omega\cap(B(x_i, r))^c}\nabla\phi_i^t a\nabla\phi_i\leq \exp(1-\frac{r}{CH})\int_\Omega \nabla\phi_i^t a\nabla\phi_i.
			\end{displaymath}
\end{thm}

The exponential decay property opens an avenue for the local approximation of global basis. Let $\Omega_i\subset \Omega$, introduce
\begin{displaymath}
	V^{\loc}_i:=\{\phi\in V| \int_{\Omega_i} \phi(x)\psi_j(x)\dx = \delta_{i,j}, j=1,\dots ,N \}.
\end{displaymath}

The local basis is given by the solution of the following constrained minimization problem which is strictly convex and has a unique minimizer $\phi^{\loc}_i\in V^{\loc}_i$,
		\begin{equation}
			\begin{cases}
				\text{Minimize }\|\phi\|_{\Omega_i}^2\\
				\text{Subject to }\phi\in V^{\loc}_i
			\end{cases}
			\label{eqn:philoc}
		\end{equation}	

We have the following properties of the localized basis.
\begin{thm}\label{thm:philoci} [Truncation error for the localized basis]: If ${\rm diam} (\Omega_i:={\rm supp}(\phi^{\rm loc}_i)) = r$, then
			\begin{equation}	
				\|\phi_i - \phi^{\rm loc}_i\|_{a}\leq C \exp(-\frac{r}{2lH}).
			\end{equation}
	
\end{thm}

and the convergence of the FEM with localized basis
 \begin{thm} \label{thm:localbasis} [Accuracy of FEM with Localized Basis]:
  If $\ds (B(x_i, CH\ln\frac{1}{H})\cap \Omega)\subset\Omega_i$, $z_H^{\loc}$ is the FEM solution in ${\rm span} \{\phi_i^{\loc}\}$, the space of localized basis, then
  \begin{equation}
  	\|z-z_H^{\loc}\|_{H^1_0(\Omega)}\leq C_{\NH}H\|\rho\|_{L^2(\Omega)}.
  \end{equation}
 \end{thm}

\section{Convergence Analysis for Optimal Control Problem}
\label{sec:convergence}
In the rest of the paper, we use the localized numerical homogenization basis $V_{H}:=\spn\{\phi_i^{loc}\}_{i\in \mathcal{N}}$ which satisfies the Theorem \ref{thm:localbasis}.

Consider the following optimal control problem
\begin{equation}
\begin{array}{cc}
a(y_H,v_H)=(f+Bu_H,v_H), &\forall v_H\in V_H,\\
a( p_H, q_H)=(g'(y_H),q_H), &\forall q_H\in V_H,\\
(h'(u_H)+B^*p_H,\tilde{u}_H-u_H)_U\geq 0, &\forall \tilde{u}_H\in K_H \subset U_H.
\end{array}
\label{eqn:ocp-opt-H}
\end{equation}
where $U_H$ is the piecewise constant finite element space over $\T^H_U$, and $K_H:= K\cap U_H$.

\subsection{A priori error estimates}

Fixed the control approximation $u_H\in V_H$, define the auxiliary
solutions
 $(y(u_H),p(u_H))\in V\times V$
which are the solutions of the following equations:
\begin{equation}
\begin{array}{cc}
a( y(u_H),v)=(f+Bu_H,v), &\forall v\in V,\\
a(p(u_H), q)=(g'(y(u_H)),q), & \forall q\in V.
\end{array}
\label{eqn:auxiliary}
\end{equation}

We have the following lemma for the accuracy of auxiliary solutions.

\begin{lem}Let $y(u_H)$ and $p(u_H)$ be the solutions of \eqref{eqn:auxiliary}, $y$ and $p$ be the finite element solutions of \eqref{eqn:ocp-opt} in $V$. It holds true that
\begin{displaymath}
\|y-y(u_H)\|_{1,\Omega}+\|p-p(u_H)\|_{1,\Omega} \leq  C\|u-u_H\|,
\end{displaymath}
where $C$ depends on $\amin$, $\amax$, $d$, $C_B$, $L_g$.
\label{lem:y-yuH}
\end{lem}

\begin{proof}
By \eqref{eqn:ocp-opt} and \eqref{eqn:auxiliary}, we have
\begin{displaymath}
\begin{array}{cc}
a( y- y(u_H),v)=(Bu-Bu_H,v), &\forall v\in V,\\
a(p - p(u_H), q)=(g'(y)-g'(y(u_H)),q), & \forall q\in V.
\end{array}
\end{displaymath}

which implies that
\begin{displaymath}
\begin{array}{cll}
\|y- y(u_H)\|_1& \leq C\|u-u_H\|, & \\
\|p - p(u_H)\|_1& \leq C\|y-y(u_H)\| & \leq C \|u - u_H\|.
\end{array}
\end{displaymath}

\end{proof}

The following lemma bounds the accuracy of the approximate solution of
\eqref{eqn:ocp-opt-H}.

\begin{lem}
Let $y(u_H)$ and $p(u_H)$ be the solutions of \eqref{eqn:auxiliary}, $y_H$ and $p_H$ be the finite element solutions of \eqref{eqn:ocp-opt-H} in $V_H$. It holds true that
\begin{equation}
\|y(u_H)-y_H\|_{1}\leq CH\|f+Bu_H\|,\label{eqn:error_yuH}
\end{equation}
where $\ds C = [\frac{\amax} {\amin} (1+d/\pi)]^{1/2} C_{\I}$. And,
\begin{equation}
\|p(u_H)-p_H\|_1\leq CH (\|f+Bu_H\| + \|g'(y_H)\|).\label{eqn:error_puH}
\end{equation}
where $\ds C = \max\{(\frac{\amax} {\amin} (1+d/\pi)\sqrt{2}+\frac{1}{\sqrt{2}})C_{NH},\frac{L_g}{\amin}(1+d/\pi)\sqrt{2}\}$, $d=\diam (\Omega)$.
\label{lem:yuH-yH}
\end{lem}

\begin{proof}
The first inequality \eqref{eqn:error_yuH} is due to usual finite element estimate, Poincar\'{e}  inequality, and \eqref{eqn:interpolation},
\begin{align*}
	\frac{\amin}{1+d/\pi}\|y(u_H)-y_H\|_1^2 & \leq a(y(u_H) - y_H, y(u_H) - y_H) \\
	& \leq a(y(u_H) - y(u_H)_{\I}, y(u_H) - y(u_H)_{\I}) \\
	 							 & \leq \amax C_{\I}^2 H^2 \|f+Bu_H\|^2.
\end{align*}
where $y(u_H)_{\I}$ is the interpolation of $y(u_H)$ in $V_H$.

By Poincar\'{e}  inequality, Lipschitz property \eqref{eqn:Lg} of $g'(\cdot)$, \eqref{eqn:ocp-opt-H}, \eqref{eqn:auxiliary}, we obtain
\begin{align*}
&\frac{\amin}{1+d/\pi}\|p(u_H)-p_H\|_1^2 \\
 \leq & a(p(u_H)-p_H,p(u_H)-p_H)\\
=&a(p(u_H)-p(u_H)_{\I},p(u_H)-p_H)+a(p(u_H)_{\I}- p_H,p(u_H)-p_H)\\
=&a(p(u_H)-p(u_H)_{\I},p(u_H)-p_H)+(g'(y(u_H))-g'(y_H), p(u_H)_{\I}- p_H)\\
\leq& (\frac{\amax^2}{2\ve}+\frac{\ve}{2}) \|p(u_H)-p(u_H)_{\I}\|_1^2+\frac{L_g^2}{2\ve}\|y(u_H)-y_H\|^2+\ve\|p(u_H)-p_H\|_1^2,
\end{align*}
where $(p(u_H))_{\I}$ is the interpolation of $p(u_H)$ in $V_H$.

Take $\ds\ve = \frac12\frac{\amin}{1+d/\pi}$, we have
\begin{align*}
\|p(u_H)-p_H\|_1^2& \leq C (\|p(u_H)-p(u_H)_{\I}\|_1^2+ \|y(u_H)-y_H\|^2)\\
					  & \leq C H^2\|g'(y_H)\|^2 + CH^2\|f+Bu_H\|^2
\end{align*}
where the constants $C$ depends on $\amin$, $\amax$, $d$, $L_g$, $C_I$ and $C_{\NH}$, but not on $H$.

\end{proof}

Define the averaging projection(\cite{liu2008adaptive}) $\Pi_{H}:  K\rightarrow K_{H}$ as,
\begin{eqnarray}
\Pi_{H} v|_{\tau}=\frac{1}{|\tau|}\int_{\tau_U}v, \forall \tau\in\mathcal{T}_{h_U}.
\end{eqnarray}
Note that $K_{H}\subset K$. We have that if $v\in H^1(\Omega_U)$ (see e.g.\cite{liu2008adaptive,ciarlet2002finite}),
\begin{eqnarray}
\|v-\Pi_{H} v\|_{0,\Omega_U}\leq C H_U|v|_{1,\Omega_U}\label{pi}.
\end{eqnarray}
Moreover, if $u\in H^1(\Omega_U)$, $p\in H^1(\Omega)$, assume that $h'(u)$
is Lipschitz continuous,  we have
\begin{eqnarray}
(h'(u)+B^*p,\Pi_{H}u-u)_U&=&\sum_{\tau_U\in \mathcal{T}^H_{U}}\int_{\tau_U}(h'(u)+B^*p-\Pi_{H}(h'(u)+B^*p))(\Pi_{H}u-u)\nn\\
&\leq&\|h'(u)+B^*p-\Pi_{H}(h'(u)+B^*p)\|_{0,\Omega_U}\|\Pi_{H}u-u\|_{0,\Omega_U}\nn\\
&\leq&CH_U^2|h'(u)+B^*p|_{1,\Omega_U}|u|_{1,\Omega_U}\nn\\
&\leq& CH_U^2(|u|_{1,\Omega_U}^2+|p|_{1,\Omega}^2).\label{pi2}
\end{eqnarray}

\begin{lem}\label{lem:u-uH}
Let $(y,p,u)$ be the solution of equation (\ref{eqn:ocp-opt}), and
$(y_H,p_H, u_H)$ be the finite element solution of
(\ref{eqn:ocp-opt-H}). Assume that $u\in H^1(\Omega_U)$,
it holds true that
\begin{displaymath}
\|u-u_H\|\leq C(H_U(|u|_{1,\Omega_U}+|p|_{1})+H(\|f+Bu\|+\|g'(y)\|)),
\end{displaymath}
 where $C$ is a constant depending on $\amin$, $\amax$, $L_h$, $L_g$, $C_B$, $C_I$ and $C_{\NH}$,  but not on $H$ and $H_U$.
 \end{lem}

\begin{proof}
It follows from (\ref{eqn:ocp-opt}) that
\begin{eqnarray*}
(B^*p(u), u-v)_U&=&(p(u),B(u-v))=a(y(u)-y(v), p(u))\\
&=&(g'(y(u)), y(u)-y(v)).
\end{eqnarray*}
Similarly,
\begin{displaymath}
(B^*p(v), u-v)_U=(g'(y(v)), y(u)-y(v)).
\end{displaymath}
The convexities of $g$ and $h$ imply that,
\begin{eqnarray}
M_h\|u_H-u\|^2_{0,\Omega_U}&\leq& (h'(u_H)-h'(u),u_H-u)_U\nn\\
&\leq& (h'(u_H)-h'(u),u_H-u)_U+(g'(y(u_H))-g'(y(u)), y(u_H)-y(u))\nn\\
&=& (h'(u_H)+B^*p(u_H),u_H-u)_U-(h'(u)+B^*p(u),u_H-u)\label{eqn:convexity}.
\end{eqnarray}
Combining \eqref{eqn:convexity}, the optimality conditions \eqref{eqn:ocp-opt} and
\eqref{eqn:ocp-opt-H}, the estimate (\ref{pi2}), we have
\begin{eqnarray}
M_h\|u_H-u\|^2_{0,\Omega_U}&\leq&(h'(u_H)+B^*p_H,u_H-\Pi_{H}u)_U+(B^*p(u_H)-B^*p_H,u_H-\Pi_{H}u)_U\nn\\
&&+(h'(u_H)+B^*p(u_H),\Pi_{H}u-u)_U\nn\\
&\leq&(B^*p(u_H)-B^*p_H,u_H-u)_U+(B^*p(u_H)-B^*p_H,u-\Pi_{H}u)_U\nn\\
&&+(h'(u)+B^*p,\Pi_{H}u-u)_U+(h'(u_H)-h'(u),\Pi_{H}u-u)_U\nn\\
&&+(B^*p(u_H)-B^*p,\Pi_{H}u-u)_U\nn\\
&\leq& \frac12 C_B^2(1/\ve+1)\|p(u_H)-p_H\|^2+C H_U^2(|p|_{1,\Omega_U}^2+|u|_{1,\Omega_U}^2)+\nn\\
&&(\frac12+\frac{L_h^2+C_B^2}{2\ve})\|\Pi_{H}u-u\|^2+\ve\|u_H-u\|^2+\frac{\ve}{2}\|p(u_H)-p\|^2\label{eqn:uH}.
\end{eqnarray}
where $\ve$ is a arbitrarily positive constant.

By Lemma \ref{lem:y-yuH}, we have
\begin{equation}
\|p(u_H)-p\|\leq \|p(u_H)-p\|_1\leq C\|u_H-u\|.\label{fz1}
\end{equation}
Thus, if we choose $\ds\ve=\frac{M_h}{2+C^2}$, it follows from \eqref{pi}, \eqref{fz1} and \eqref{eqn:uH}  that,
\begin{eqnarray}
\|u_H-u\|^2_{0,\Omega_U}&\leq&  C\|p(u_H)-p_H\|^2+C H_U^2(|p|_{1,\Omega_U}^2+|u|_{1,\Omega_U}^2).\label{eqn:uh1}
\end{eqnarray}
Using  \eqref{eqn:error_yuH}, \eqref{eqn:error_puH},  Lemma \ref{lem:y-yuH}, and choosing $H$ sufficient small such that $CL_g^2H^2\leq 1$, we have
\begin{eqnarray}
\|p(u_H)-p_H\|_1^2&\leq& CH^2 (\|f+Bu_H\|^2 + \|g'(y_H)\|^2)\nn\\
&\leq& CH^2 (\|f+Bu\|^2+\|g'(y)\|^2)+ C(C_B^2+L_g^2)H^2\|u-u_H\|^2.
\label{eqn:puh}
\end{eqnarray}
Thus,  it follows from \eqref{eqn:uh1} and \eqref{eqn:puh} that,
\begin{eqnarray*}
\|u-u_H\|^2&\leq&C(H_U^2(|u|_{1,\Omega_U}^2+|p|_{1}^2)+H^2(\|f+Bu\|^2+\|g'(y)\|^2))\nn\\
&&+C(C_B^2+L_g^2)H^2\|u-u_H\|^2.
\end{eqnarray*}

By choosing $H$ sufficient small such that $\ds 1-C(C_B^2+L_g^2)H^2\geq \frac12$, we conclude the proof of the theorem.
\end{proof}

Combining Lemmas \ref{lem:y-yuH}, \ref{lem:yuH-yH}, \ref{lem:u-uH}, we have the following a priori error estimates.

\begin{thm}
Let $(y,p,u)$ be the solution of equation \eqref{eqn:ocp-opt}, and
$(y_H,p_H, u_H)$ be the finite element solution of \eqref{eqn:ocp-opt-H}. Assume that $u\in H^1(\Omega_U)$, then it holds true that
\begin{displaymath}
\|y-y_H\|_{1}+\|p-p_H\|_{1}+\|u-u_H\| \leq CH_U(|u|_1+|p|_1)+H(\|f\|+\|u\|+\|g'(y)\|).
\end{displaymath}
\label{thm:main}
\end{thm}
\begin{proof}
Note that
\begin{eqnarray*}
\|y-y_H\|_{1,\Omega}&\leq& \|y-y(u_H)\|_{1,\Omega}+\|y(u_H)-y_H\|_{1,\Omega},\\
\|p-p_H\|_{1,\Omega}&\leq& \|p-p(u_H)\|_{1,\Omega}+\|p(u_H)-p_H\|_{1,\Omega}.
\end{eqnarray*}
Lemma \ref{lem:y-yuH} leads to
\begin{displaymath}
\|y-y(u_H)\|_{1}+\|p-p(u_H)\|_{1} \leq C\|u-u_H\|.
\end{displaymath}
Lemma \ref{lem:yuH-yH} leads to
\begin{displaymath}
\|y(u_H)-y_H\|_{1}+\|p(u_H)-p_H\|_{1} \leq C H(\|f+Bu\|+C_B\|u-u_H\|+\|g'(y)\|+L_g\|y-y_H\|).
\end{displaymath}
Lemma \ref{lem:u-uH} leads to
\begin{displaymath}
\|u-u_H\|\leq C(H_U(|u|_{1,\Omega_U}+|p|_{1})+H(\|f+Bu\|+\|g'(y)\|)).
\end{displaymath}
Combining all those estimates, we conclude the proof by taking $\ds CL_g H \leq \frac12$.
\end{proof}
\section{Algorithm and Numerical Experiments}
\label{sec:numerics}

\subsection{Numerical algorithm}
To solve the optimal control problem \eqref{eqn:ocp-opt-H}, we will introduce the following projection algorithm.

\def\Pbk{P^b_K}

Define the projection operator $\Pbk:U\rightarrow K$: for $w\in U$, find $\Pbk w\in K$ such that
\begin{eqnarray}
(\Pbk w-w, \Pbk w-w)=\min_{u\in K}(u-w,u-w),\label{7.11}
\end{eqnarray}
which is equivalent to the inequality
\begin{eqnarray}
(\Pbk w-w, v-\Pbk w)\geq 0,\forall v\in K.\label{7.12}
\end{eqnarray}
It is clear that $\Pbk$ is well-defined for any closed convex subset $K\subset U$. For example, when $ K=\{v\in L^2(\Omega_U)|\int_{\Omega_U}v\geq 0\}$, for any $w\in U$, we have
\begin{equation}
\Pbk w=\max\{0,-\bar{w}\}-(-w)=-\min\{0,\bar{w}\}+w.
\label{eqn:Pbk}
\end{equation}
where $\bar{w}=\int_{\Omega_U}w/\int_{\Omega_U}1$ denotes the average of $w$ on $\Omega_U$.
The formulas for other important cases can be found in
\cite{liu2008adaptive}.

We have the following lemma.

\begin{lem}
For the solution $u$ of
\begin{displaymath}
	\min\limits_{u\in K \subset L^2(\Omega_U)}J(u)
\end{displaymath}
or the equivalent optimality condition
\begin{displaymath}
	(J'(u), v-u)\geq 0, \quad \forall v\in K,
\end{displaymath}
where $J(u)$ is a convex functional of $u$. It holds that
\begin{displaymath}
	u = \Pbk (u-\rho J'(u)).
\end{displaymath}
Furthermore, for any $u$, $v\in U$
\begin{displaymath}
	\|\Pbk u - \Pbk v\| \leq \|u-v\|.
\end{displaymath}
\end{lem}

\begin{proof}
\begin{displaymath}
	(u-(u-\rho J'(u)), v-u)  = \rho(J'(u), v-u)\geq 0, \forall v\in K.
\end{displaymath}
therefore $u = \Pbk (u-\rho J'(u))$.
	
Furthermore,
\begin{align*}
	(\Pbk u - \Pbk v, \Pbk u - \Pbk v) & = (\Pbk u - u + u - v + v-\Pbk v, \Pbk u - \Pbk v)\\
								  & = (u-v, \Pbk u - \Pbk v) + (\Pbk u - u,  \Pbk u - \Pbk v) \\
&+								    	    (\Pbk v - v, \Pbk v - \Pbk u)\\
							         & \leq (u-v, \Pbk u - \Pbk v) \\
							         & \leq \| u - v\| \|\Pbk u - \Pbk v\|.
\end{align*}

Hence,
\begin{displaymath}
	\|\Pbk u - \Pbk v\| \leq \| u - v\|.
\end{displaymath}
\end{proof}

The convergence results for this algorithm
are given in Theorem 8.2.1 and Remark 8.2.1 of \cite{liu2008adaptive}.

Let $V_h\subset H^1_0(\Omega)$ be the fine scale finite element space
associated with fine mesh triangulation $\mathcal{T}^h$, and $U_h\subset L^2(\Omega)$ is the piecewise constant finite element space associated with $\mathcal{T}^h_U$. $V_H$ and $U_H$ are the corresponding coarse mesh finite element spaces introduced in \S~\ref{sec:grps}. We have the following iterative algorithm to solve \eqref{eqn:ocp-opt-H}.

\begin{algorithm}[H]
\caption{Main algorithm for optimal control problem}
\label{alg:main}
\begin{algorithmic}
\STATE{\textbf{STEP 1}: Initialize
$u^{(0)}:=0, n:=0$, and tolerance $\ve>0$.}

\textbf{STEP 2}: Compute
\begin{equation}
\begin{array}{ll}
a(y^{(n)}, w) & = (f+Bu^{(n)}, w), y^{(n)} \in V_H, \forall w\in V_H\\
a(p^{(n)}, q) & = (g'(y^{(n)}), q), p^{(n)}\in V_H, \forall q\in V_H\\
(u^{(n+\frac12)}, v) & = (u^{(n)}, v) - \rho_n (h'(u^{(n)})+B^*p^{(n)} , v), u^{(n+\frac12)}, u^{(n)} \in U_H, \forall v\in U_H\\
u^{(n+1)} & = \Pbk u^{(n+\frac12)}.
\end{array}
\label{eqn:algorithm1}
\end{equation}

\textbf{STEP 3}: If $\|y^{(n+1)}-y^{(n)}\|< \ve$, \textbf{STOP}; else let $n=n+1$, goto \textbf{STEP 2}.
\end{algorithmic}
\end{algorithm}

Now, we are in the position to prove the convergence of the above algorithm.

\begin{thm}
The triplet $(y^{(n)}, u^{(n)}, p^{(n)})$ in Algorithm 1 converges to the triplet $(y_H, u_H, p_H)$ in \eqref{eqn:ocp-opt-H}. To be more precise, if we take $\rho^{(n)} = \rho$ with $0<\rho<1$ such that $1-2\rho M_h + 2\rho^2 L_h^2 < \delta^2/2 $ and $2\rho C^2_B L_g + 2\rho^2 C_B^4L_g^2 <  \delta^2/2$, with $0<\delta<1$, we have
\begin{displaymath}
\|y^{(n)}-y_H\|_1+ \|p^{(n)}-p_H\|_1+\|u^{(n)}-u_H\| \leq C\delta^n \|u_H\|.
\end{displaymath}
\end{thm}

\begin{proof}
	By \eqref{eqn:algorithm1} and \eqref{eqn:ocp-opt-H}, we have the following equation for $y^{(n)} - y_H$ and $p^{(n)}-p_H$,
	\begin{align*}
	a(y^{(n)}-y_H, w) & = (Bu^{(n)}-Bu_H, w), \quad y^{(n)} \in V_H, \forall w\in V_H,\\
	a(p^{(n)}-p_H, q) & = (g'(y^{(n)})-g'(y_H), q), \quad p^{(n)}\in V_H, \forall q\in V_H.
	\end{align*}
Therefore, we have
\begin{align*}
	\|y^{(n)}-y_H\|_1 & \leq CC_B \| u^{(n)}-u_H\|,\\
	\|p^{(n)}-p_H\|_1 & \leq CL_g \| y^{(n)}-y_H\|\leq CL_gC_B \|u^{(n)}-u_H\|,
\end{align*}
where the constant $C$ only depends on $\amin$, $\amax$, and $\diam (\Omega)$.

For simplicity of notation, we refer to $\Pbk$ as the projection from $U_H$ to $K_H$, and let
$\rho^{(n)} = \rho$. Therefore, $u^{(n+1)} = \Pbk (u^{(n)} - \rho (h'(u^{(n)})+B^*p^{(n)}))$, and
$u_H = \Pbk (u_H - \rho(h'(u_H)+B^*p_H))$. Hence,
\begin{align*}
	\|u^{(n+1)}-u_H\|^2  \leq &\| u^{(n)} -u_H - \rho (h'(u^{(n)})-h'(u_H)) - \rho (B^*p^{(n)}-B^*p_H)\|^2 \\
					 = &\| u^{(n)} -u_H \|^2 - 2\rho(h'(u^{(n)})-h'(u_H), u^{(n)}-u_H)\\
                       & - 2\rho(B^*p^{(n)}-B^*p_H, u^{(n)}-u_H)\\
						& + 	\rho^2 \|h'(u^{(n)})-h'(u_H) + B^*p^{(n)}-B^*p_H\|^2	\\		
				        \leq 	& \| u^{(n)} -u_H \|^2 - 2\rho M_h\| u^{(n)} -u_H \|^2+ 2\rho^2 L_h^2 \|u^{(n)} \\
				        	& + 2\rho C_B\|p^{(n)}-p_H\| \|u^{(n)} - u_H\|-u_H \|^2 + 2\rho^2C_B^2\|p^{(n)}-p_H\|^2 \\
					\leq  &(1-2\rho M_h + 2\rho^2 L_h^2 + 2\rho C^2_B L_g + 2\rho^2 C_B^4L_g^2) \|u^{(n)} -u_H \|^2.
\end{align*}
Take $0<\rho<1$ such that $1-2\rho M_h + 2\rho^2 L_h^2 < \delta^2/2 $ and $2\rho C^2_B L_g + 2\rho^2 C_B^4L_g^2 <  \delta^2/2$, with $\delta<1$, we have
\begin{displaymath}
	\|u^{(n+1)}-u_H\|  \leq \delta\|u^{(n)}-u_H\|.
\end{displaymath}
Therefore, $\|y^{(n)}-y_H\|_1+ \|p^{(n)}-p_H\|_1+\|u^{(n)}-u_H\| \leq C\delta^n \|u_H\|$.
\end{proof}

We can also write down the matrix form of the Algorithm \ref{alg:main} in the following. Denote
\begin{eqnarray*}
\mathcal{F}&=&(F_1,F_2,\ldots, F_N)^T,\text{ where }F_i=\int_{\Omega} f\phi_i,\\
\mathcal{Y}_d&=&(Y_{d,1},Y_{d,2},\ldots, Y_{d,N})^T, \text{ where }Y_{d,i} = \int_{\Omega}\phi_i y_d,\\
\mathcal{S}&=&(S_{ij})_{N\times N},\text{ where } S_{ij}=\int_{\Omega}a_{ij}\frac{\partial
\phi_i}{\partial x_i}\frac{\partial \phi_j}{\partial x_j},\\
\mathcal{D}&=&(D_{ik})_{N\times m},\text{ where } D_{ik}=\int_{\Omega}B\varphi_k \phi_i,\\
\mathcal{Q}&=&(Q_{ij})_{N\times N},\text{ where } Q_{ij}=\int_{\Omega}\phi_i\phi_j,\\
\mathcal{M}&=&(M_{kl})_{m\times m},\text{ where } M_{kl}=\int_{\Omega_U}\varphi_k\varphi_l.
\end{eqnarray*}
where $i,j=1,2,\ldots, N$; $k,l=1,2,\ldots,m$.  The iterative scheme for
gradient descent  algorithm is as follows:
\begin{algorithm}[H]
\caption{}
\label{alg:buildtree}
\begin{algorithmic}

 \STATE{\textbf{STEP 1}: Initialize
$U^{(0)}:=(0,0,\ldots,0)_{1\times m}^T, n:=0$.}

\textbf{STEP 2}: Compute
\begin{eqnarray*}
Y^{(n)}&=&\mathcal{S}^{-1}(\mathcal{F}+\mathcal{D}U^{(n)})\\
P^{(n)}&=&\mathcal{S}^{-1}(\mathcal{Q}Y^{(n)}-\mathcal{Y}_d)\\
U^{(n+\frac{1}{2})}&=&U^{(n)}-\rho^{(n)}(U^{(n)}+\mathcal{M}^{-1}\mathcal{D}^TP^{(n)})\\
U^{(n+1)}&=&\Pbk (U^{(n+1/2)}).
\end{eqnarray*}

\textbf{STEP 3}: If $\|y^{(n+1)}-y^{(n)}\|$ is sufficiently small, then
stop; else let $n=n+1$, goto step 2.
\end{algorithmic}
\end{algorithm}
Therefore, the main computational cost of the iterative algorithm is to
solve the state and control equations.

\subsection{Numerical Results}
\label{sec:results}
In this section, we present two numerical examples to verify the error
estimates presented in the previous sections. We
consider the elliptic optimal control problem with rough
coefficients,
\begin{eqnarray*}
\min_{u\in K \subset L^2(\Omega_U)}\frac{1}{2}\int_{\Omega}(y-y_d)^2dx+\frac{1}{2}\int_{\Omega}u^2dx\\
-\diiv(a(x)\nabla y)=Bu+f.
\end{eqnarray*}
and the constraint set $K=\{v\in L^2(\Omega)|\int_{\Omega}v\geq
0\}$. The dual equation of the state equation is given by,
\begin{displaymath}
-\diiv(a(x)\nabla p)=y-y_d.
\end{displaymath}
For simplicity, we use Dirichlet boundary condition and set $f=1$. The continuous linear operator is chosen $B = I$,
where $I$ is the identity operator.

We test the numerical methods for two different types of diffusion coefficients. The first is a multiscale trigonometric function, and the second is the SPE10 benchmark for reservoir simulation (\url{http://www.spe.org/web/csp/}).

\def\Nc{N_c}
\def\TH{\mathcal{T}_H}

\subsubsection{Multiscale Trigonometric Example}
\label{sec:trig}

For the first example,  $a(x)$ is a scalar function given by the following expression,
\begin{eqnarray}
 a(x)&=&\frac{1}{6}\big(\frac{1.1+\sin(2\pi x/\varepsilon_1)}{1.1+\sin(2\pi y/\varepsilon_1)}+\frac{1.1+\sin(2\pi y/\varepsilon_2)}{1.1+\cos(2\pi x/\varepsilon_2)}+\frac{1.1+\cos(2\pi x/\varepsilon_3)}{1.1+\sin(2\pi y/\varepsilon_3)}+\nn\\
&&\frac{1.1+\sin(2\pi y/\varepsilon_4)}{1.1+\cos(2\pi x/\varepsilon_4)}+\frac{1.1+\cos(2\pi x/\varepsilon_5)}{1.1+\sin(2\pi y/\varepsilon_5)}+\sin(4x^2y^2)+1\big). \label{9.3}
\end{eqnarray}
 where $\ds\varepsilon_1=\frac{1}{5}, \varepsilon_2=\frac{1}{13}, \varepsilon_3=\frac{1}{17}, \varepsilon_4=\frac{1}{31},
\varepsilon_5=\frac{1}{65}.$ The coefficient $a(x)$ is highly oscillatory with non-separable scales, as shown in Figure \ref{fig:trig}.

\begin{figure}[H]
\centering
\includegraphics[width=0.45\textwidth]{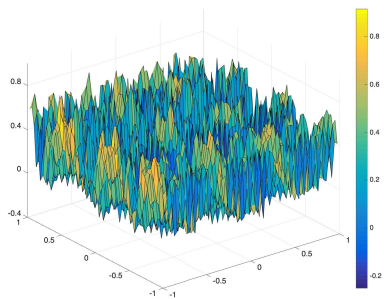}
\caption{Coefficients $a(x)$ in $log_{10}$ scale.}
\centering
\label{fig:trig}
\end{figure}

We take the domain as the unit square
$\Omega_U=\Omega=[0,1]\times[0,1]$. The regular coarse mesh
$\mathcal{T}_H$ is obtained by first subdividing $\Omega$ uniformly into $\Nc\times \Nc$
squares, then each square can be partitioned into two triangles along the
 $(1,1)$ direction. We can further refine the coarse mesh uniformly by dividing each triangle into four similar subtriangles. We refine $\TH$ $J$ times to obtain the fine mesh $\mathcal{T}_h$, therefore, $H = 2^J h$.  Let us refer to Figure \ref{fig:mesh} for an
illustration of the partition. The degrees of freedom of global RPS basis and the global GRPS basis are $(N_c-1)^2$ and $2N_c^2$, respectively.

\begin{figure}[H]
\centering
\subfigure[Coarse mesh, $N_c$=2]{
\includegraphics[width=0.31\textwidth]{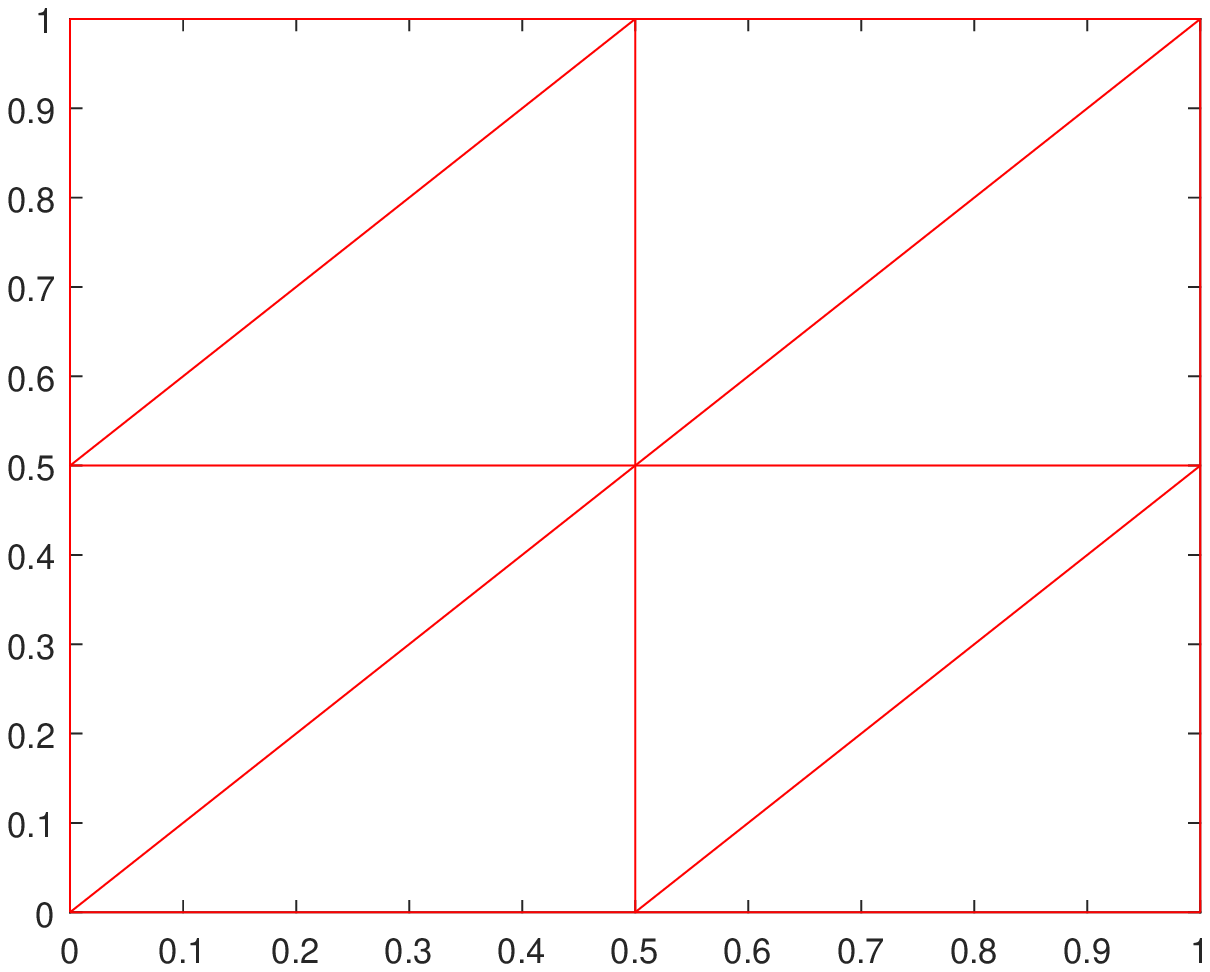}}
\subfigure[Fine mesh, $N_c$=1,J=2]{
\includegraphics[width=0.31\textwidth]{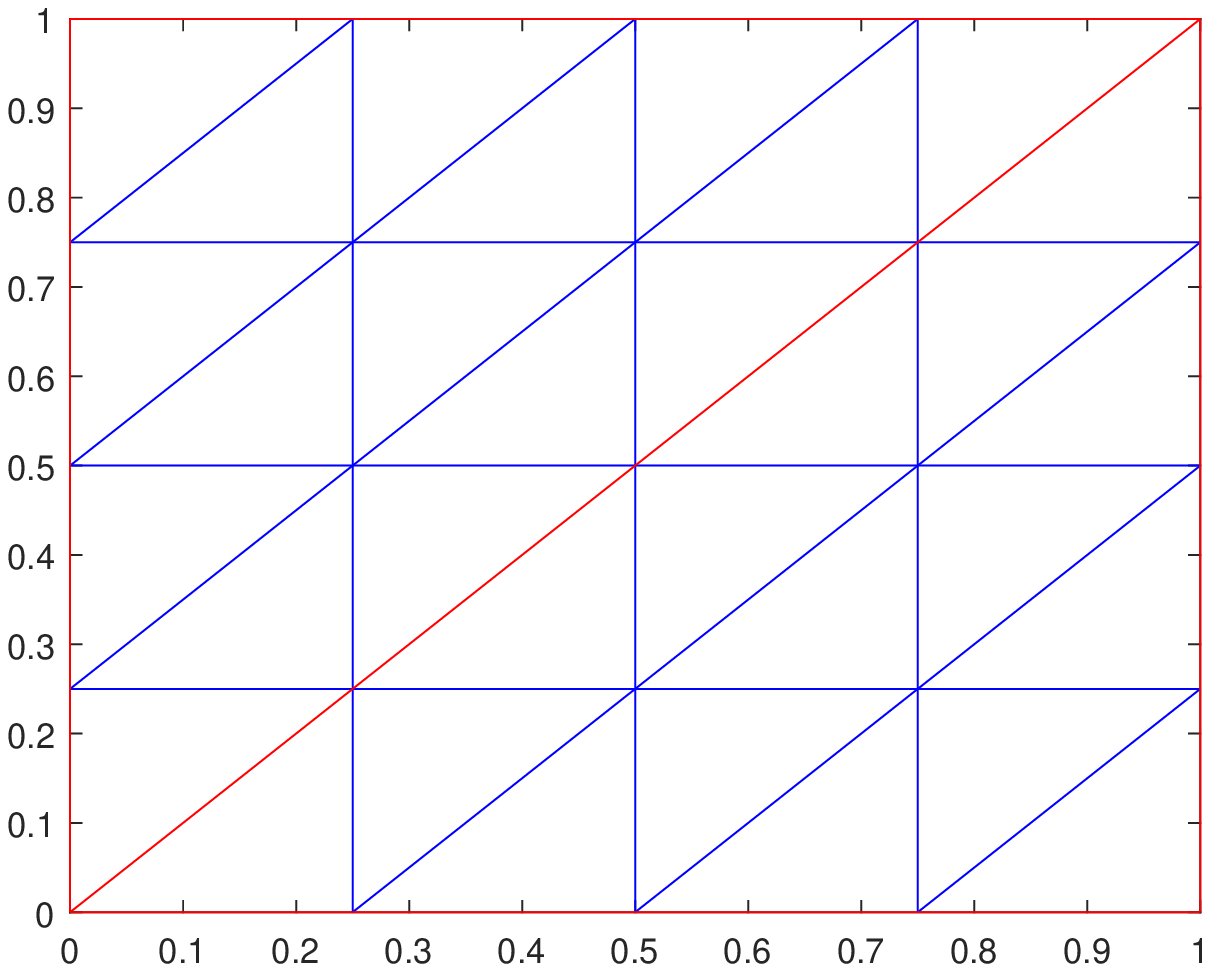}}
\subfigure[Fine mesh, $N_c$=1,J=3]{
\includegraphics[width=0.31\textwidth]{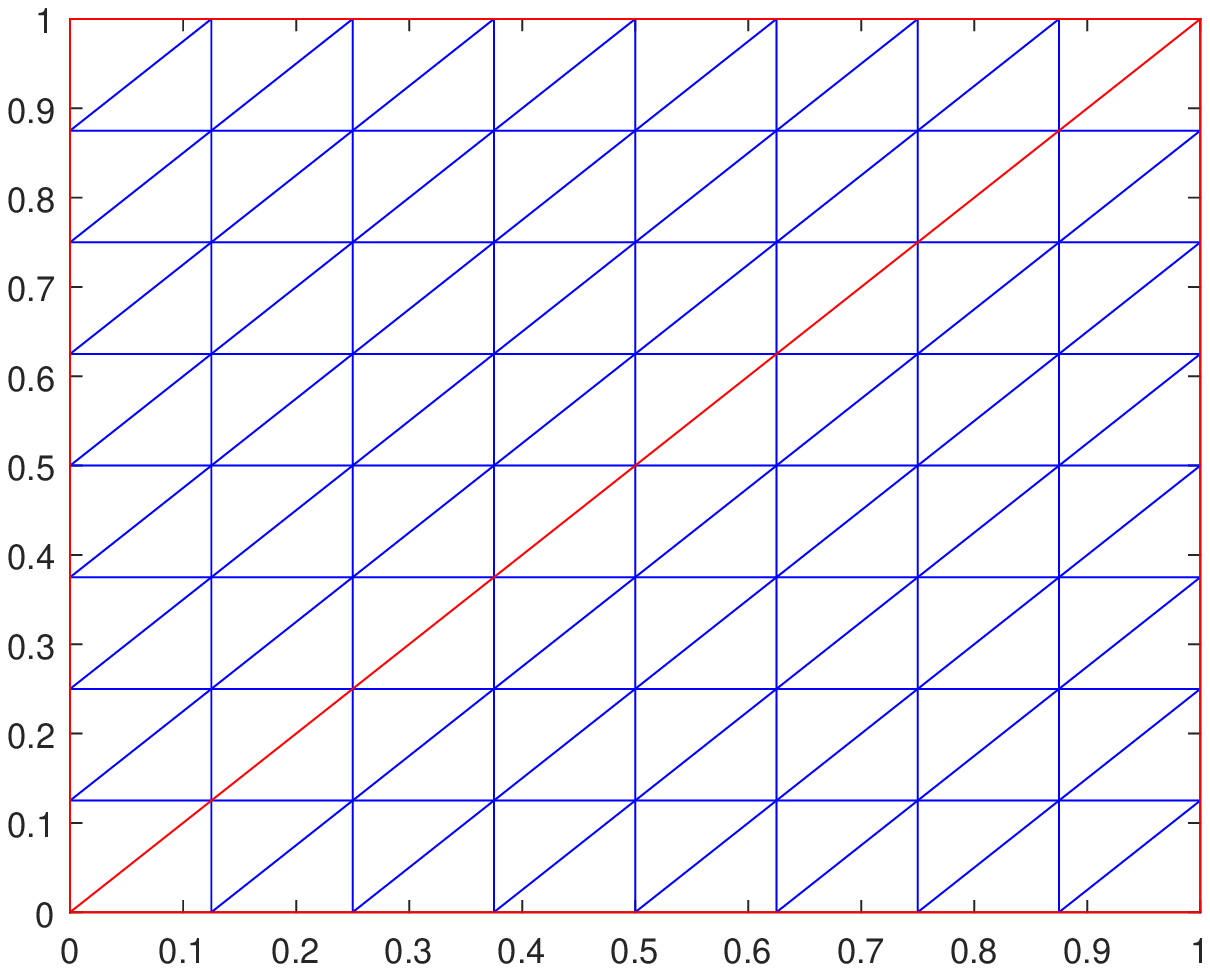}}
\caption{Coarse and fine mesh of the unit square.}
\centering
\label{fig:mesh}
\end{figure}

We compute localized RPS basis
$\phi_i^{l}$ on localized sub-domains $\Omega_i^l$ defined by adding $l$
layers of coarse triangles around coarse node $x_i$. More precisely $\Omega_i^1$ is the
union of triangles $T\in \mathcal{T}_H$ with $x_i$ as a common node, and $\ds\Omega_i^{l+1}: = \cup_{T\cap \Omega_i^l \neq \emptyset} T$.

Similarly, we compute localized GRPS basis
$\phi_i^{l}$ on localized sub-domains $\bar \Omega_i^l$ defined by adding $l$
layers of coarse triangles around coarse triangle $T_i$. More precisely $\bar \Omega_i^1$ is the
union of triangles $T\in \mathcal{T}_H$ with $T_i$ as common node, common  edge, and $\ds\bar\Omega_i^{l+1}: = \cup_{T\cap \bar\Omega_i^l \neq \emptyset} T$.  We refer to Figure \ref{fig:patch} for an illustration of the patches $\Omega_i^l$ and $\bar\Omega_i^l$ of the RPS basis and GRPS basis with $l=1,2,3$.

\begin{figure}[H]
\centering
\subfigure[$\Omega_i^1$]
{\includegraphics[width=0.31\textwidth]{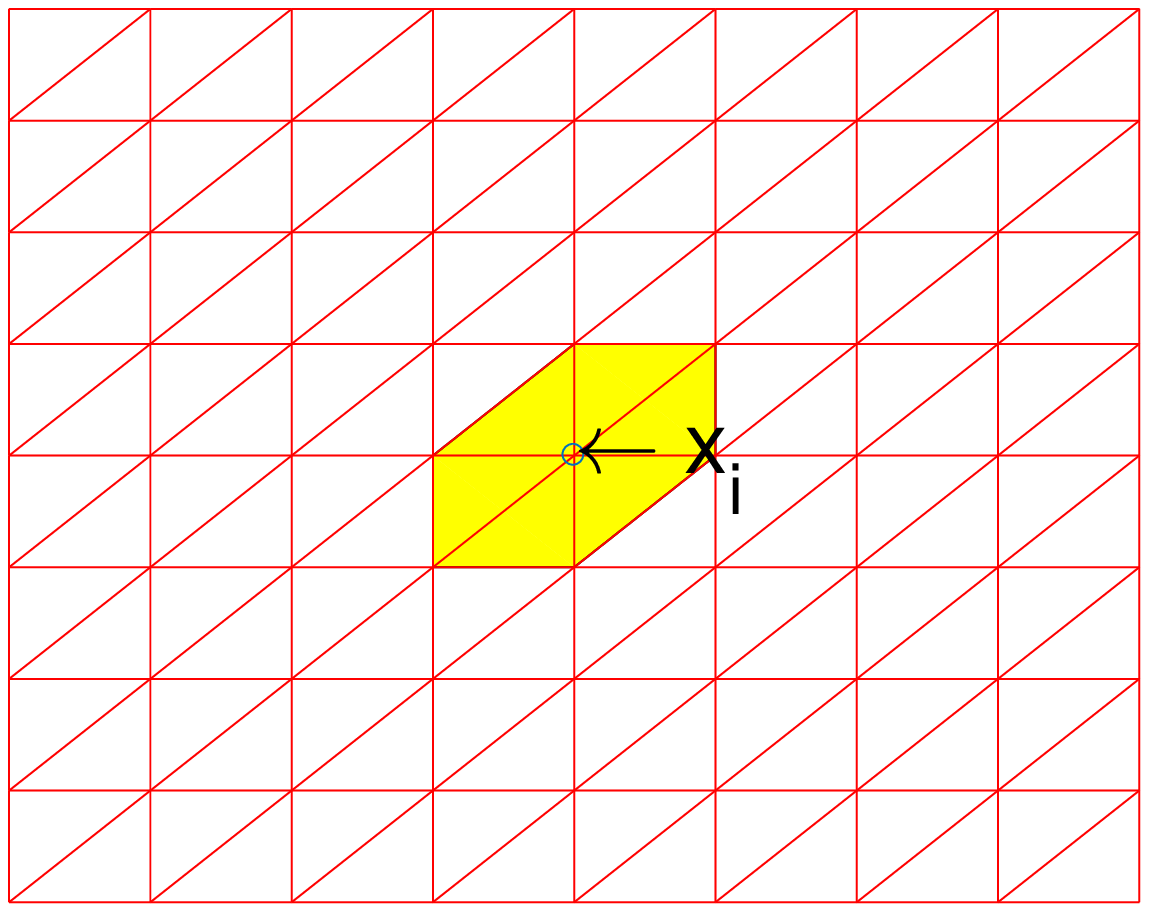}}
\subfigure[$\Omega_i^2$]
{\includegraphics[width=0.31\textwidth]{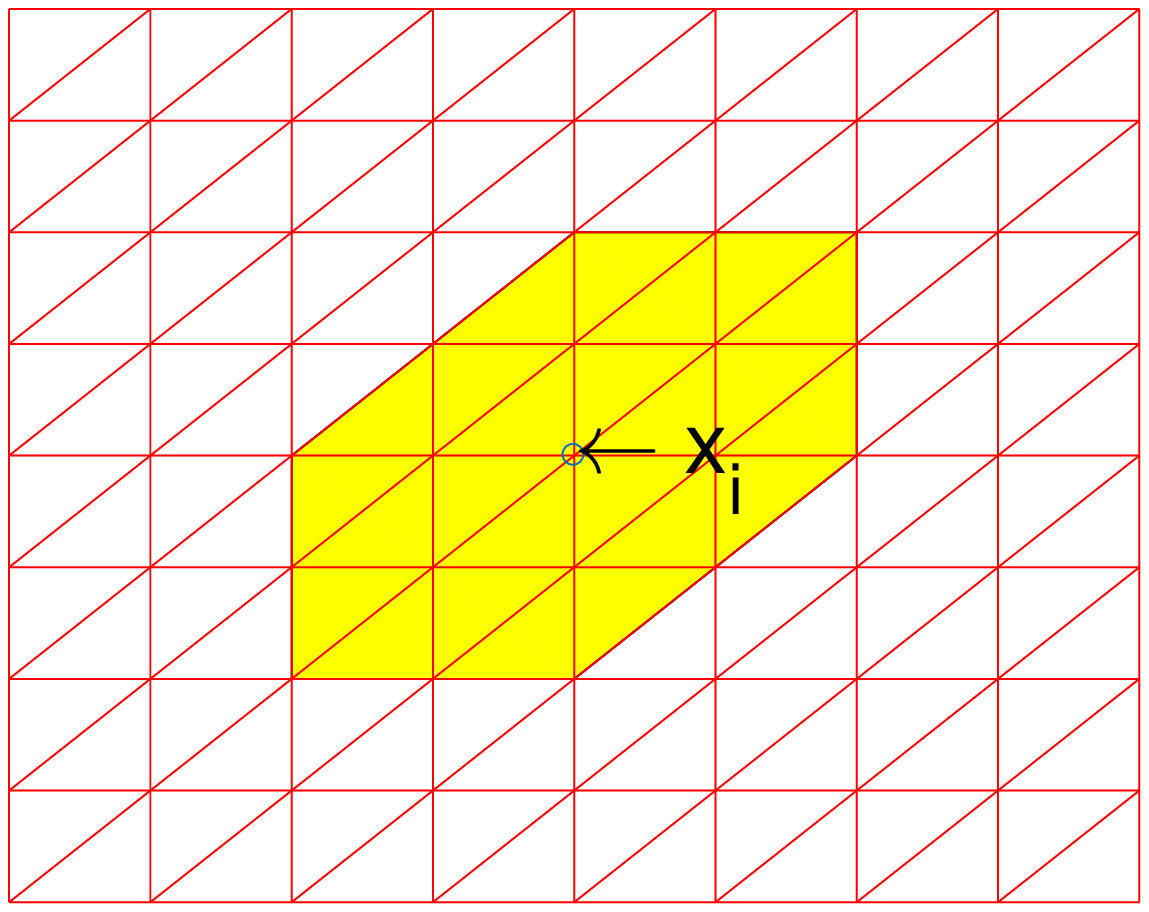}}
\subfigure[$\Omega_i^3$]
{\includegraphics[width=0.31\textwidth]{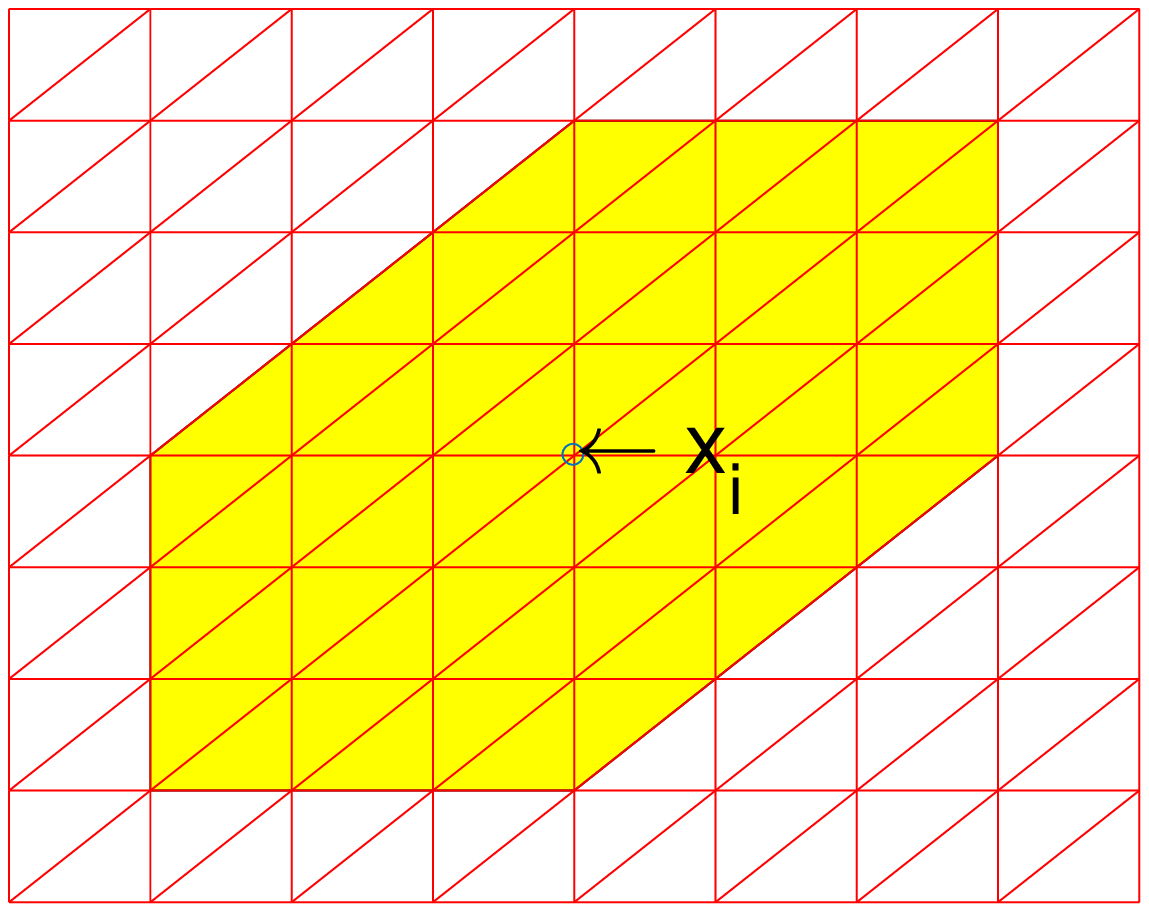}}

\subfigure[$\bar\Omega_i^1$]
{\includegraphics[width=0.31\textwidth]{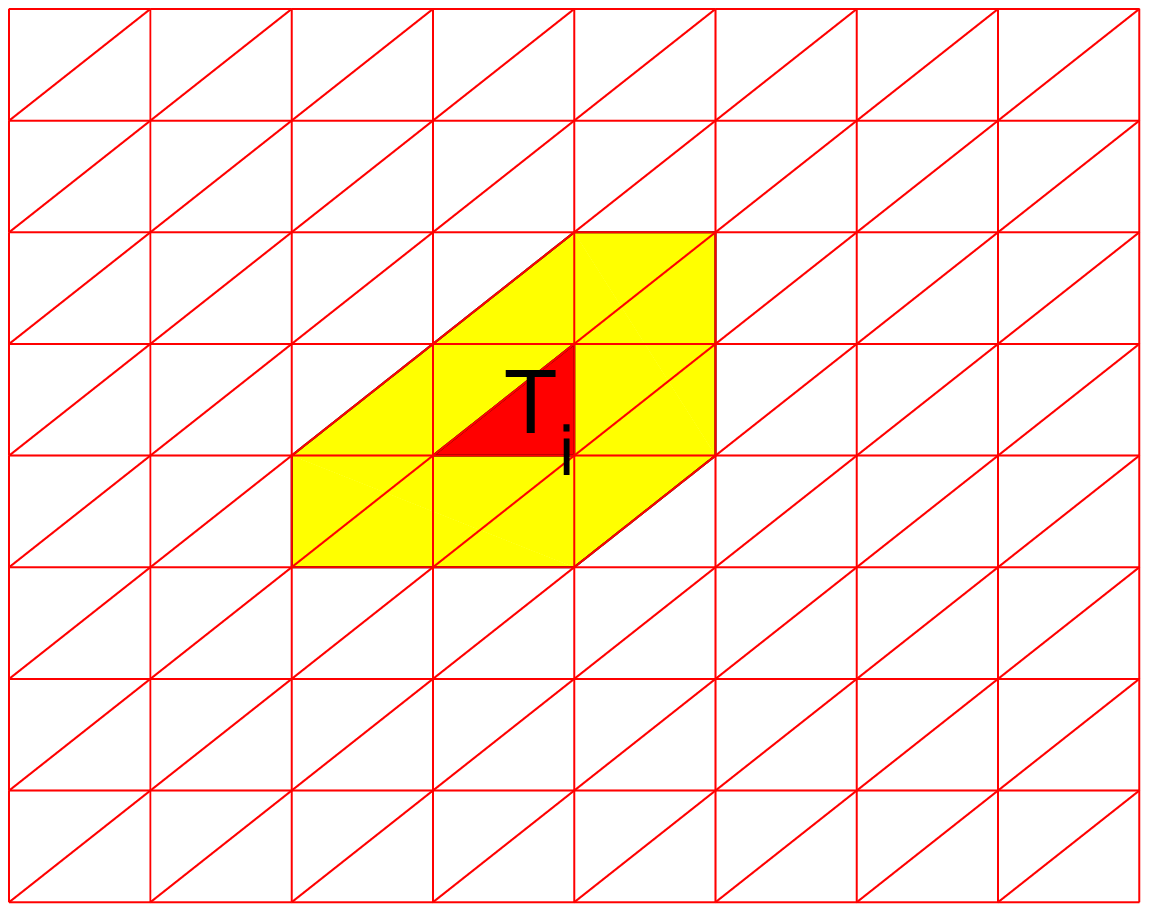}}
\subfigure[$\bar\Omega_i^2$]
{\includegraphics[width=0.31\textwidth]{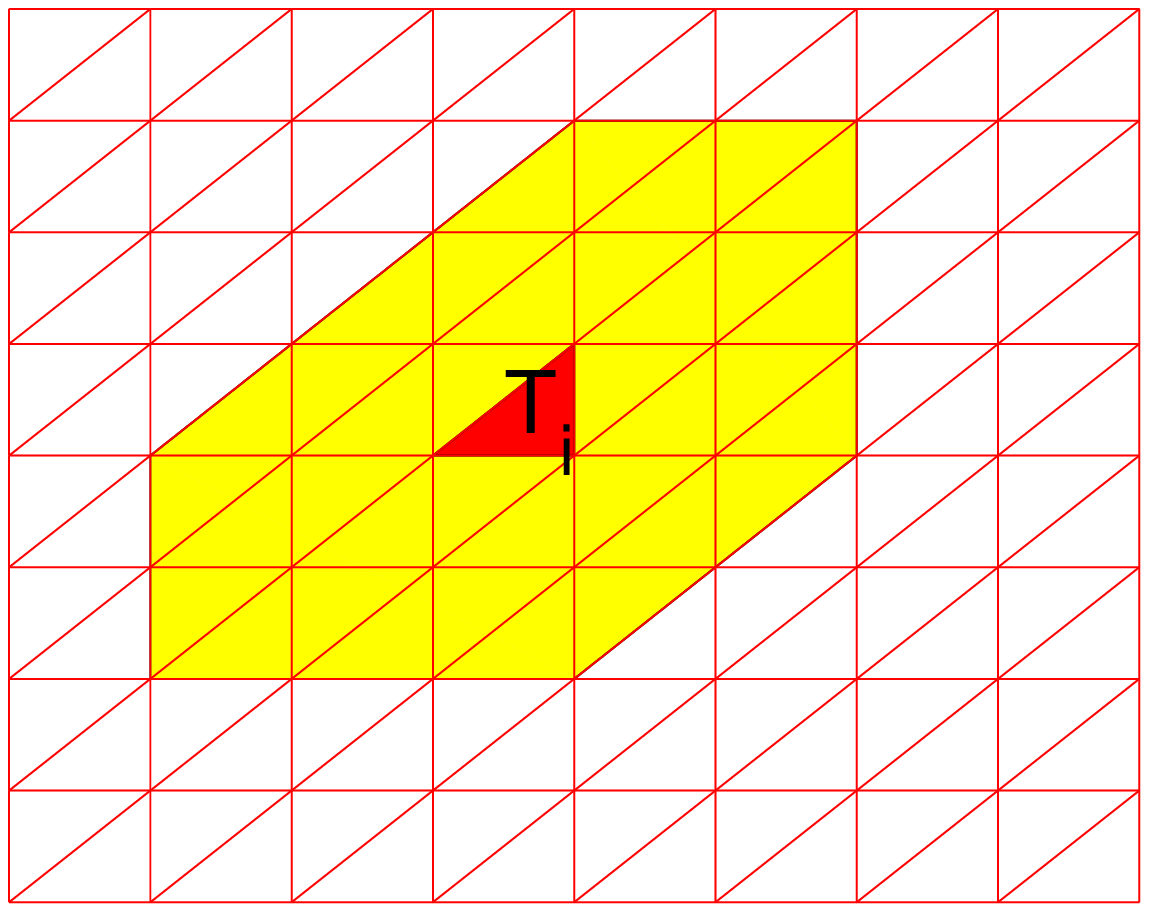}}
\subfigure[$\bar\Omega_i^3$]
{\includegraphics[width=0.31\textwidth]{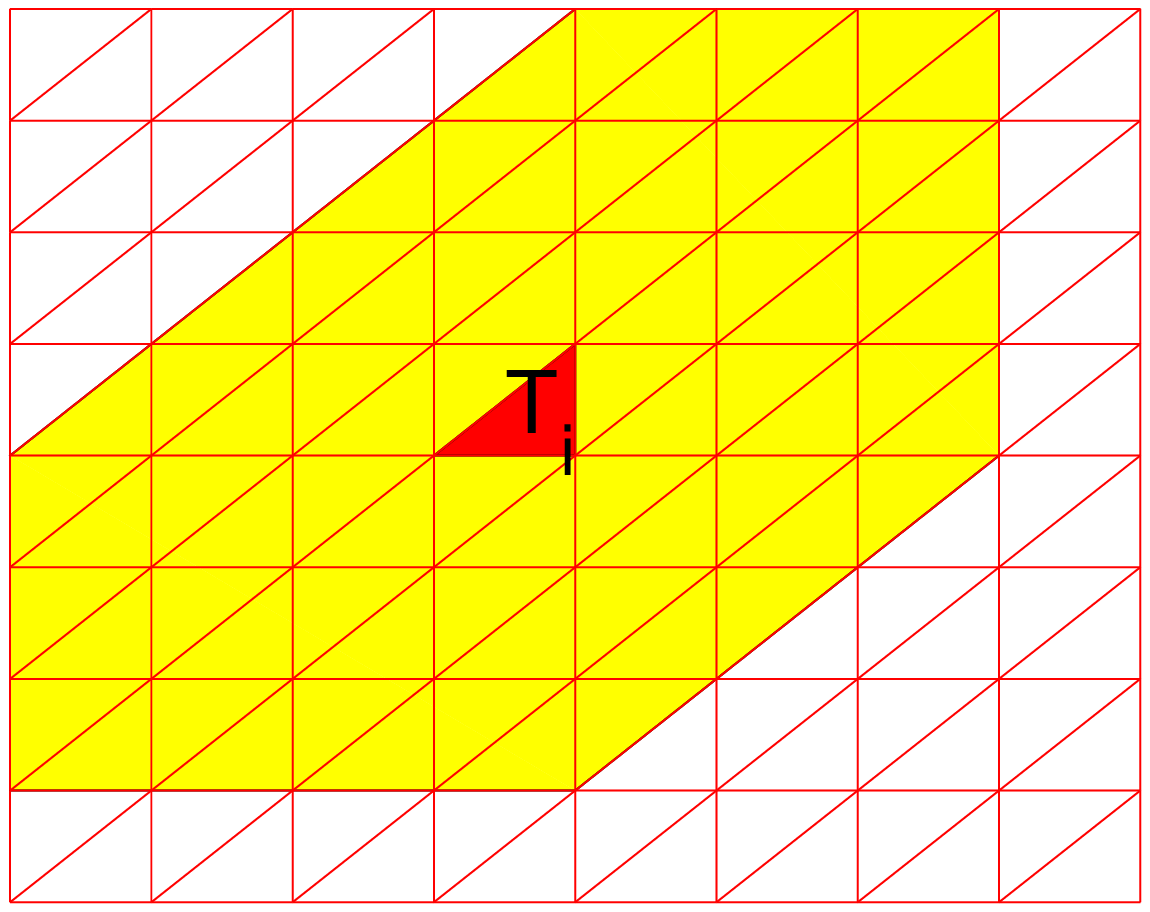}}
\caption{Local patches for RPS and GRPS basis}
\centering
\label{fig:patch}
\end{figure}

To better understand the exponential decay properties of the RPS basis and GRPS basis, we plot
the global RPS basis and the localized RPS basis in Figure
\ref{fig:decay}. The coarse mesh has $H=1/32$, and the fine mesh has
$h=1/256$ (namely, $\Nc=32$, $J=3$). Figure \ref{fig:decay}(a) shows the
shape of global basis $\phi_i$ centered at $(0.5, 0.5)$ in the $\log_{10}$ scale.
Figure \ref{fig:decay}(b) shows a slice of $\phi_{481}$ along the
 x-axis. Figure \ref{fig:localbasis} shows the localized basis $\phi_i^{loc}$ for the same node
for various levels of localization(i.e., for $l=1,4,7$).  It
is consistent with the exponential decay and localization results.
\vspace{-10pt}
\begin{figure}[H]
\centering
\subfigure{
\includegraphics[width=0.45\textwidth]{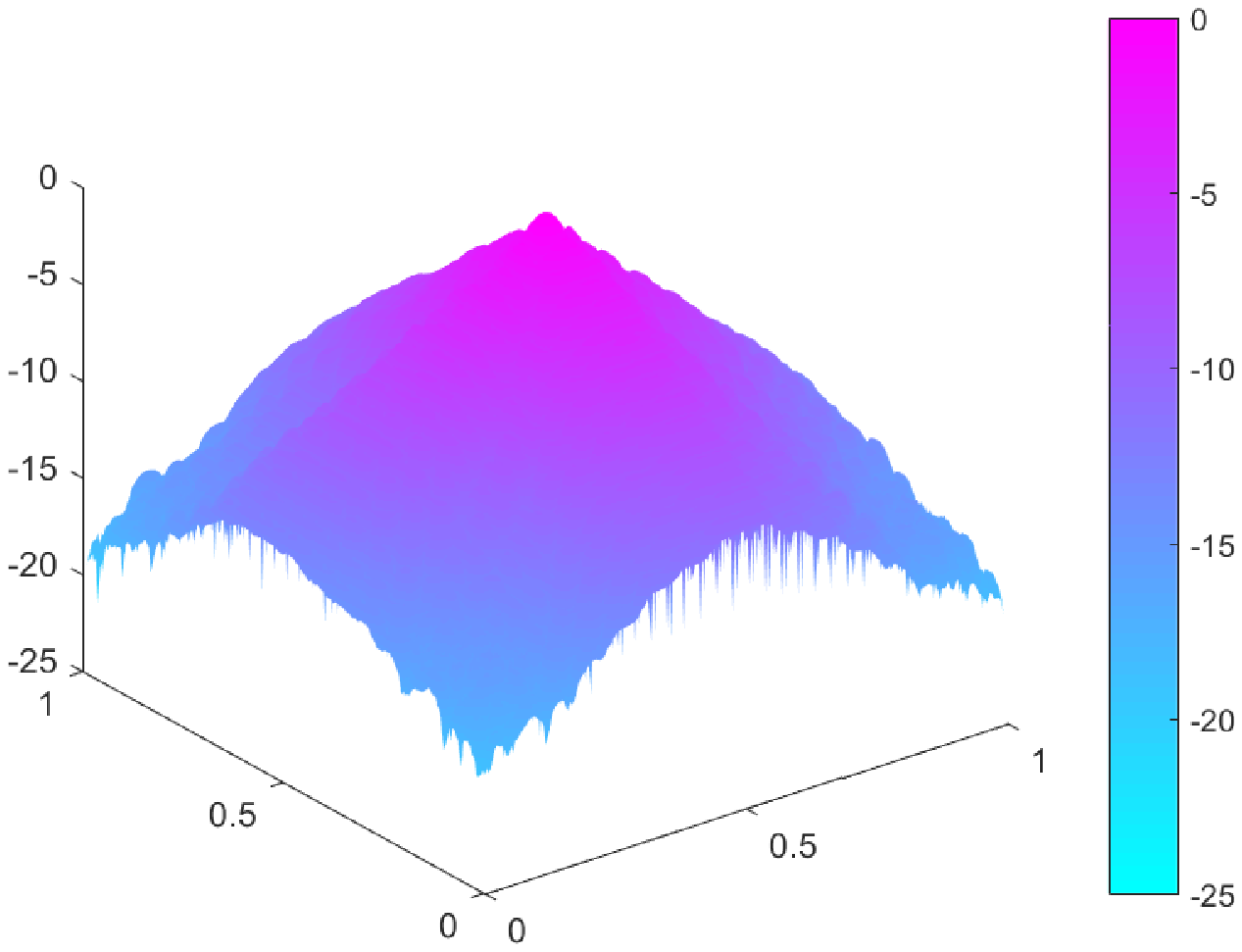}}
\subfigure{
\includegraphics[width=0.45\textwidth]{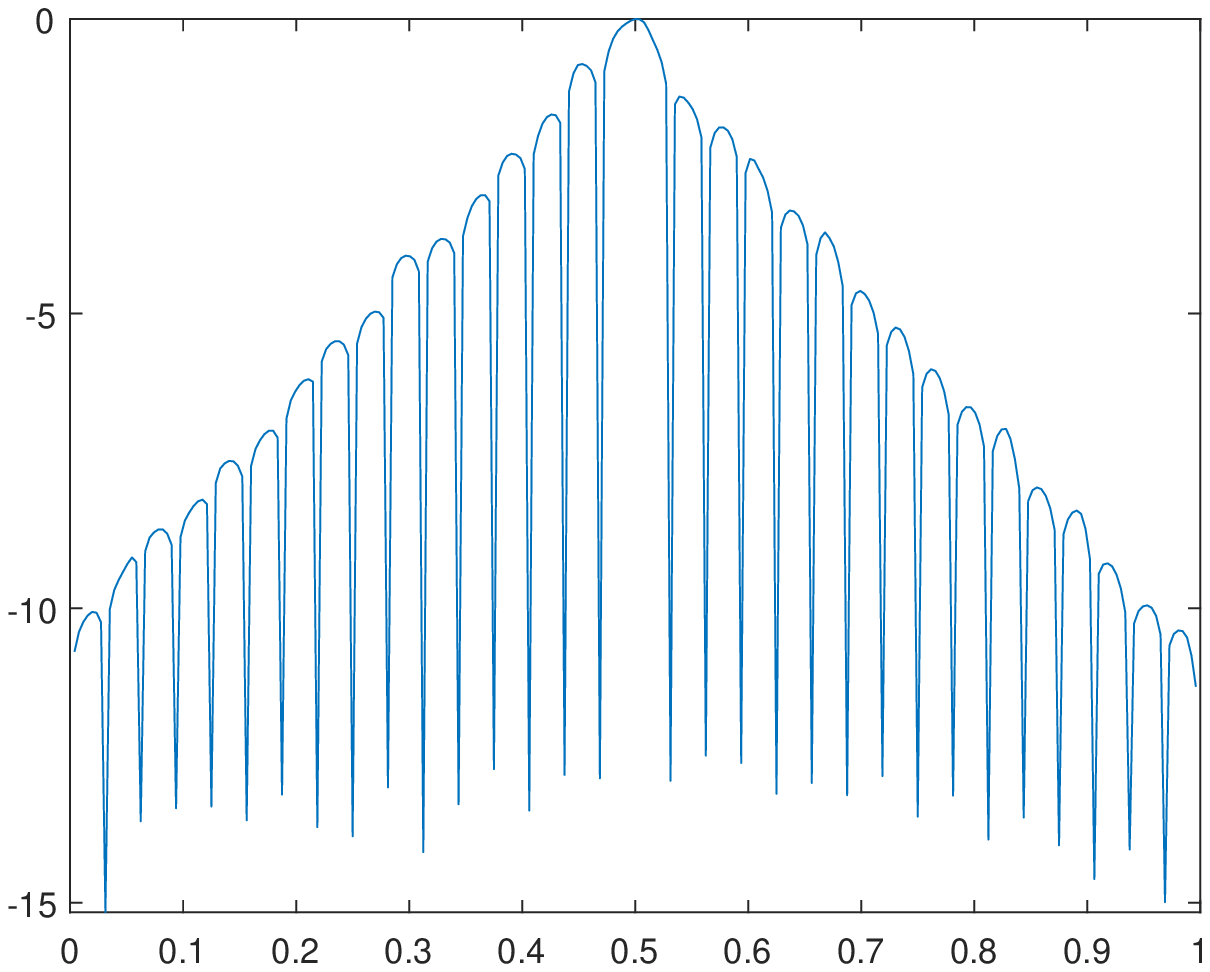}}
\caption{The shape and 1d slice along $x$ axis for the global RPS basis $\phi_i$ in $log_{10}$ scale.}
\centering
\label{fig:decay}
\end{figure}

\vspace{-10pt}

\begin{figure}[H]
\centering
\subfigure[Slice of $\phi_i^{1}$ along $x$ axis]{
\includegraphics[width=0.3\textwidth]{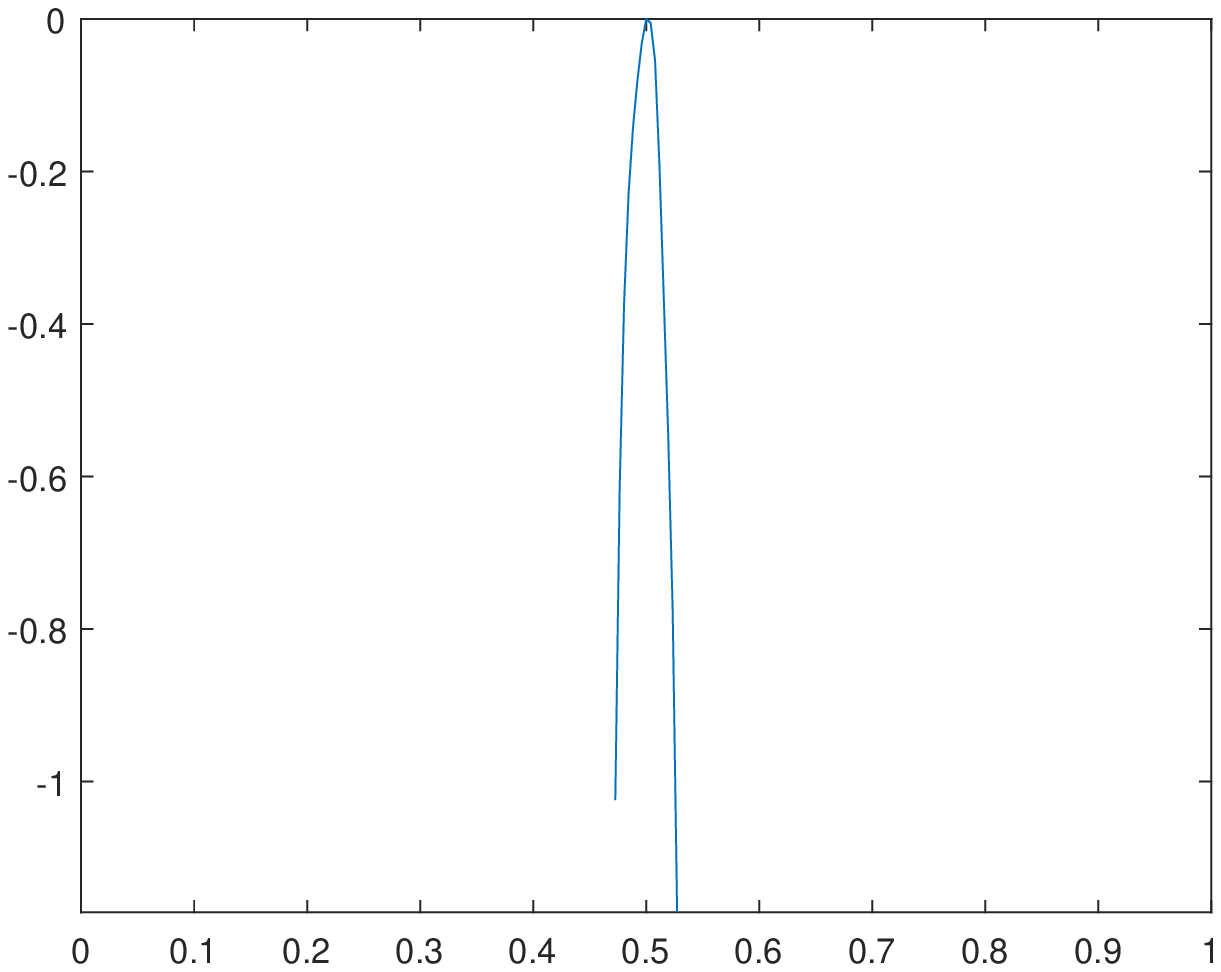}}
\subfigure[Slice of $\phi_i^{4}$ along $x$ axis]{
\includegraphics[width=0.3\textwidth]{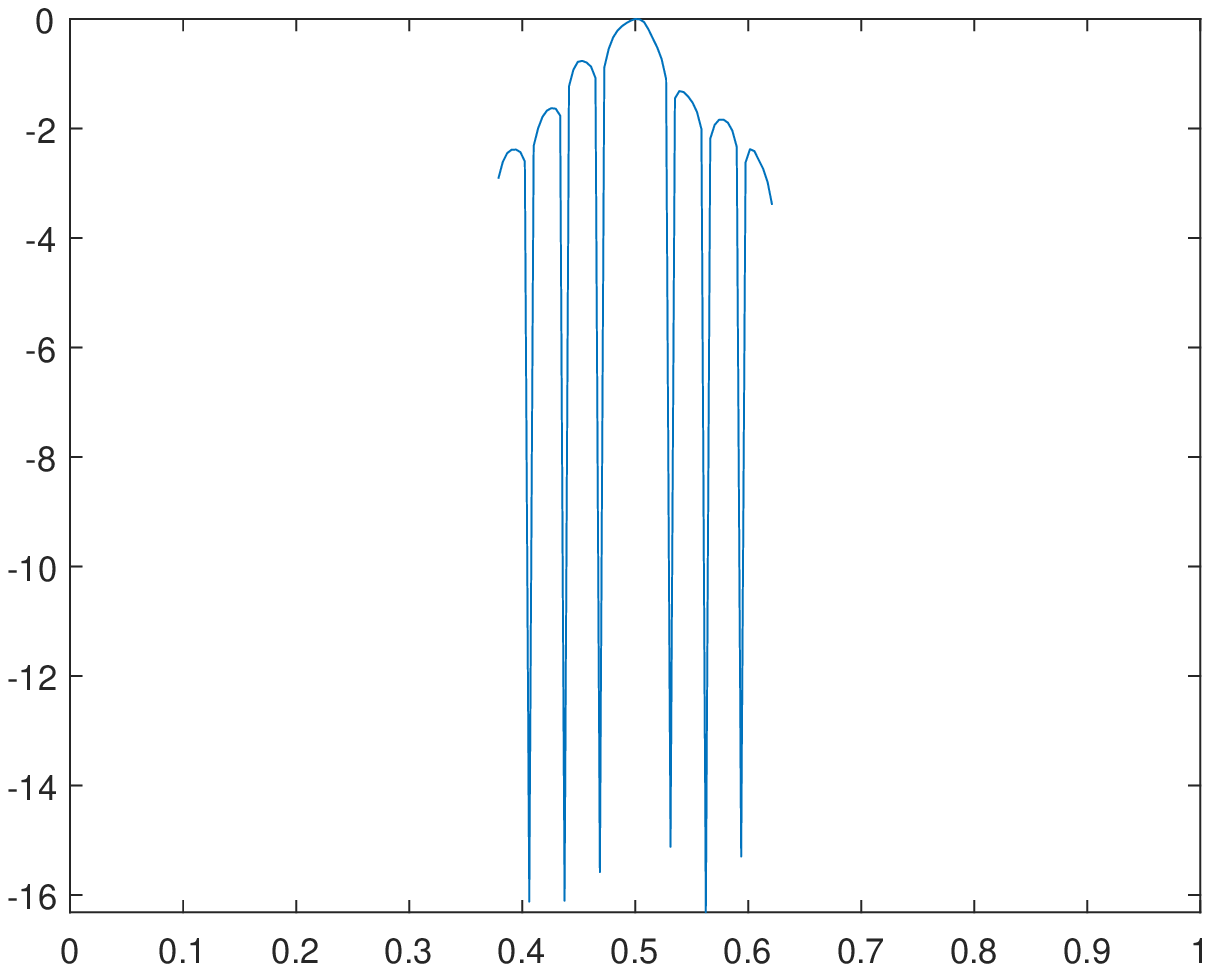}}
\subfigure[Slice of $\phi_i^{7}$ along $x$ axis]{
\includegraphics[width=0.3\textwidth]{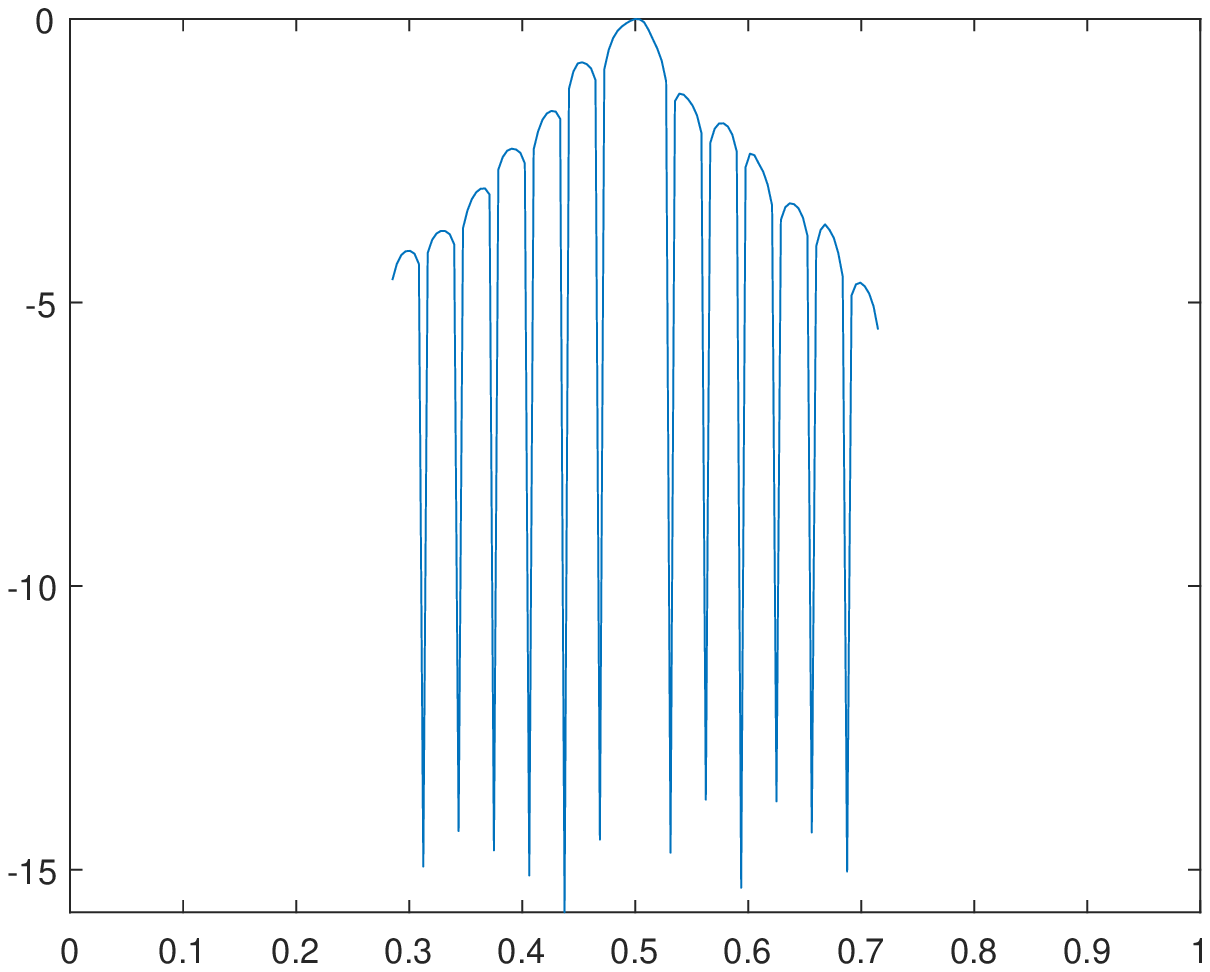}}
\caption{The slice of the localized RPS basis $\phi_i^{l}$ along $x$ axis in $log_{10}$ scale for the node $i=481$ for various degrees of localization (i.e., for $l=1,4, 7$).
The coarse mesh $H=1/32$, the fine mesh $h=1/256$, namely, $Nc=32$ and $J=3$. }
\centering
\label{fig:localbasis}
\end{figure}

We compute the optimal control problem using Algorithm \ref{alg:main} with RPS space $V_H$ or GRPS space $\bar V_H$.

Figure \ref{fig:errloc} shows the relative error of $y_H$ and $p_H$  using GRPS $\bar V_H$, in $L^2$, $H^1$ and $L^{\infty}$ norm in the
$\log_{10}$ scale as a function of the number of layers $l$.

\begin{figure}[H]
\centering
\subfigure[$\|y-y_{H}\|$ in $L^2$, $H^1$ and $L^{\infty}$]{
\includegraphics[height = 3cm, width=0.45\textwidth]{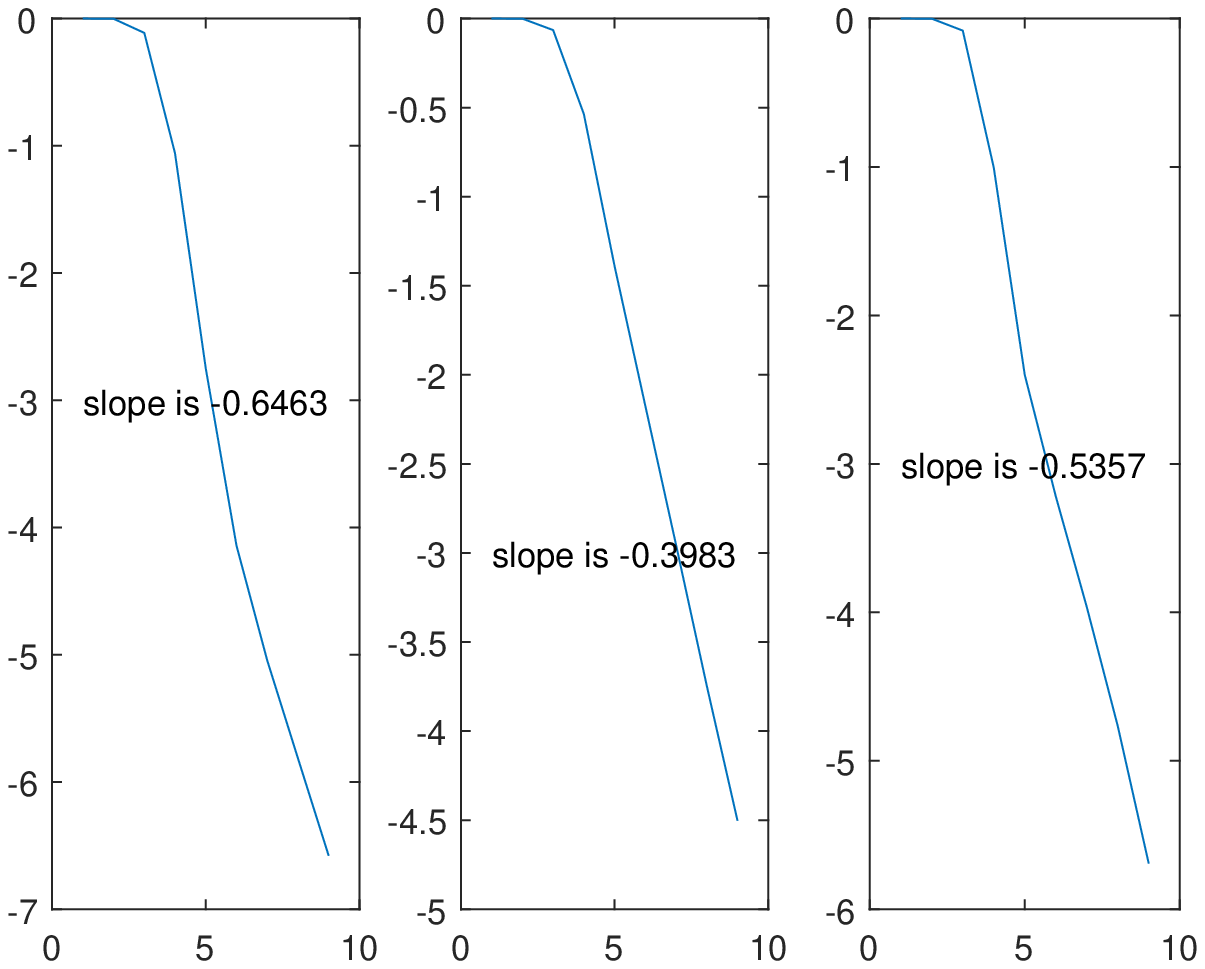}}
\subfigure[$\|p-p_{H}\|$ in $L^2$, $H^1$ and $L^{\infty}$]{
\includegraphics[height = 3cm, width=0.45\textwidth]{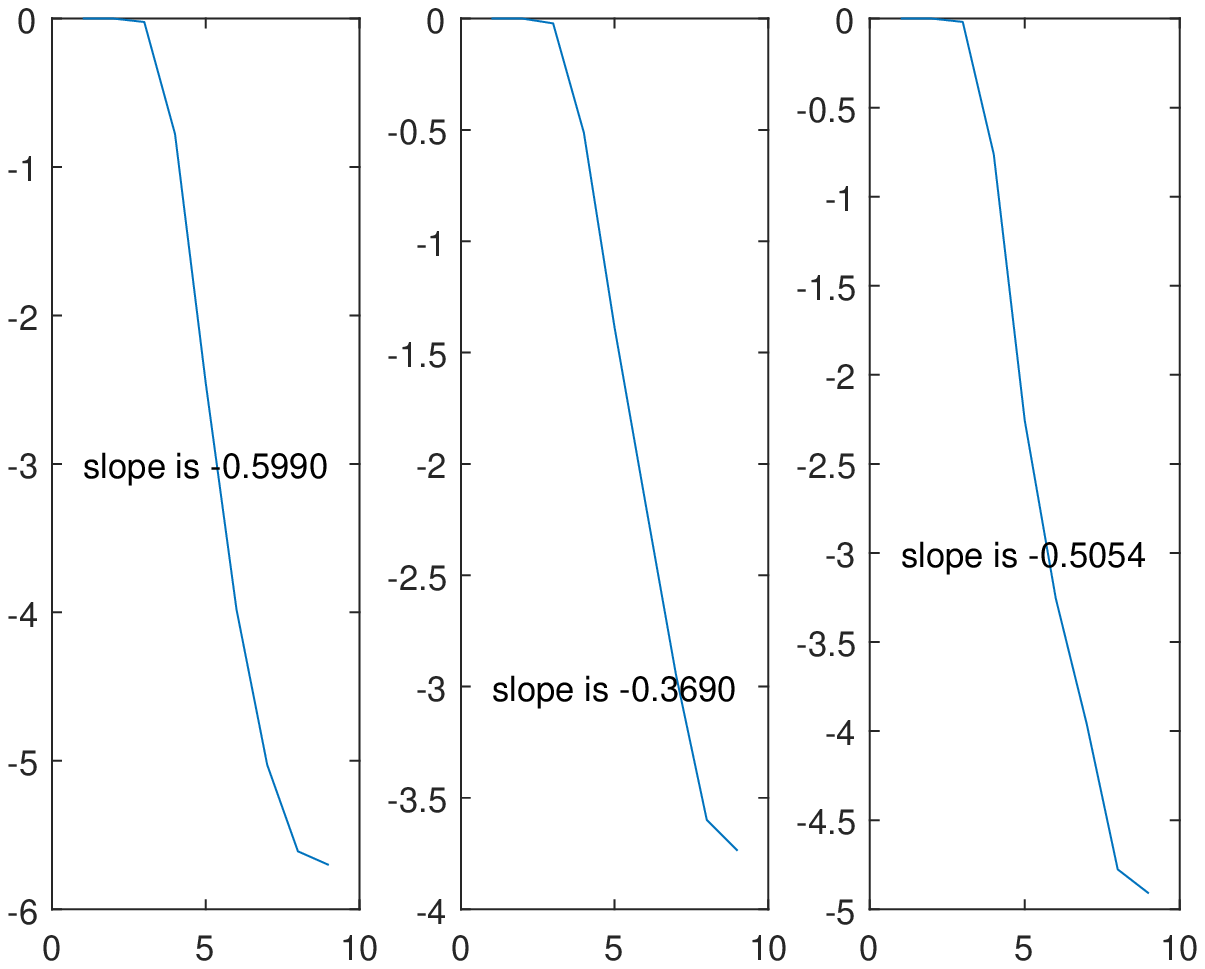}}
\caption{The error of $y_H$ and $p_H$ in $L^2$, $H^1$ and $L^{\infty}$ norms in the $\log_{10}$ scale as a function of $l$ (number of layers).
x-axis stands for $l$, y-axis stands for  the coarse mesh error estimates in $\log_{10}$ scales using GRPS basis. The coarse mesh size is $H=1/64$, the fine mesh size is $h=1/256$, namely, $N_c=64$ and $J=2$. }
\centering
\label{fig:errloc}
\end{figure}

Figure \ref{fig:coarseerr}(a), (b) and (c) show the relative errors of
$\|y-y_{H}\|_{H^1(\Omega)}$, $\|p-p_{H}\|_{H^1(\Omega)}$ and
$\|u-u_{H}\|_{L^2(\Omega)}$ in $\log_{10}$ scale with respect to the number of layers $l = 1, \cdots, 9$, and (d), (e), (f) show the relative error with respect to coarse \dof. Note that the fine mesh is fixed with $h=1/256$.

\begin{figure}[H]
\centering
\subfigure[$\|y-y_{H}\|_{H^1(\Omega)}$ in log scale]{
\includegraphics[width=0.31\textwidth]{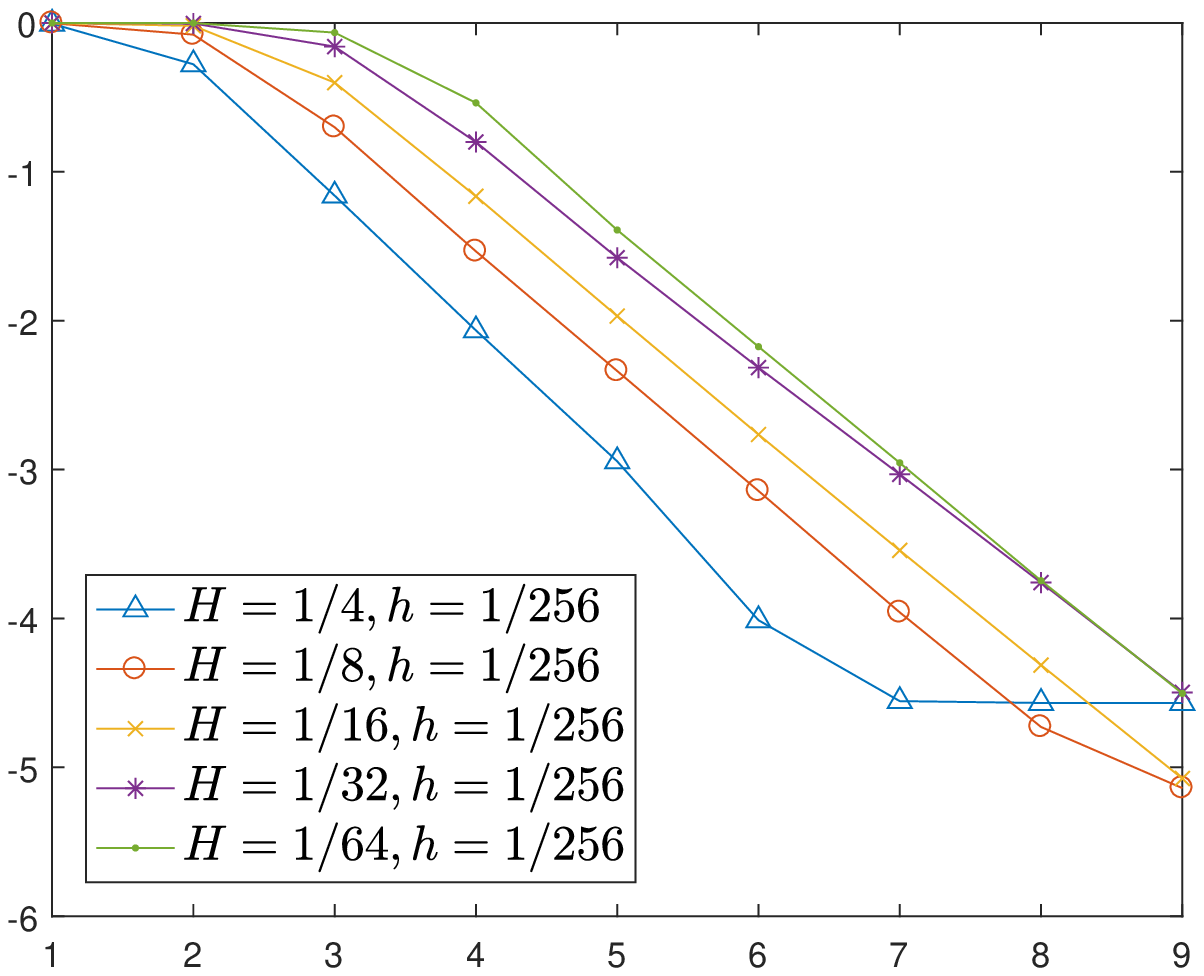}}
\subfigure[$\|p-p_{H}\|_{H^1(\Omega)}$ in log scale]{
\includegraphics[width=0.31\textwidth]{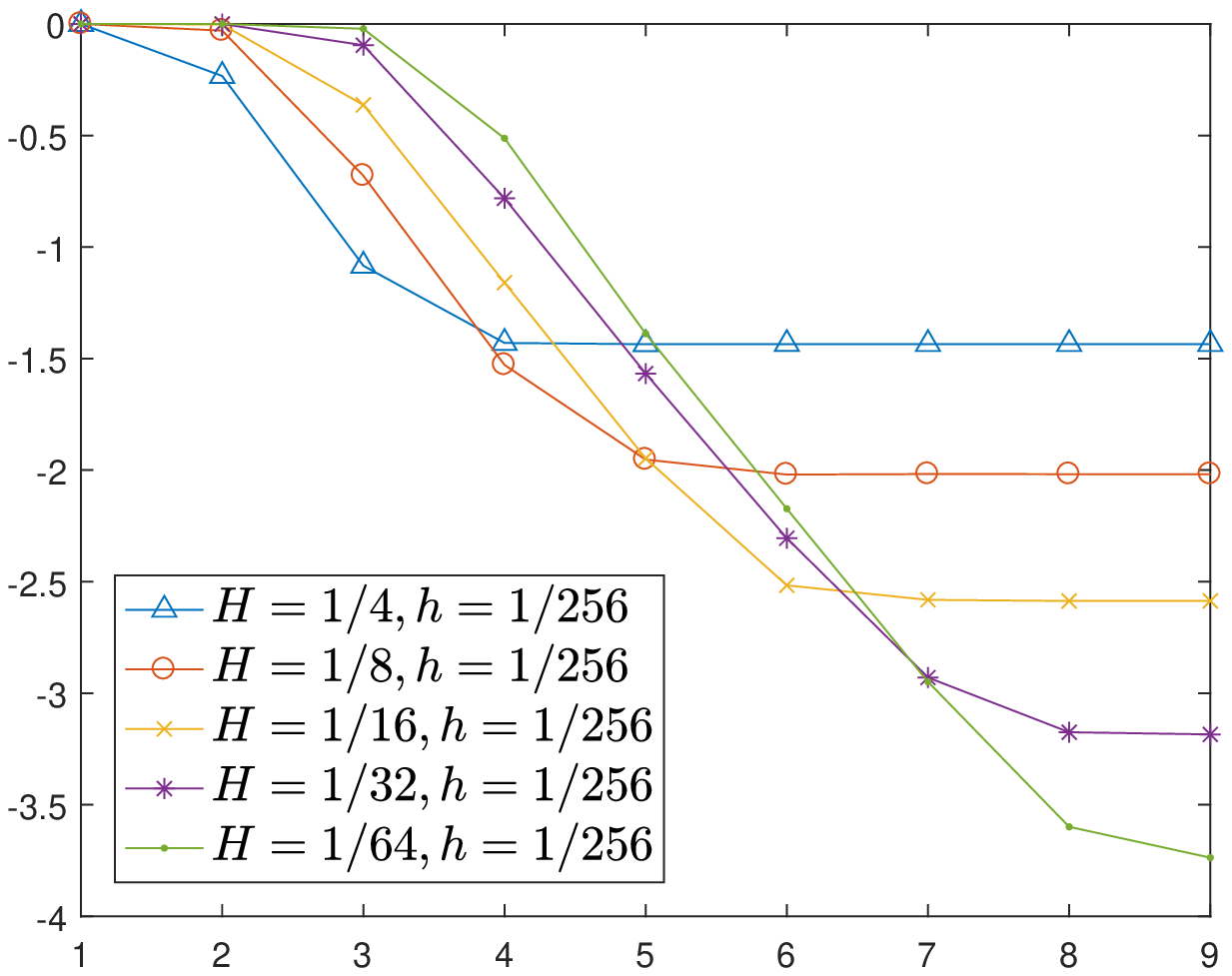}}
\subfigure[$\|u-u_{H}\|_{L^2(\Omega)}$ in log scale]{
\includegraphics[width=0.31\textwidth]{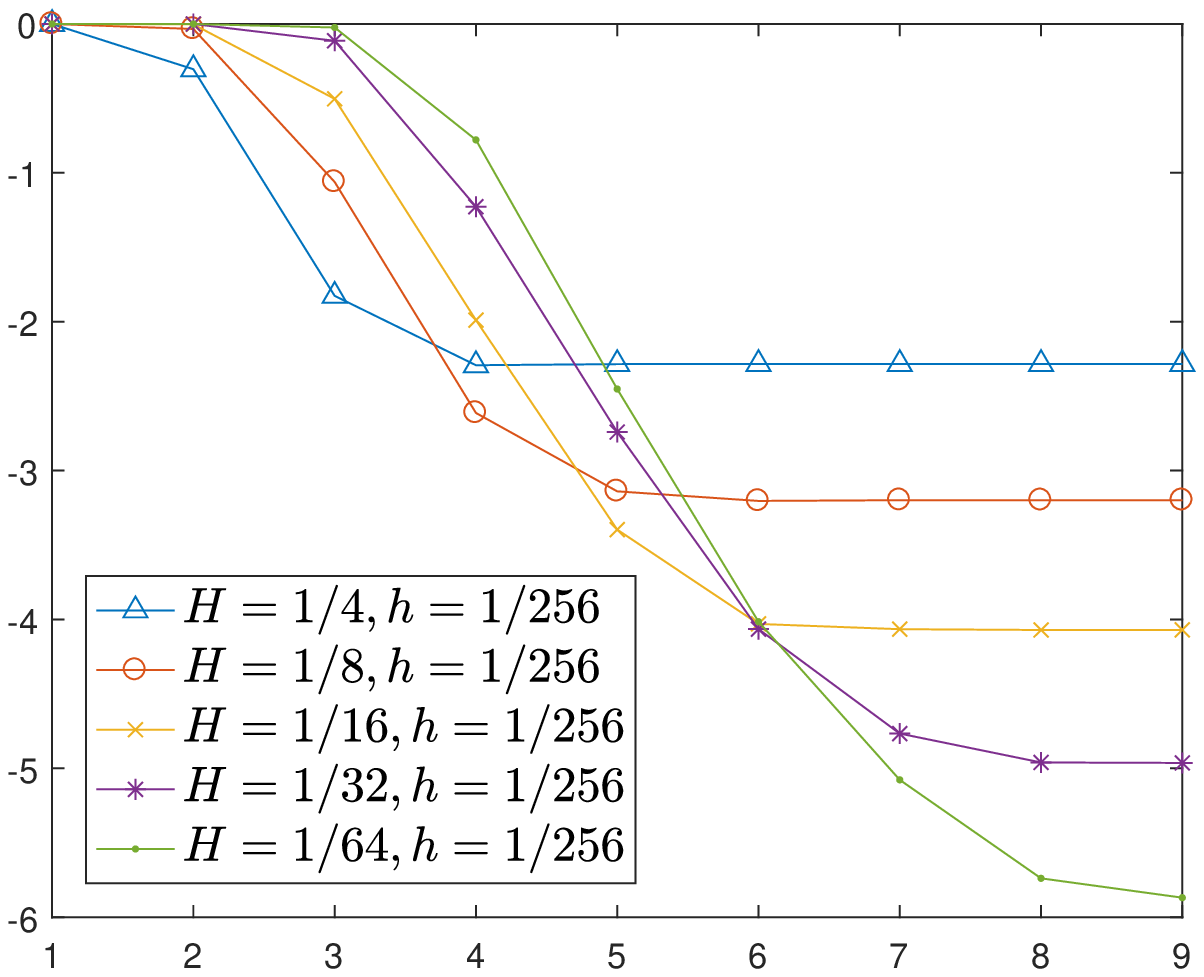}}

\subfigure[$\|y-y_{H}\|_{H^1(\Omega)}$ at different layers]{
\includegraphics[width=0.31\textwidth]{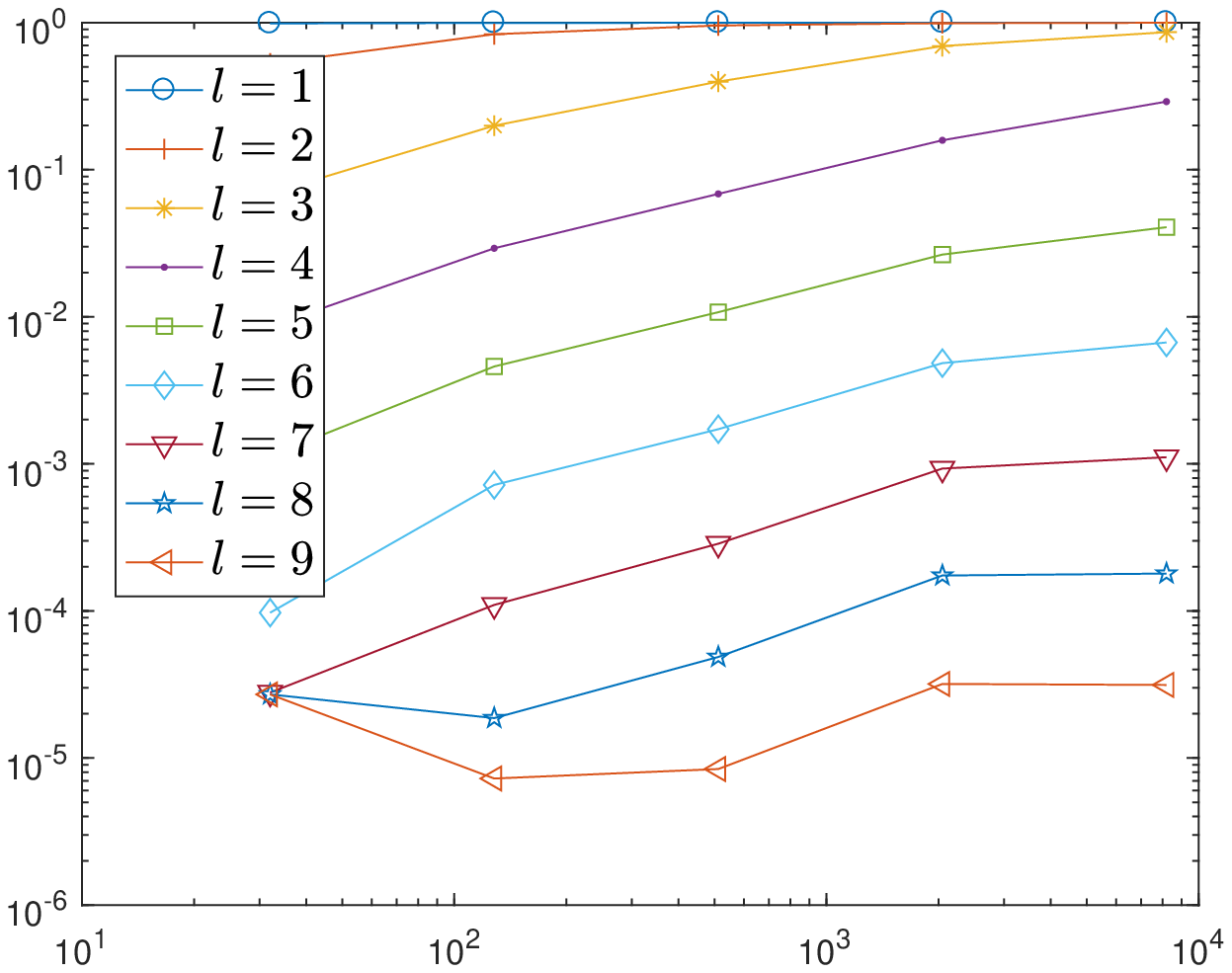}}
\subfigure[$\|p-p_{H}\|_{H^1(\Omega)}$ at different layers]{
\includegraphics[width=0.31\textwidth]{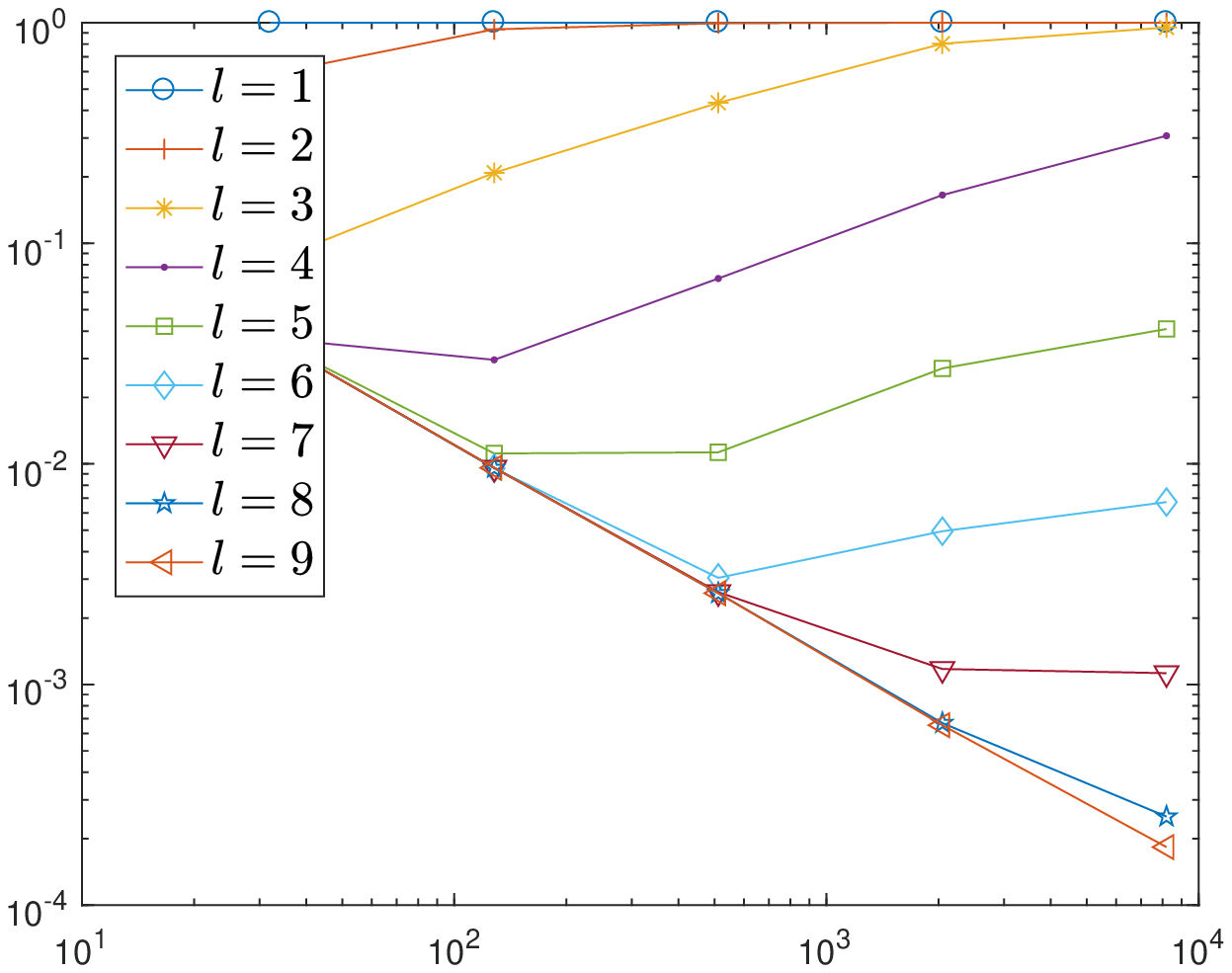}}
\subfigure[$\|u-u_{H}\|_{L^2(\Omega)}$ at different layers]{
\includegraphics[width=0.31\textwidth]{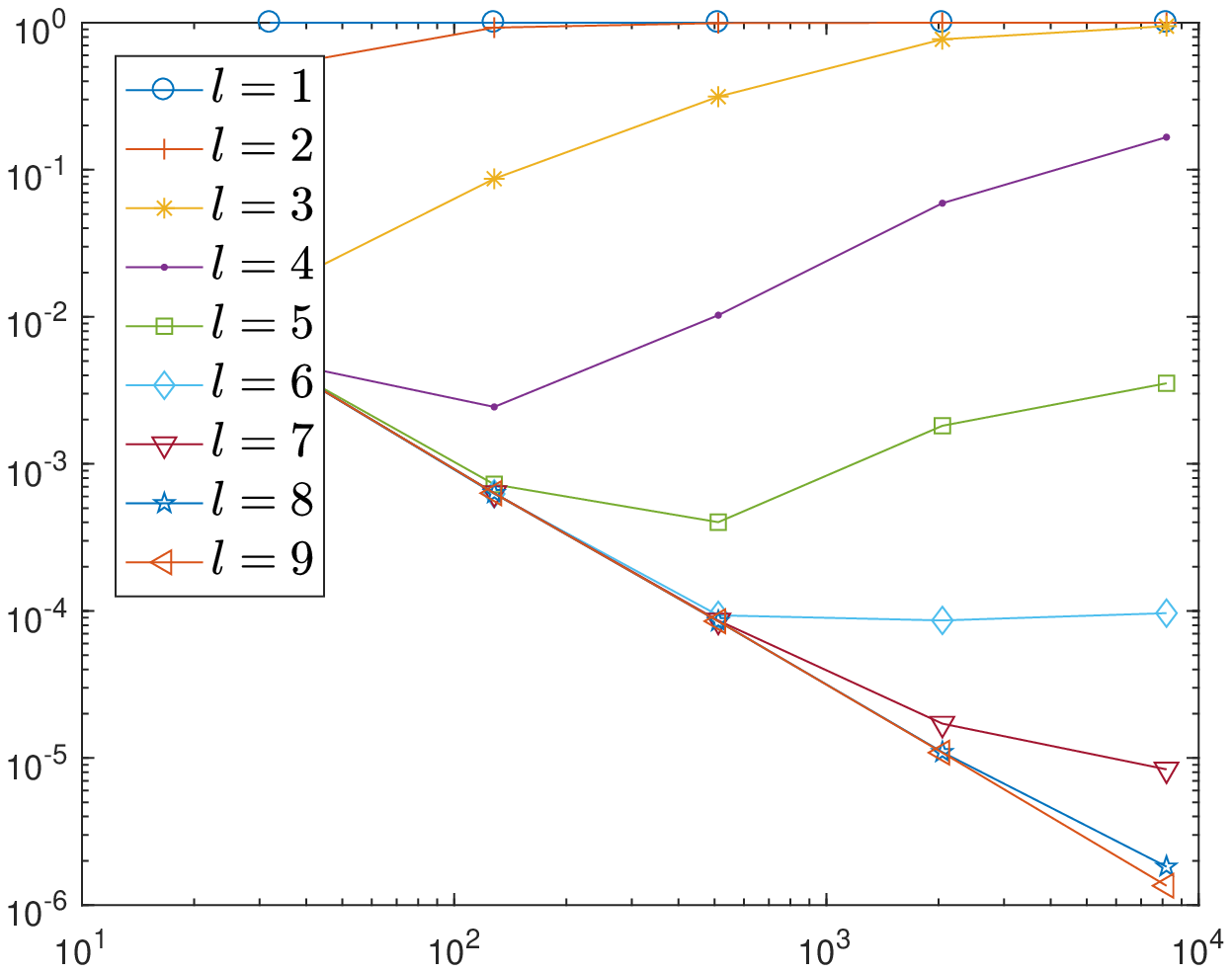}}
\caption{The relative error estimates of each variable using GRPS volume
basis.  In the sub-figure (a) (b) (c), x-axis stands for the layer $l$,
y-axis stands for the coarse mesh error estimates  in $\log_{10}$ scales.
Each curve stands for the relative  error estimations  as the function of
the number of layers with a fixed coarse mesh. In the sub-figure (d) (e)
(f), x-axis stands for the degrees of freedom of the coarse space $V_H$
which is also the number of basis functions (i.e. $2\Nc^2=[32,\ 128,\ 512,\
2048,\ 8192]$), y-axis stands for the coarse mesh error
estimates. Each curve stands for the relative errors
as the function of the coarse $\dof$ with a fixed number of
layers.
  } \centering
\label{fig:coarseerr}
\end{figure}

Figure \ref{fig:combinederr}(a)(b) shows that: with fixed $H$, when the number of layers increase ,
$\|y-y_{H}\|_1+\|p-p_{H}\|_1+\|u-u_{H}\|$  decreases first and then saturates;
as $H$ decrease, and $l \sim 1/H \log (1/H)$, we have the optimal convergence rate $O(H)$.
This is consistent with the Theorem \ref{thm:main}.

\begin{figure}[H]
\centering
\subfigure[$\|y-y_{H}\|_1+\|p-p_{H}\|_1+\|u-u_{H}\|$ with respect to $l$]{
\includegraphics[width=0.3\textwidth]{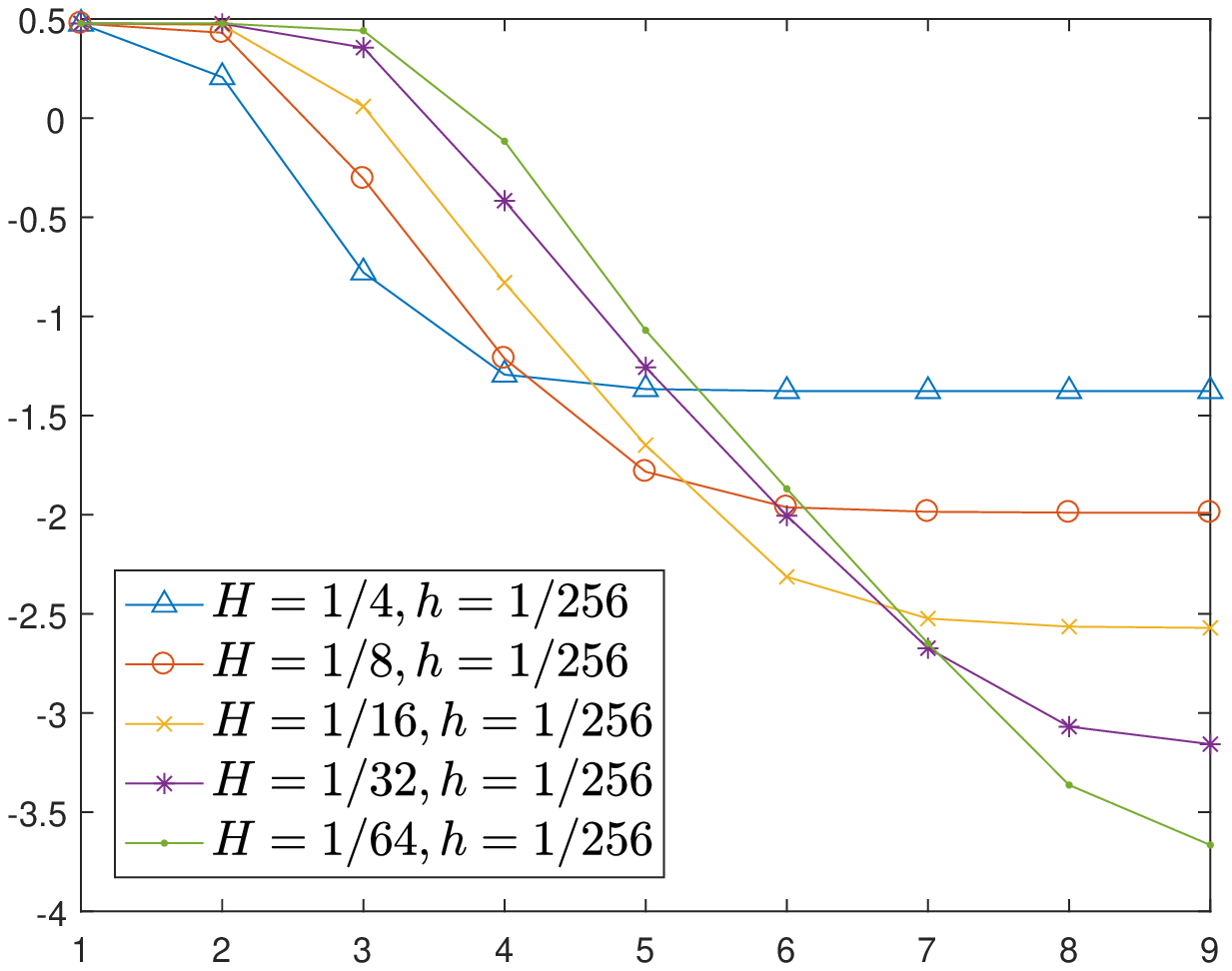}}
\subfigure[$\|y-y_{H}\|_1+\|p-p_{H}\|_1+\|u-u_{H}\|$ with respect to coarse $\dof$]{
\includegraphics[width=0.3\textwidth]{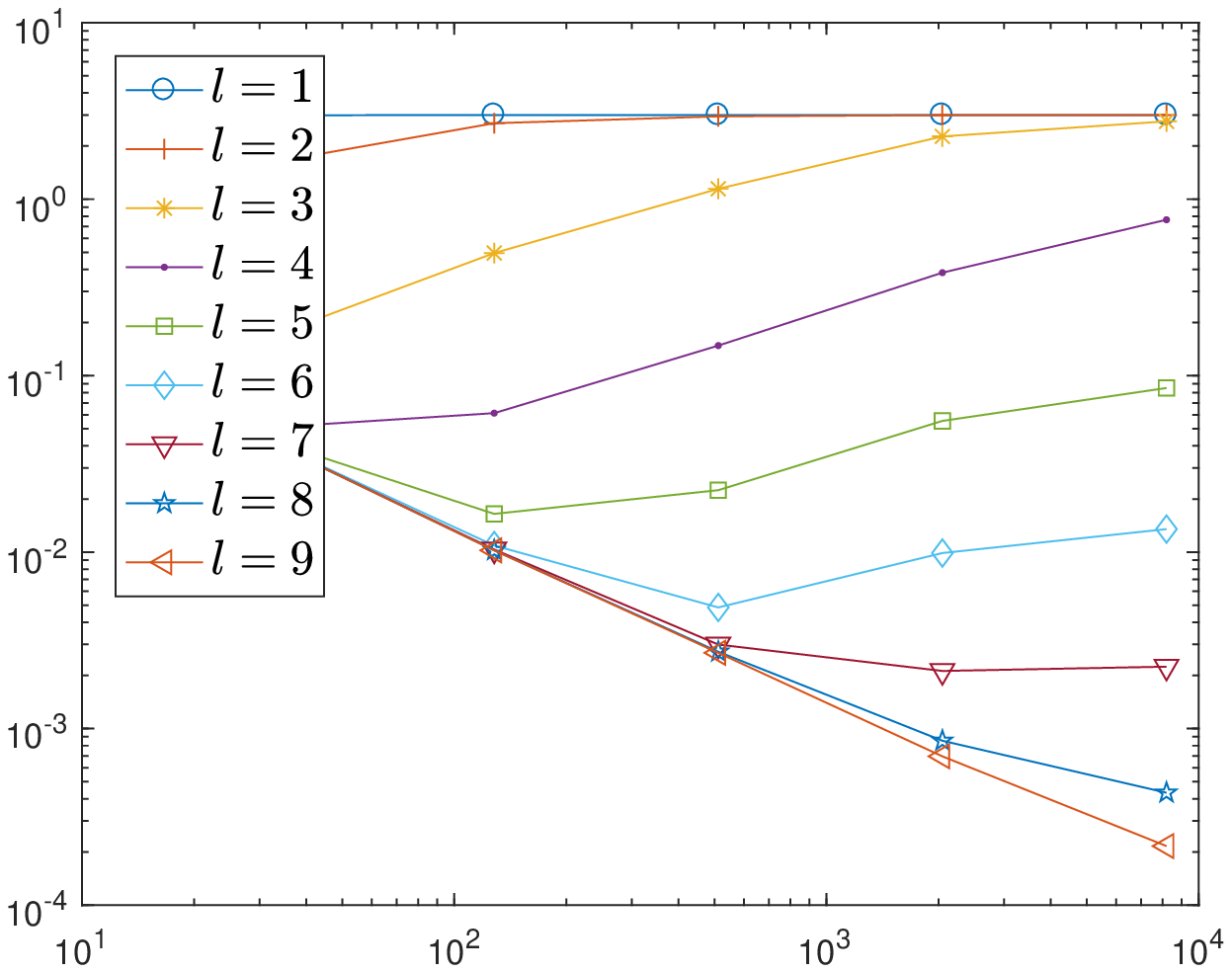}}
\caption{Relative error $\|y-y_{H}\|_1+\|p-p_{H}\|_1+\|u-u_{H}\|$ for GRPS solutions.}\centering
\label{fig:combinederr}
\end{figure}

For comparison with GRPS solution, we show in Figure \ref{RPS1} the relative errors of RPS solutions,  for each variable $y_H$, $p_H$ and $u_H$, and show in Figure \ref{RPS2} the combined error of the RPS solutions of the optimal control problem. It seems the numerical performance of RPS and GRPS are similar, and GRPS is a little more stable.

\begin{figure}[H]
\centering
\subfigure[$\|y-y_{H}\|_{H^1(\Omega)}$ in log scale]{
\includegraphics[width=0.3\textwidth]{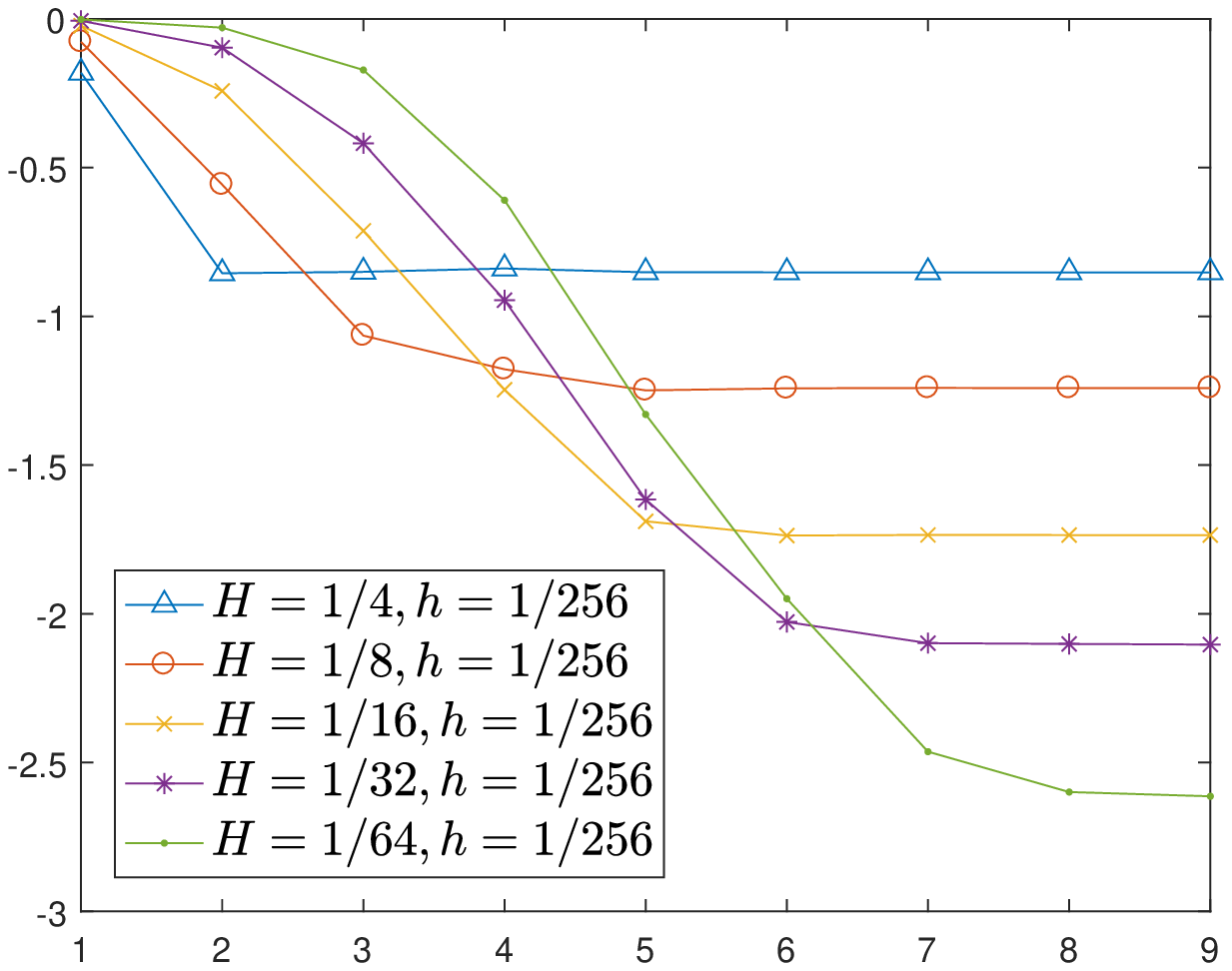}}
\subfigure[$\|p-p_{H}\|_{H^1(\Omega)}$ in log scale]{
\includegraphics[width=0.3\textwidth]{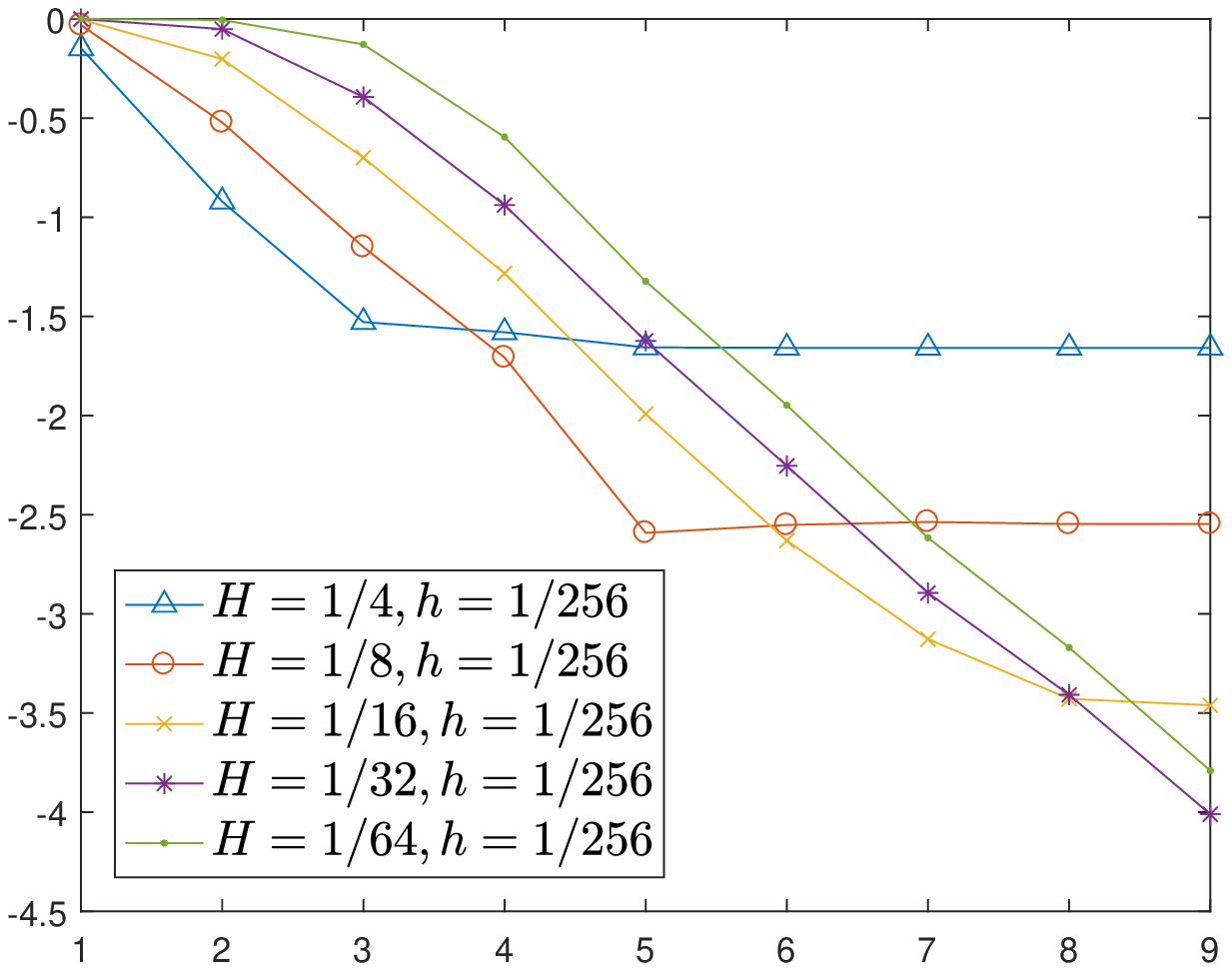}}
\subfigure[$\|u-u_{H}\|_{L^2(\Omega)}$ in log scale]{
\includegraphics[width=0.3\textwidth]{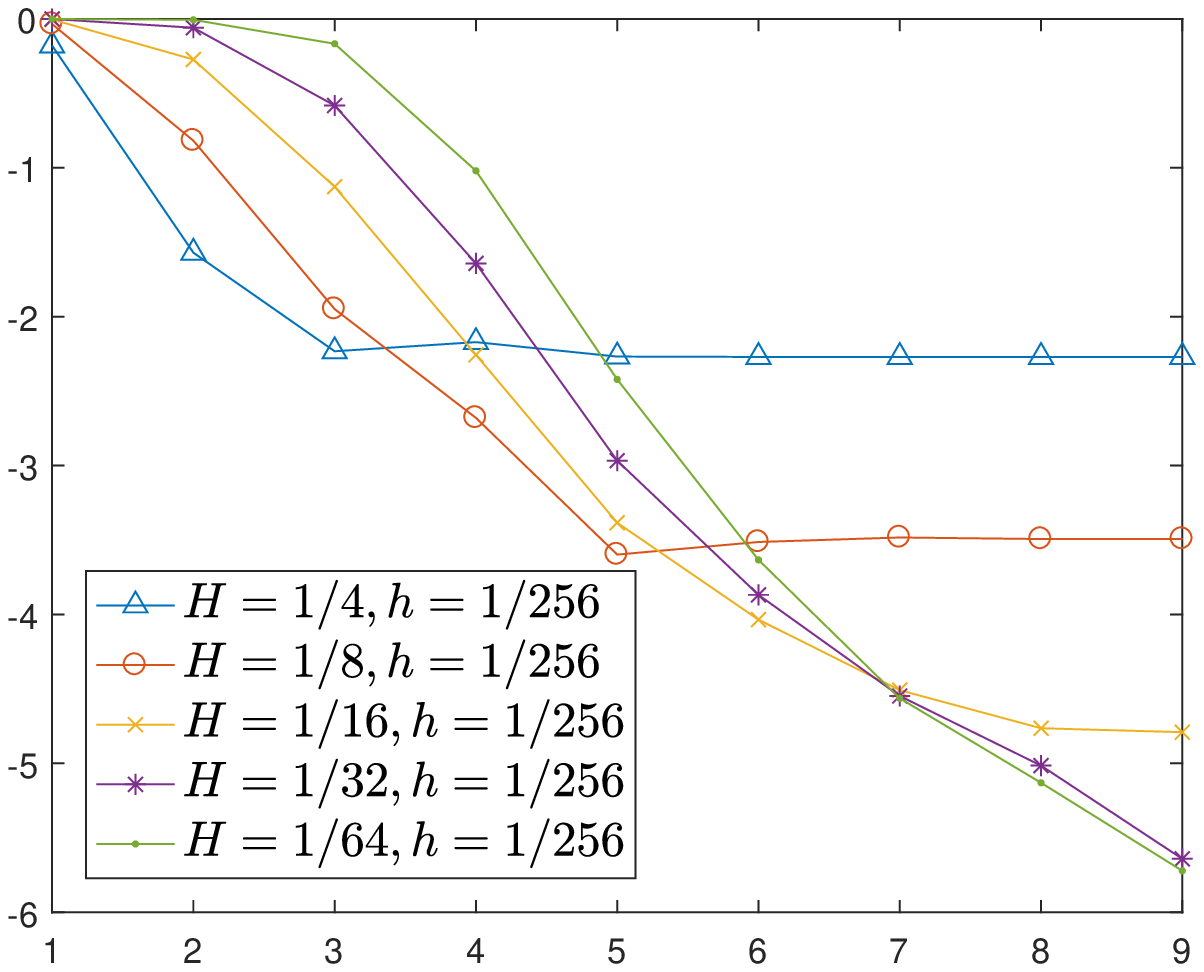}}

\subfigure[$\|y-y_{H}\|_{H^1(\Omega)}$ at different layer]{
\includegraphics[width=0.3\textwidth]{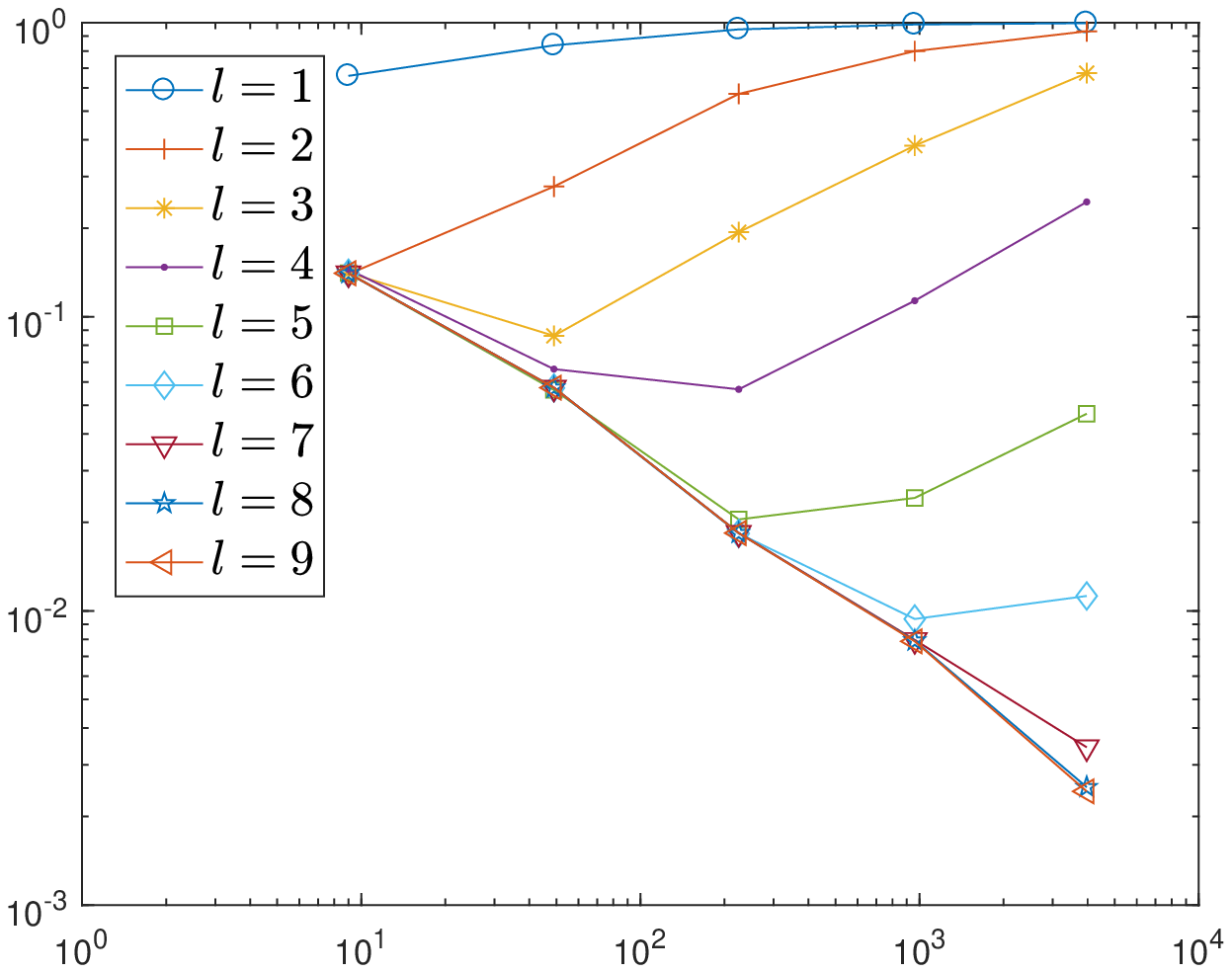}}
\subfigure[$\|p-p_{H}\|_{H^1(\Omega)}$ at different layer]{
\includegraphics[width=0.3\textwidth]{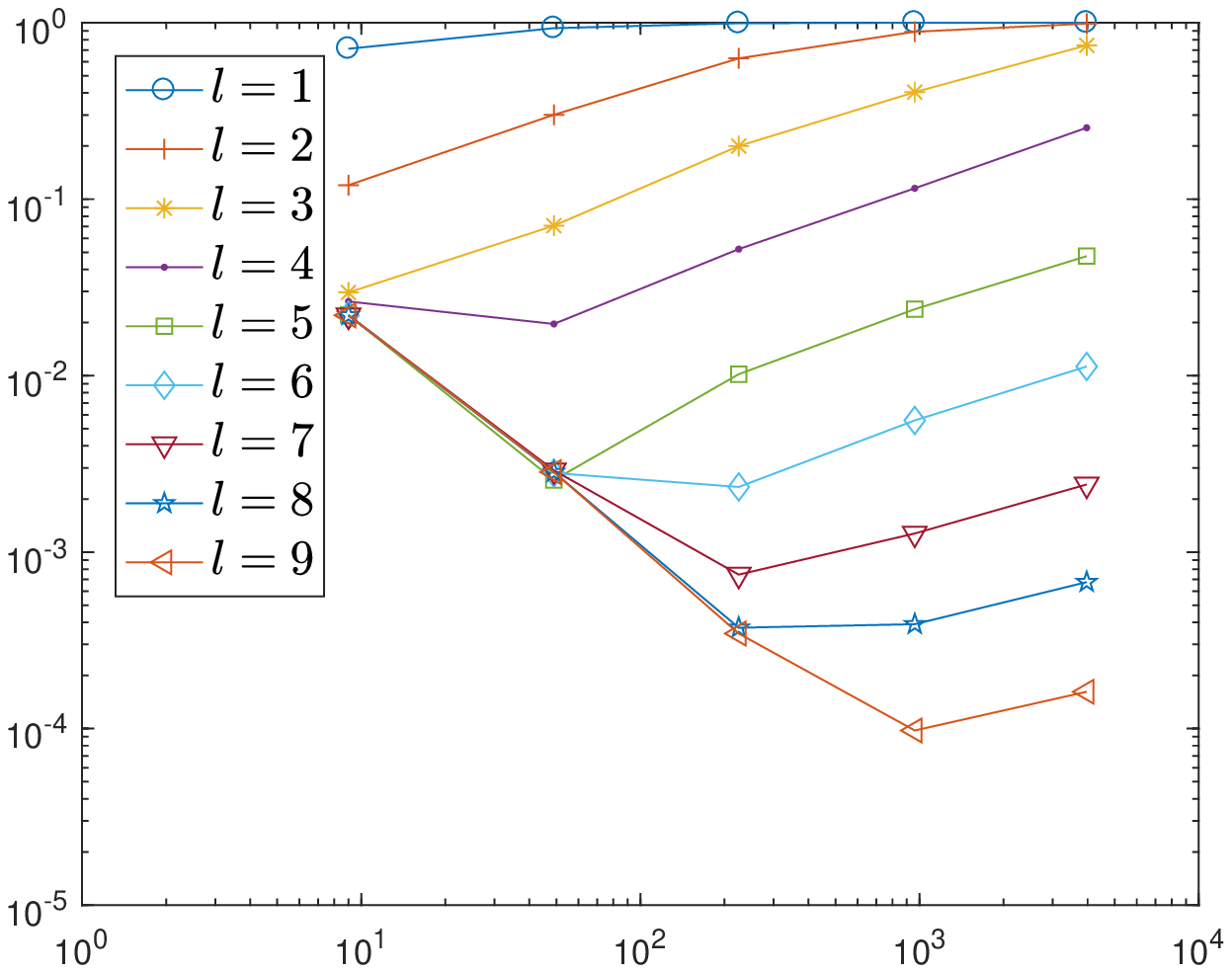}}
\subfigure[$\|u-u_{H}\|_{L^2(\Omega)}$ at different layer]{
\includegraphics[width=0.3\textwidth]{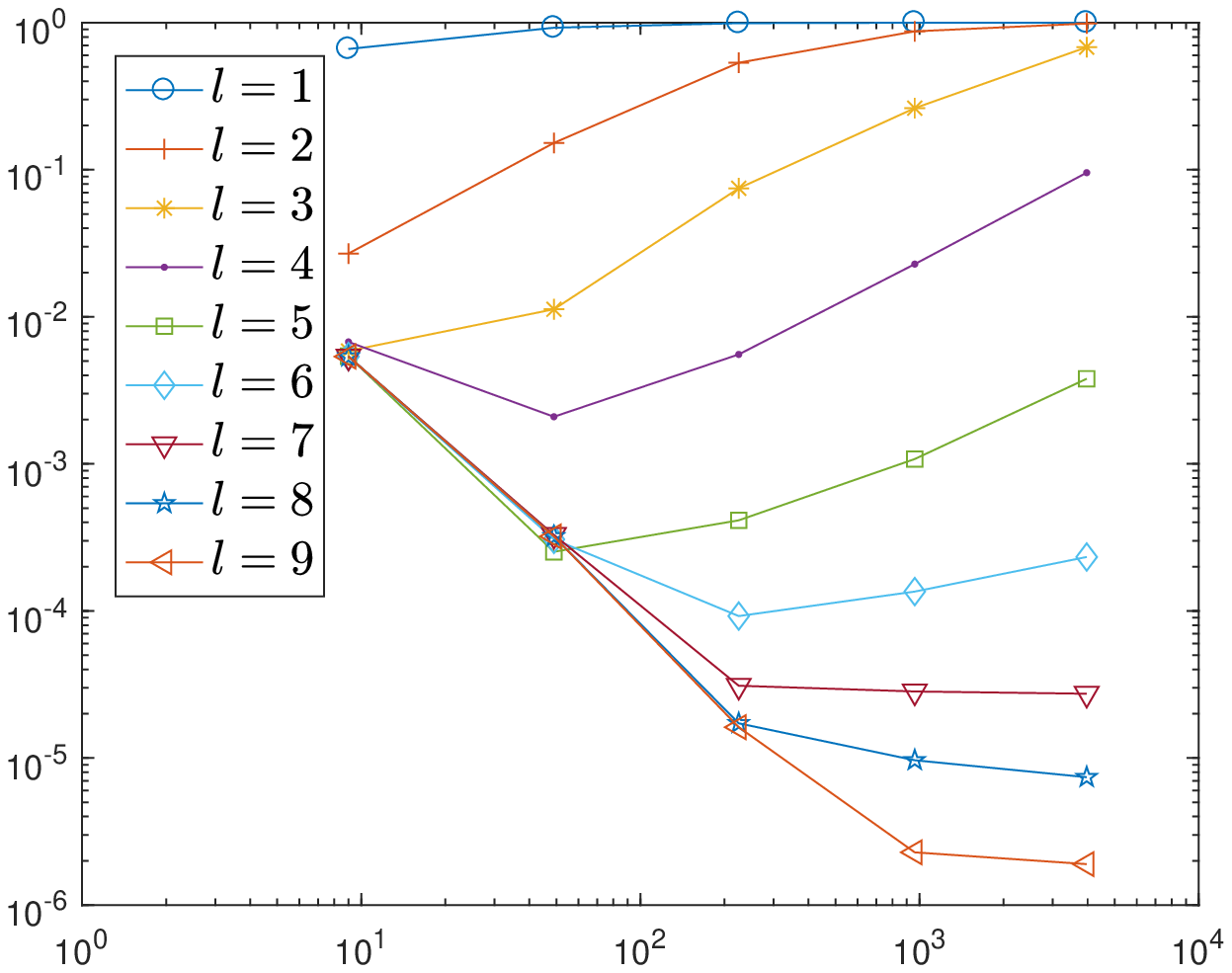}}
\caption{The relative error estimates of each variable using RPS
basis. In the sub-figure (a)(b)(c), x-axis stands for the layer $l$, y-axis
stands for  the coarse mesh error estimates in $\log_{10}$ scales. In the
sub-figure (d)(e)(f), x-axis  stands for the coarse $\dof$ (i.e.
$(N_c-1)^2=[ 9,\ 49,\ 225,\ 961,\ 3969]$), y-axis stands
for the coarse mesh error estimates. }\label{RPS1}
\centering
\end{figure}

\begin{figure}[H]
\centering
\subfigure[$\|y-y_{H}\|_1+\|p-p_{H}\|_1+\|u-u_{H}\|$  with respect to $l$]{
\includegraphics[width=0.3\textwidth]{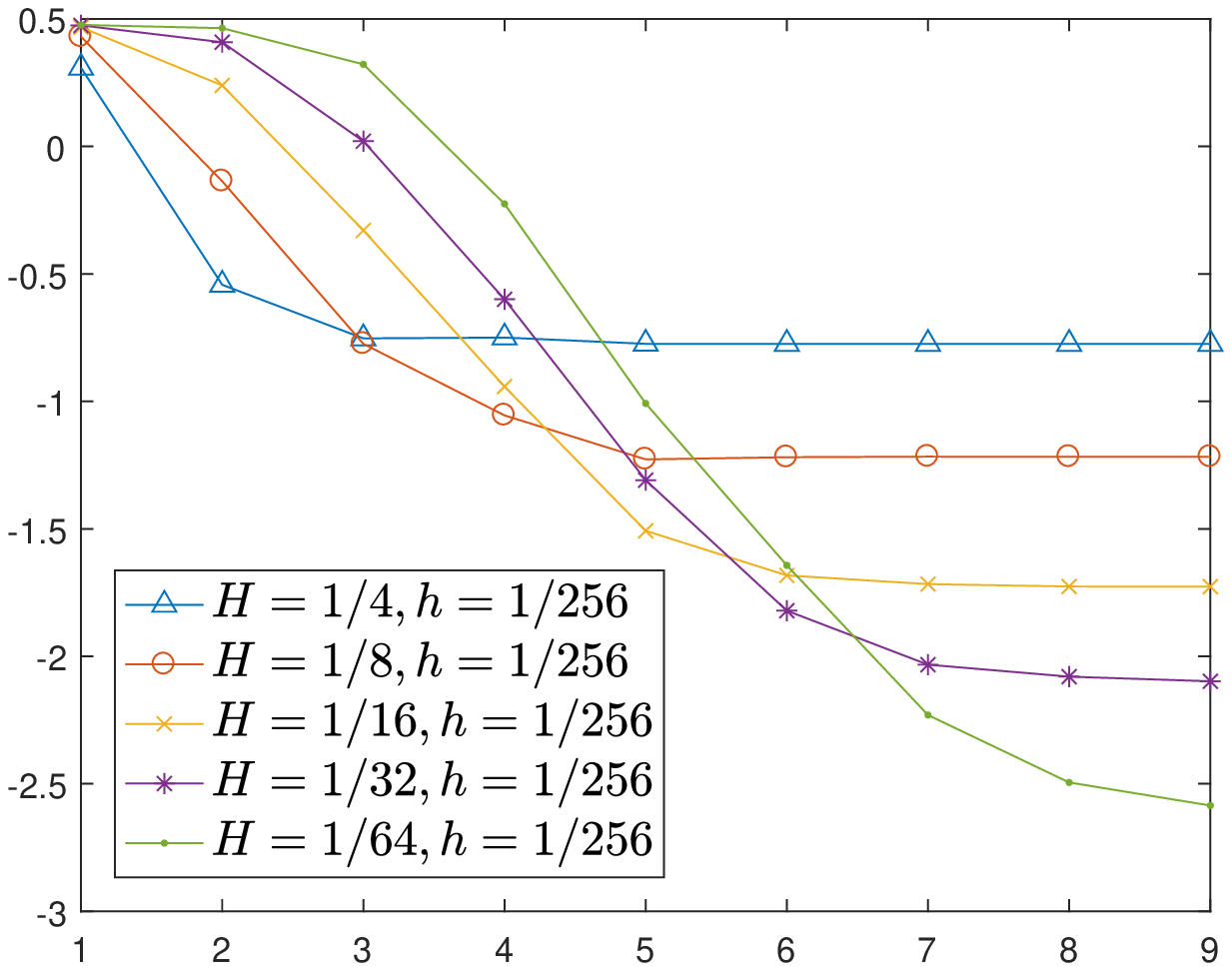}}
\subfigure[$\|y-y_{H}\|_1+\|p-p_{H}\|_1+\|u-u_{H}\|$ with respect to coarse $\dof$]{
\includegraphics[width=0.3\textwidth]{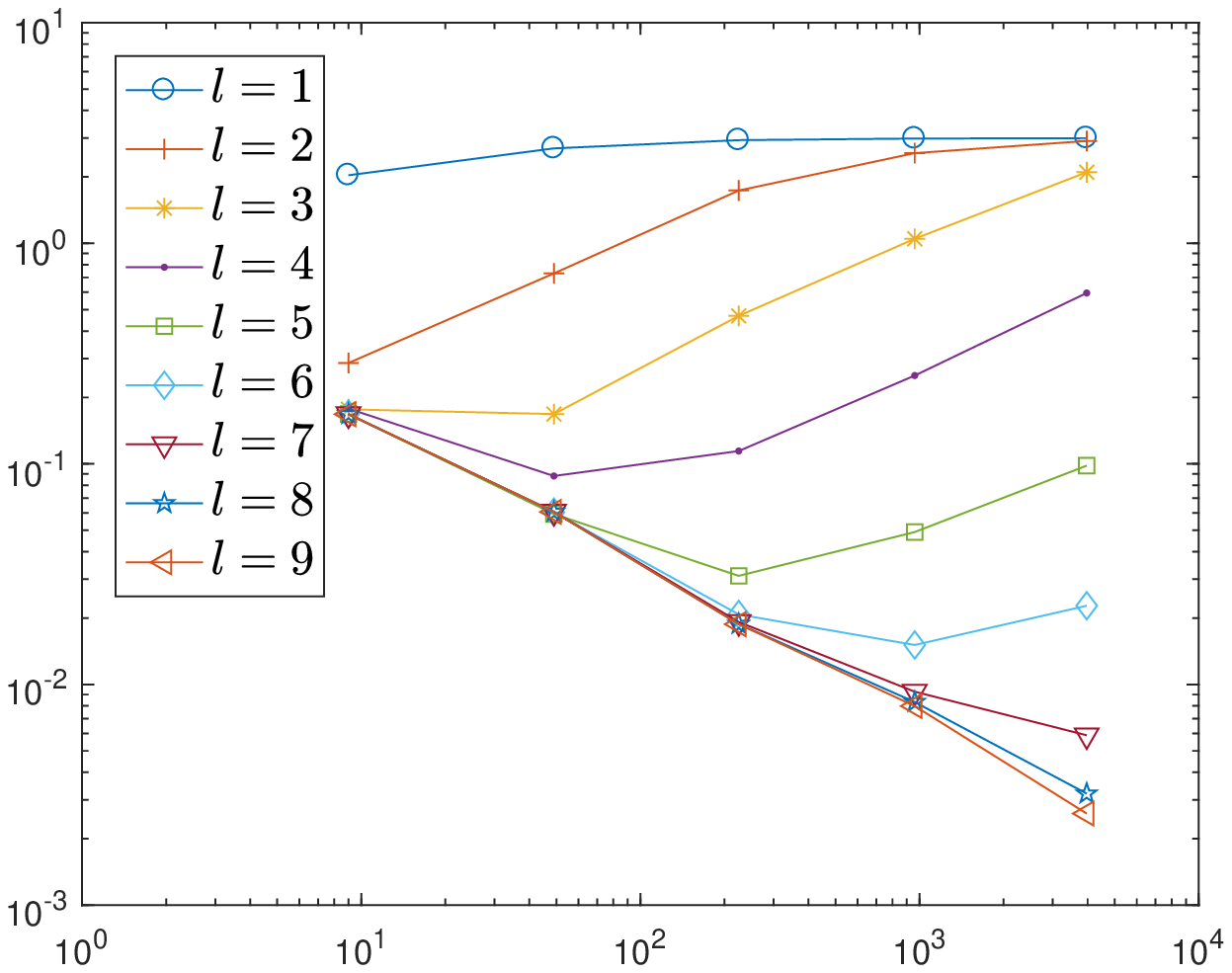}}
\caption{Relative error $\|y-y_{H}\|_1+\|p-p_{H}\|_1+\|u-u_{H}\|$ for RPS solutions.}\label{RPS2}
\centering
\end{figure}

\subsubsection{SPE 10}
\label{sec:SPE10}

The purpose of this section is to show the performance of GRPS method for a more practical example, namely, the SPE10 benchmark problem (\url{http://www.spe.org/web/csp/}).
SPE10 is the latest industry benchmark problems (SPE10) from the Society of Petroleum Engineers (SPE). The physical domain is a cube with dimension $[0,220]\times [0,60]\times [0,85]$. The
coefficients $a(x,y,z)$ are given as piecewise constant numerical values
rather than a continuous function. We select the layer 39 with
respect to the height dimension and treat them as coefficients of two dimensional problems. The contrast of the coefficients is $1.74956\times 10^{7}$. The domain is chosen as $\Omega=[0,2.2]\times [0,0.6]$. We first uniformly divide
$\Omega$ with the coarse mesh size $H$, then we further refine the mesh with the
fine mesh size $h$. In the numerical experiment, we choose $H=1/25,
1/50,1/100,1/200$, respectively. The fine mesh size is fixed as $h=1/800$. We illustrate the SPE10 coefficients in Figure \ref{fig:spe10coef}.

\vspace{-10pt}

\begin{figure}[H]
\centering
\subfigure[Layer 39  contours in $\log_{10}$ scale.]{\includegraphics[height = 3.5cm,width=0.45\textwidth]{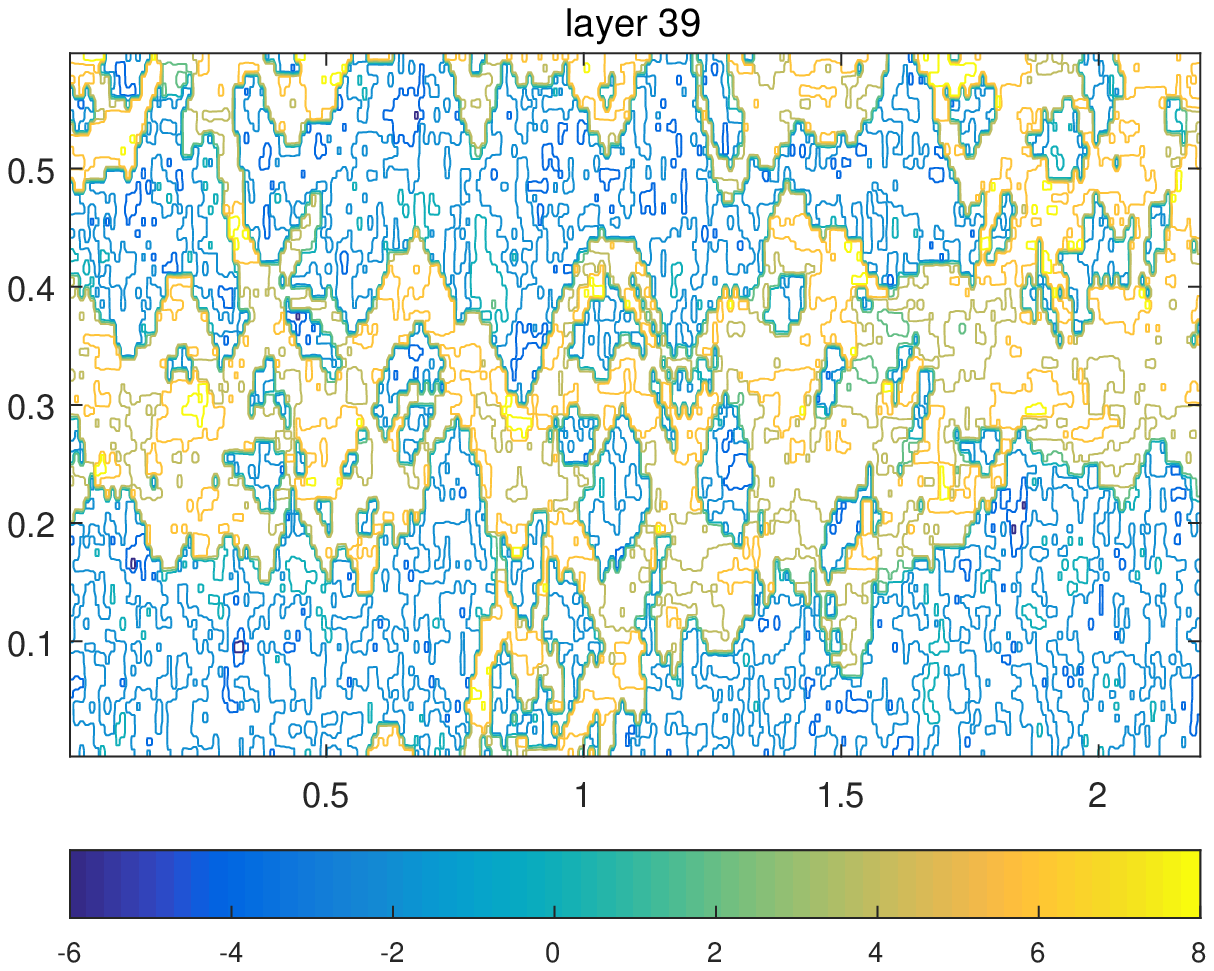}}
\subfigure[Layer 39.]{\includegraphics[width=0.4\textwidth]{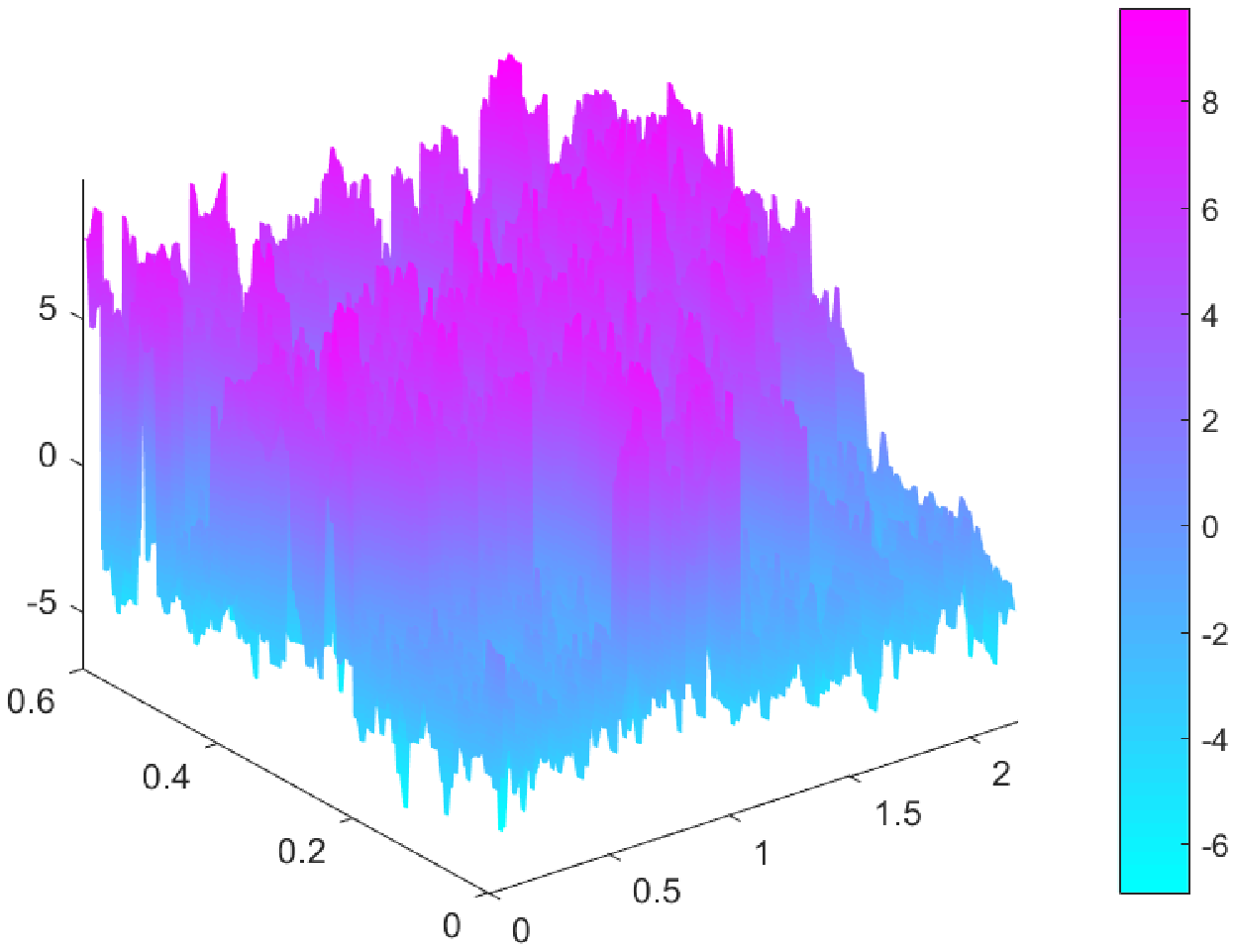}}
\caption{SPE10 coefficients}\label{fig:spe10coef}
\centering
\end{figure}

The errors of GRPS solutions are shown in Figures \ref{coarseerr1} and  \ref{coarseerr2}.

\begin{figure}[H]
\centering
\subfigure[$\|y-y_{H}\|_{1}$ SPE10 layer 39 ]{
\includegraphics[width=0.23\textwidth]{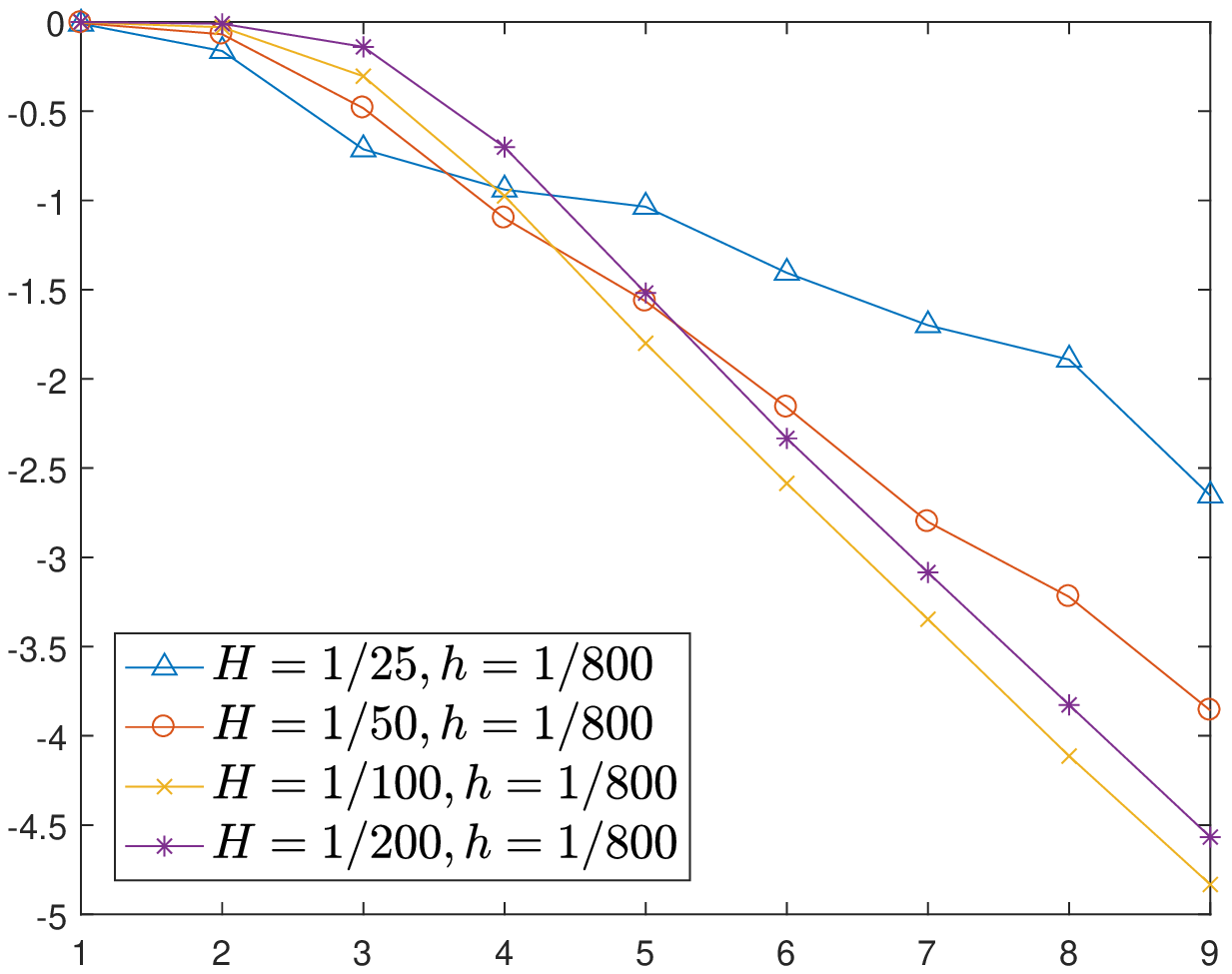}}
\subfigure[$\|p-p_{H}\|_{1}$ SPE10 layer 39 ]{
\includegraphics[width=0.23\textwidth]{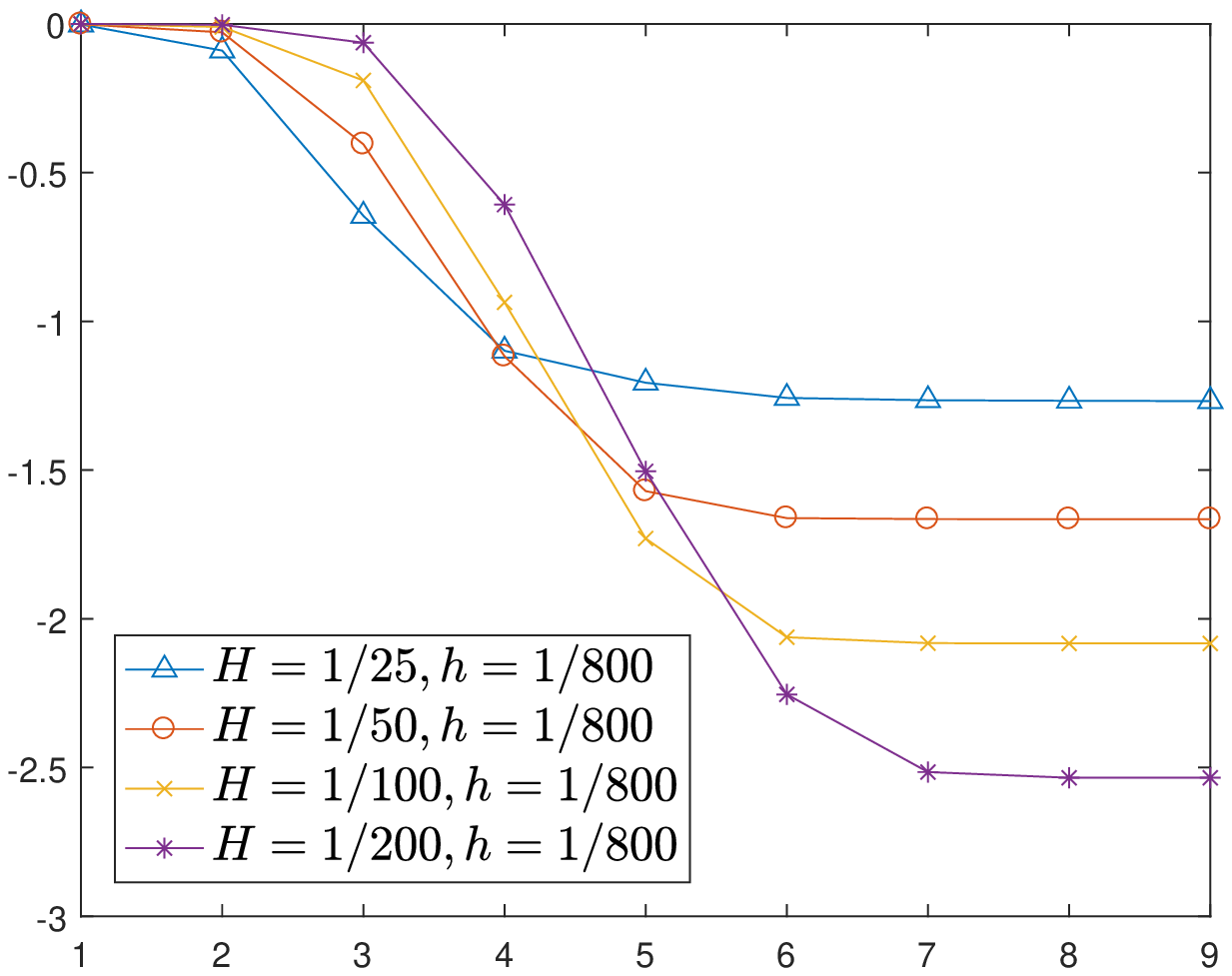}}
\subfigure[$\|u-u_{H}\|$ SPE10 layer 39]{
\includegraphics[width=0.23\textwidth]{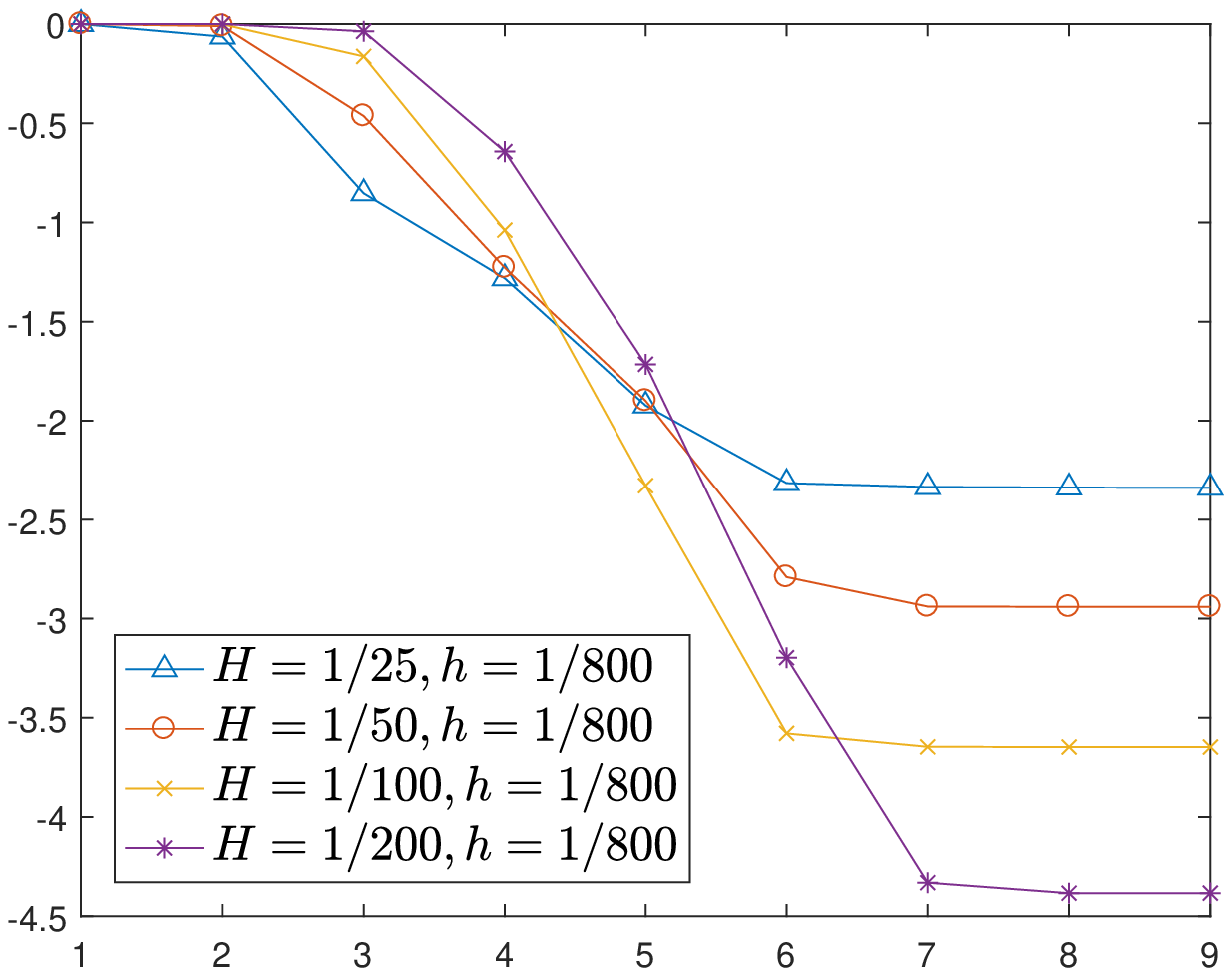}}
\subfigure[$\|y-y_{H}\|_{1}+\|p-p_{H}\|_{1}+\|u-u_{H}\|$ SPE10 layer 39]{
\includegraphics[width=0.23\textwidth]{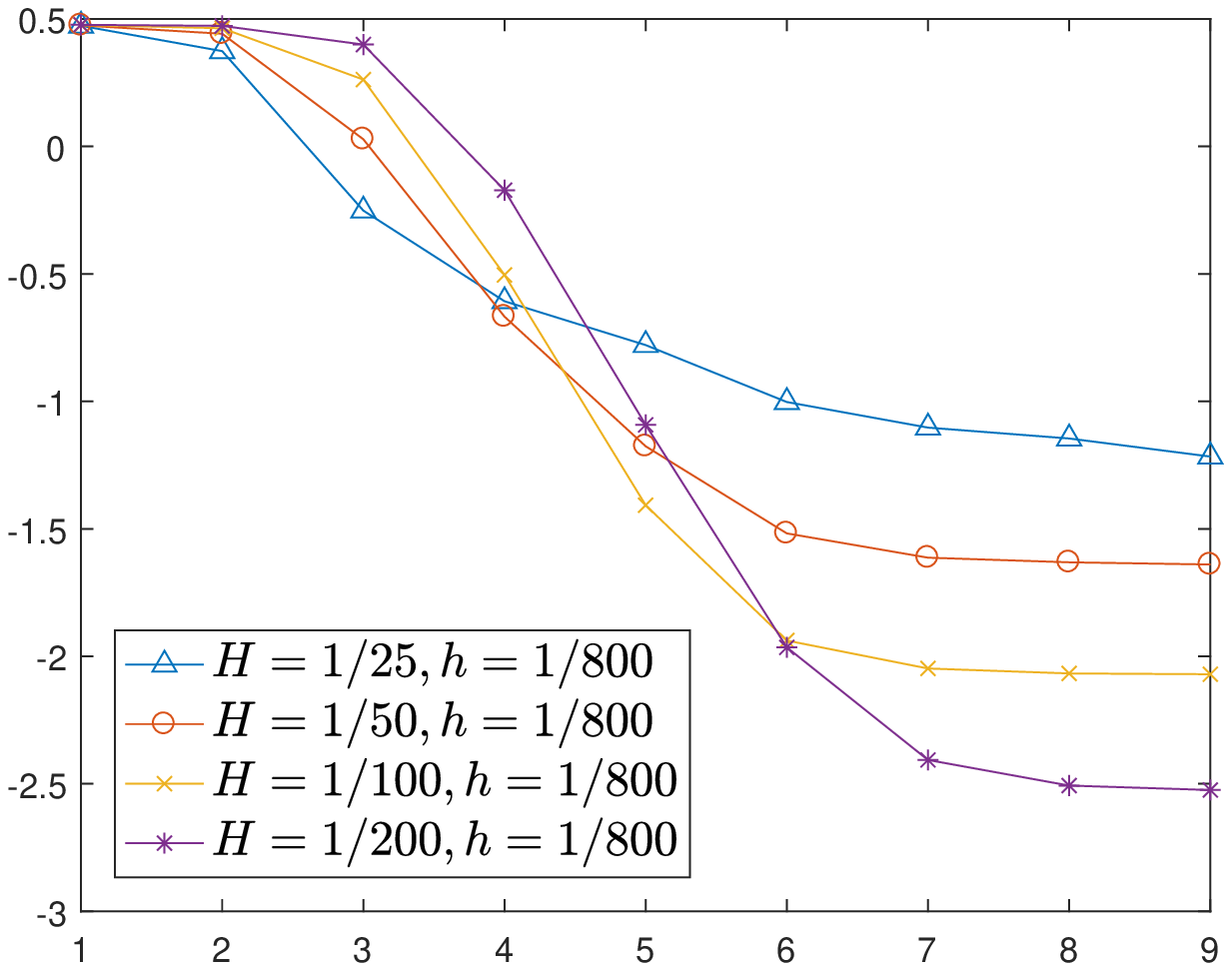}}
\caption{The relative errors of GPRS solutions $y_H$, $p_H$ and $u_H$ with respect to $l$.}\label{coarseerr1} \centering
\end{figure}

\begin{figure}[H]
\centering
\subfigure[$\|y-y_{H}\|_{1}$ SPE10 layer 39 ]{
\includegraphics[width=0.23\textwidth]{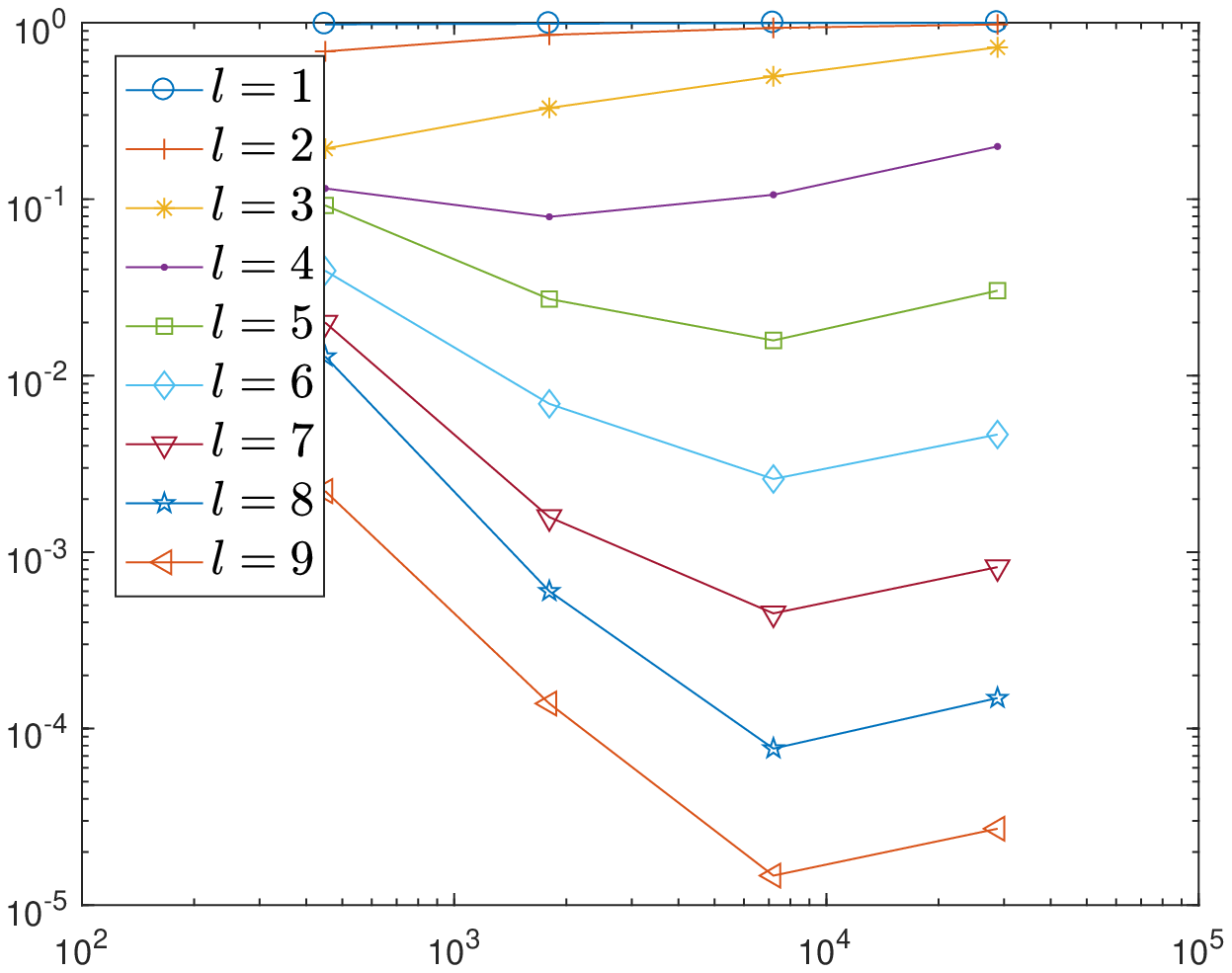}}
\subfigure[$\|p-p_{H}\|_{1}$ SPE10 layer 39 ]{
\includegraphics[width=0.23\textwidth]{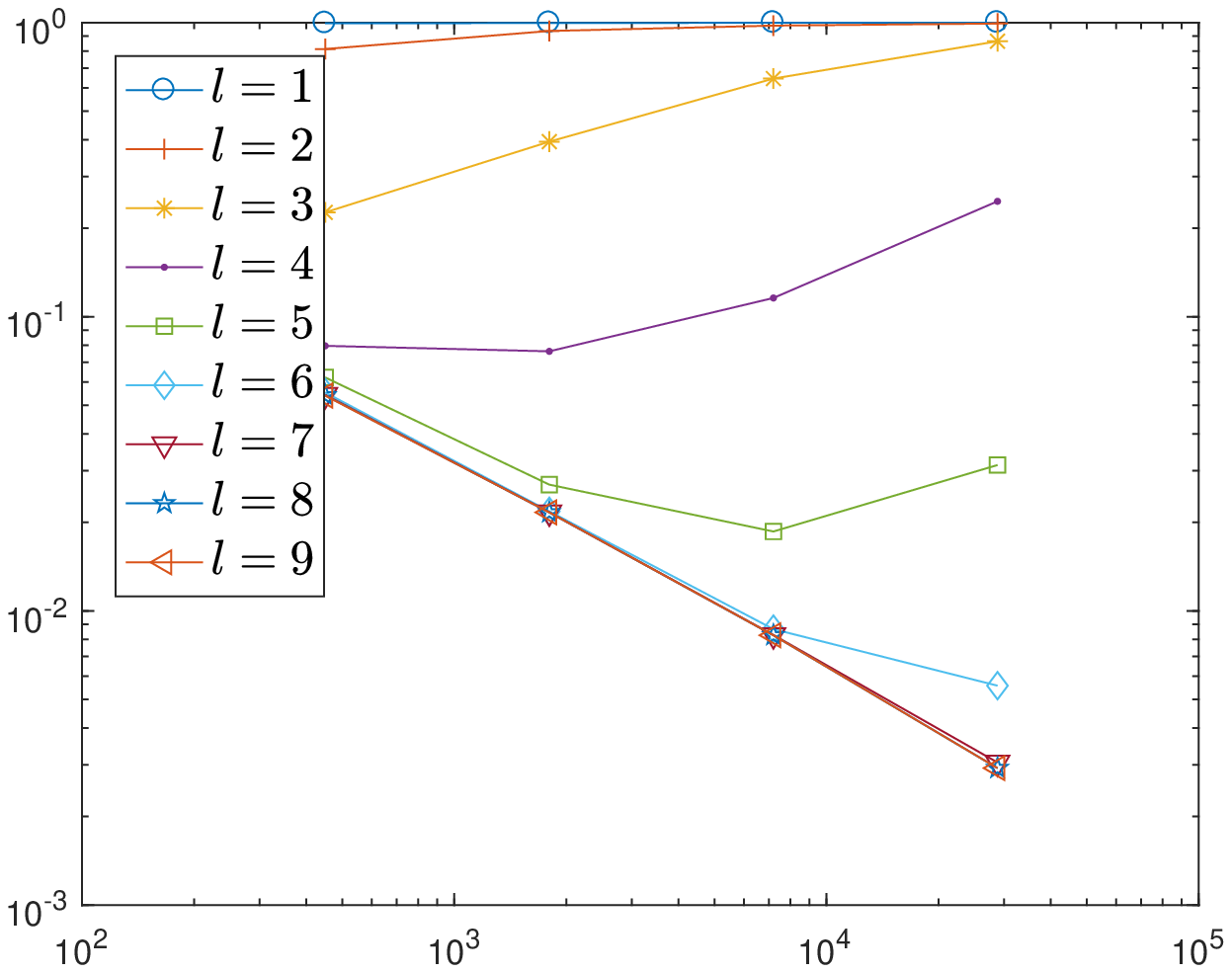}}
\subfigure[$\|u-u_{H}\|$ SPE10 layer 39]{
\includegraphics[width=0.23\textwidth]{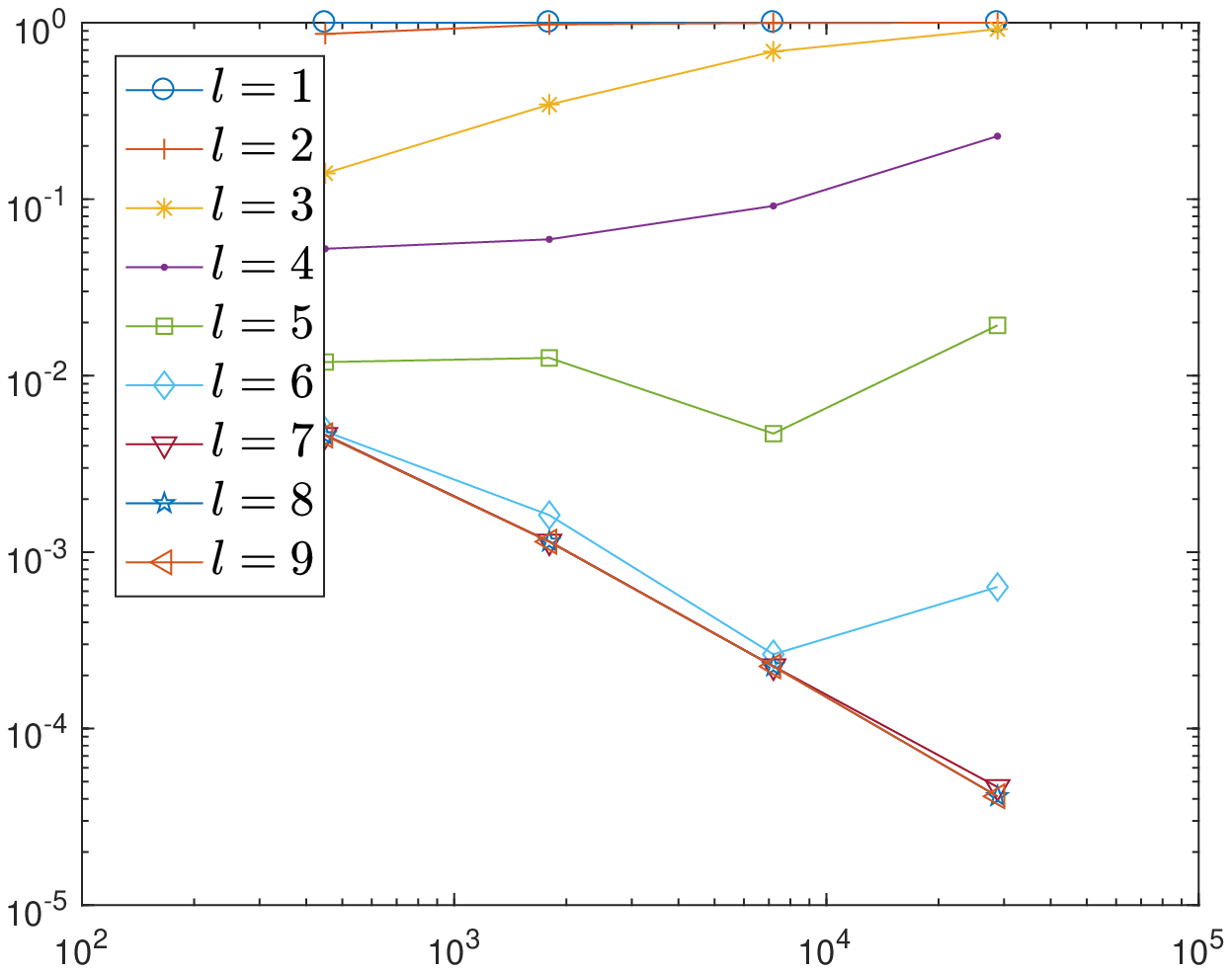}}
\subfigure[$\|y-y_{H}\|_{1}+\|p-p_{H}\|_{1}+\|u-u_{H}\|$ SPE10 layer 39]{
\includegraphics[width=0.23\textwidth]{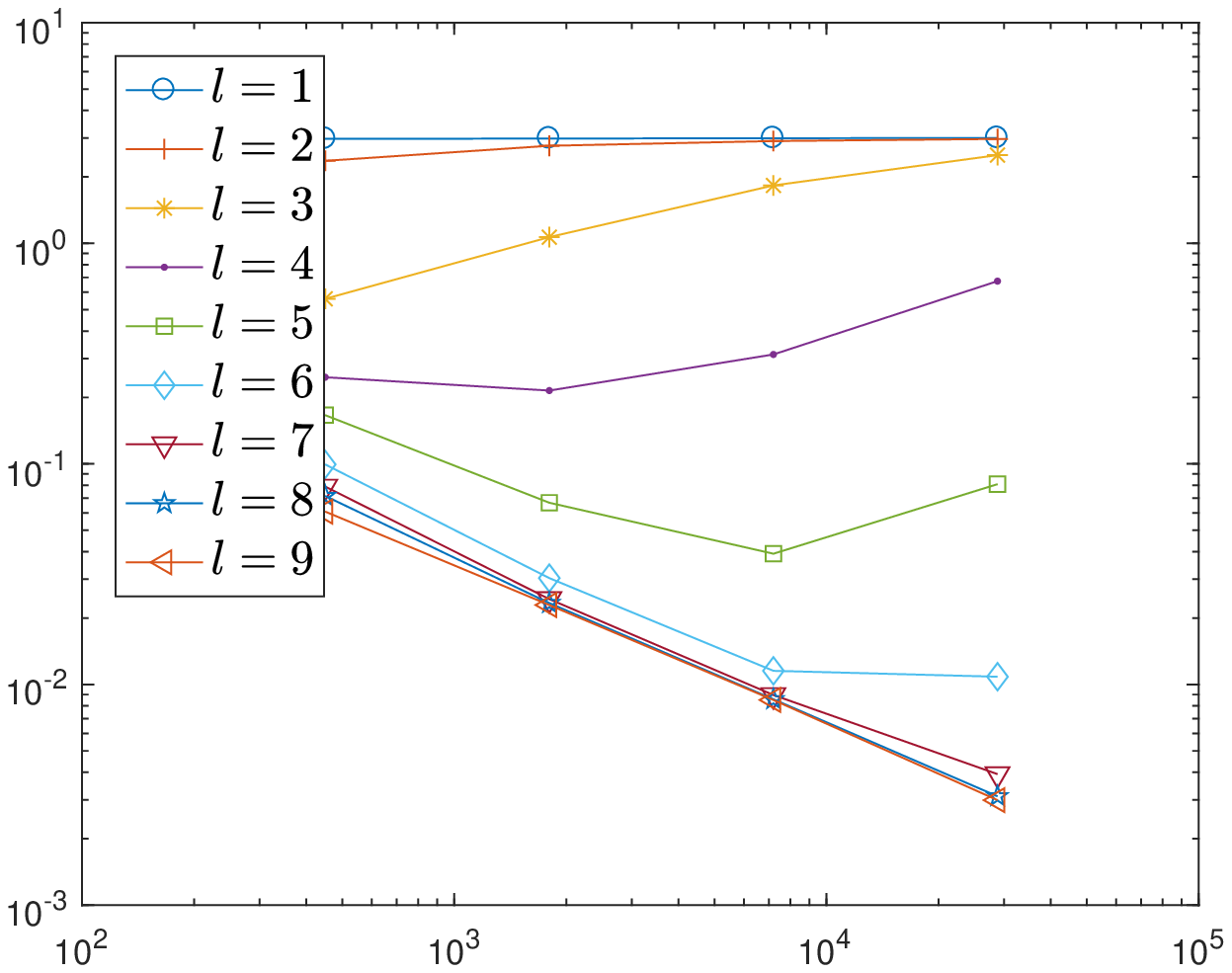}}
\caption{The relative errors of GRPS solutions $y_H$, $p_H$ and $u_H$ with respect to coarse $\dof$.}\label{coarseerr2} \centering
\end{figure}

\section{Conclusions}
\label{sec:conclusion}

In this paper, we have introduced the generalized rough polyharmonic splines (GRPS, including RPS) method for the efficient solution of optimal control problem govern by multiscale elliptic equation with rough coefficients. We have derived rigorous error estimates and the numerical experiments complement well with the theoretical analysis.

We plan to apply this strategy to optimal control problem with more general control conditions, such as boundary control, as well as optimal control problem for elasticity and Stokes equations with rough coefficients. 

The numerical homogenization based method in this paper can provide coarse scale accuracy of the optimal control solution. If fine scale accuracy is desired, numerical homogenization may help design efficient preconditioners, and furthermore, efficient mutlgrid/multilevel method such as Gamblet based multigrid method \cite{OwhadiMultigrid:2017, Zhang:2018}, can be used for the efficient resolution of the optimal control governed by multiscale problems.

\bibliographystyle{siamplain}
\bibliography{optimal}
\end{document}